\documentclass{article}
\PassOptionsToPackage{svgnames}{xcolor}
\usepackage{hyperref, url}               
\hypersetup{colorlinks=true,linkcolor=DarkGreen,urlcolor=black,citecolor=Purple}                                 
\usepackage[english]{babel}         
\usepackage[utf8]{inputenc}         

\usepackage{csquotes}               
\usepackage{amssymb, amsmath, amsthm}
\usepackage{dsfont}                 
\usepackage{tikz-cd, tikz, mathtools}
\usepackage{enumerate}			    
\usepackage[capitalise]{cleveref}   
\usepackage{geometry}
\newcommand{\iso}{\xrightarrow{\;\smash{\raisebox{-0.5ex}{\ensuremath{\scriptstyle\sim}}}\;}}

\DeclareMathOperator{\Hom}{Hom}
\DeclareMathOperator{\iHom}{\underline{\operatorname{Hom}}}
\DeclareMathOperator{\Map}{Map}
\DeclareMathOperator{\Fun}{Fun}
\newcommand{\FunL}{\Fun^{\rm L}}
\newcommand{\FunR}{\Fun^{\rm R}}
\DeclareMathOperator{\Nat}{Nat}

\DeclareMathOperator{\fgt}{fgt}
\DeclareMathOperator{\id}{id}
\DeclareMathOperator{\pr}{pr}

\DeclareMathOperator{\cofib}{cofib}

\DeclareMathOperator{\colim}{colim}
\let\lim\relax
\DeclareMathOperator{\lim}{lim}
\newcommand{\ev}{{\rm ev}}
\newcommand{\coev}{{\rm coev}}
\DeclareMathOperator{\res}{res}
\DeclareMathOperator{\ind}{ind}
\DeclareMathOperator{\coind}{coind}
\DeclareMathOperator{\Nm}{Nm}
\DeclareMathOperator{\Nmadj}{\widetilde{\Nm}}

\DeclareMathOperator{\End}{End}
\DeclareMathOperator{\CMon}{CMon}
\DeclareMathOperator{\CAlg}{CAlg}
\DeclareMathOperator{\Alg}{Alg}
\DeclareMathOperator{\LConst}{LConst}

\DeclareMathOperator{\PSh}{PSh}
\DeclareMathOperator{\Ho}{Ho}
\DeclareMathOperator{\Ar}{Ar}

\newcommand{\prop}{\mathrm{pr}}
\newcommand{\inv}{\mathrm{inv}}
\DeclareMathOperator{\pt}{pt}
\DeclareMathOperator{\aug}{{aug}}
\DeclareMathOperator{\st}{st}

\DeclareMathOperator{\unit}{\mathds{1}}    
\newcommand{\catop}{^{\rm op}}

\newcommand{\catname}[1]{{\rm{#1}}}

\newcommand{\Spc}{\catname{Spc}}
\newcommand{\Sp}{\catname{Sp}}
\newcommand{\Vect}{\catname{Vect}}
\DeclareMathOperator{\Rep}{Rep}
\newcommand{\Top}{\catname{Top}}
\DeclareMathOperator{\TopGrpd}{TopGrpd}
\newcommand{\sTop}{\catname{sTop}}
\newcommand{\Mod}{\catname{Mod}}
\newcommand{\LMod}{\catname{LMod}}

\newcommand{\Fin}{\catname{Fin}}

\newcommand{\Cat}{\catname{Cat}}
\newcommand{\PrL}{\textup{Pr}^{\textup{L}}}
\newcommand{\PrR}{\textup{Pr}^{\textup{R}}}
\newcommand{\xPrLB}[1]{\PrL(\Bb_{/#1})}

\def\ulSp{\smash{\ul{{\setbox0=\hbox{\Sp}\dp0=0.5pt \ht0=0pt \box0\relax}}}}
\def\ulSpc{\smash{\ul{{\setbox0=\hbox{\Spc}\dp0=0.5pt \ht0=0pt \box0\relax}}}}
\def\ulOrbSp{\smash{\ul{{\setbox0=\hbox{\Orb\Sp}\dp0=0.5pt \ht0=0pt \box0\relax}}}}
\def\ulOrbSpc{\smash{\ul{{\setbox0=\hbox{\Orb\Spc}\dp0=0.5pt \ht0=0pt \box0\relax}}}}

\newcommand{\qednow}{\pushQED{\qed}\qedhere\popQED}

\newcommand{\Orb}{\rm{Orb}}
\newcommand{\Glo}{\rm{Glo}}

\newcommand{\qquadtext}[1]{\qquad\textrm{#1}\qquad}
\newcommand{\qin}{\quad\in\quad}
\newcommand{\ul}[1]{\underline{#1}}
\newcommand{\abs}[1]{\lvert #1 \rvert}

\let\op\relax
\DeclareMathOperator{\op}{op}

\let\S\relax
\DeclareMathOperator{\S}{\mathbb S}
\newcommand{\bbB}{\mathbb{B}}
\renewcommand{\phi}{\varphi}
\renewcommand{\epsilon}{\varepsilon}

\newcommand{\Aa}{\mathcal{A}}
\newcommand{\Bb}{\mathcal{B}}
\newcommand{\Cc}{\mathcal{C}}
\newcommand{\Dd}{\mathcal{D}}
\newcommand{\Ee}{\mathcal{E}}
\newcommand{\Ff}{\mathcal{F}}
\newcommand{\Gg}{\mathcal{G}}

\newcommand{\Ll}{\mathcal{L}}
\newcommand{\Mm}{\mathcal{M}}

\newcommand{\Oo}{\mathcal{O}}

\newcommand{\Ss}{\mathcal{S}}

\newcommand{\Uu}{\mathcal{U}}

\theoremstyle{plain}
\newtheorem{theorem}{Theorem}[section]

\newtheorem{corollary}[theorem]{Corollary}
\newtheorem{lemma}[theorem]{Lemma}
\newtheorem{proposition}[theorem]{Proposition}
\newtheorem{thmx}{Theorem}

\theoremstyle{definition}
\newtheorem{construction}[theorem]{Construction}
\newtheorem{convention}[theorem]{Convention}
\newtheorem{definition}[theorem]{Definition}
\newtheorem{example}[theorem]{Example}
\newtheorem{observation}[theorem]{Observation}
\newtheorem{remark}[theorem]{Remark}
\newtheorem{warning}[theorem]{Warning}
\newtheorem*{remark*}{Remark}

\newcommand\noloc{%
  \nobreak
  \mspace{6mu plus 1mu}
  {:}
  \nonscript\mkern-\thinmuskip
  \mathpunct{}
  \mspace{2mu}
}

\newcommand{\pullbacksign}{\hspace{-0.325ex}\tikz[baseline=(pb.base)]{\draw[line width=rule_thickness, line cap=round] (0,0) ++ (-2.45ex,0.45ex) -- ++ (1ex,0ex) -- ++ (0ex,1ex);\node (pb) at (0,0) {\phantom{x}};}}
\newcommand{\pushoutsign}{\hspace{0.2ex}\tikz[baseline=(po.base)]{\draw (0,0) ++ (2.45ex,-1.45ex) -- ++(0ex,1ex) -- ++ (1ex,0ex);\node (po) at (0,0) {\phantom{x}};}}
\tikzcdset{
	pullback/.style = {dash, phantom, "\vphantom{x}\smash{\pullbacksign}", start anchor=center, end anchor=center},
	pushout/.style = {dash, phantom, "\vphantom{x}\smash{\pushoutsign}", start anchor=center, end anchor=center},
}

\title{Twisted ambidexterity in equivariant homotopy theory}
\author{Bastiaan Cnossen}
\date{\today}

\begin{document}
\maketitle

\begin{abstract}
	We develop the concept of twisted ambidexterity in a parametrized presentably symmetric monoidal $\infty$-category, which generalizes the notion of ambidexterity by Hopkins and Lurie and the Wirthmüller isomorphisms in equivariant stable homotopy theory, and is closely related to Costenoble-Waner duality. Our main result establishes the parametrized $\infty$-category of genuine $G$-spectra for a compact Lie group $G$ as the universal example of a presentably symmetric monoidal $\infty$-category parametrized over $G$-spaces which is both stable and satisfies twisted ambidexterity for compact $G$-spaces. We further extend this result to the settings of orbispectra and proper genuine $G$-spectra for a Lie group $G$ which is not necessarily compact.
\end{abstract}

\section{Introduction}

A well-known phenomenon in representation theory, sometimes referred to as \textit{ambidexterity}, is that if $H$ is a subgroup of a finite group $G$, the restriction functor from $G$-representations to $H$-representations admits a left adjoint $\ind^G_H$ and a right adjoint $\coind^G_H$ which are naturally equivalent to each other. An analogue of this result in stable equivariant homotopy theory was established by Wirthmüller \cite{wirthmuller1974equivariant}: there is an equivalence $\ind^G_H \simeq \coind^G_H$ between the induction and coinduction functors from genuine $H$-spectra to genuine $G$-spectra.

The situation becomes more interesting when $G$ is a non-discrete compact Lie group. For a closed subgroup $H \leqslant G$, the induction and coinduction functors are only equivalent up to a `twist': for every genuine $H$-spectrum $X$ there is a natural equivalence of genuine $G$-spectra
\[
\ind^G_H(X \otimes S^{-L}) \iso \coind^G_H(X),
\]
called the \textit{Wirthmüller isomorphism}, where $S^{-L}$ is the inverse of the representation sphere of the tangent $H$-representation $L = T_{eH}(G/H)$. The construction of the comparison map in this case is geometric in nature and therefore substantially more involved than in the case of finite groups, see for example Schwede \cite[Section~3.2]{schwede2018global}. As a result, it is not immediately clear how to extend it to more general settings, like that of proper equivariant homotopy theory \cite{DHLPS2019Proper} or that of orbispectra \cite{Pardon2023Orbifold}, where some of the required geometric constructions are not available.

In \cite{maysigurdsson2006parametrized}, May and Sigurdsson approach the Wirthmüller isomorphism using the language of \textit{parametrized homotopy theory}. They split up the construction of the isomorphism into two steps:
\begin{enumerate}[(1)]
	\item The first step is formal: inspired by work of Costenoble and Waner \cite{CostenobleWaner2016equivariant}, May and Sigurdsson set up a notion of parametrized duality theory called \textit{Costenoble-Waner duality}. This theory provides an entirely categorical construction of a genuine $H$-spectrum $D_{G/H}$ equipped with a natural transformation
	\[
	\ind^G_H(- \otimes D_{G/H}) \implies \coind^G_H(-).
	\]
	This transformation is an equivalence if and only if $\S_H$ is \textit{Costenoble-Waner dualizable}, in which case its Costenoble-Waner dual is given by $D_{G/H}$.
	\item The second step is geometric: May and Sigurdsson use a parametrized form of the Pontryagin-Thom construction to produce explicit duality data that exhibit the genuine $H$-spectrum $S^{-L}$ as Costenoble-Waner dual to $\S_H$. As a consequence, one obtains an equivalence $D_{G/H} \simeq S^{-L}$ and the resulting map $\ind^G_H(X \otimes S^{-L}) \to \coind^G_H(X)$ is an equivalence.
\end{enumerate}

A key advantage of separating the construction into these two steps is that the first step does not need any geometric input and works in much greater generality: it is an instance of a notion we call \textit{twisted ambidexterity}. This allows for the construction of similar transformations in a wider range of contexts, including proper equivariant homotopy theory. It further allows one to formulate the Wirthmüller isomorphisms as a \textit{property} of a parametrized homotopy theory, making it possible to talk about the universal example of a parametrized homotopy theory satisfying this property.

\subsubsection{Twisted ambidexterity in parametrized homotopy theory}

The concept of twisted ambidexterity comes up in homotopy theory in the setting of local systems on spaces. For a space $A$, let $\Sp^A$ denote the functor category $\Fun(A,\Sp)$, also known as the $\infty$-category of local systems of spectra on $A$. The constant local system functor $A^*\colon \Sp \to \Sp^A$ admits both a left adjoint $A_! = \colim_A\colon \Sp^A \to \Sp$ as well as a right adjoint $A_* = \lim_A\colon \Sp^A \to \Sp$. Although $A_*$ does not preserve colimits in general, it can be universally approximated from the left by a colimit-preserving functor via a \textit{twisted norm map}
\[
\Nm_A\colon A_!(- \otimes D_A) \implies A_*(-),
\]
as shown by Nikolaus and Scholze \cite[Theorem~I.4.1(v)]{NikolausScholze2018Cyclic}. The parametrized spectrum $D_A \in \Sp^A$ is the \textit{dualizing spectrum of $A$}, introduced and studied by John Klein \cite{klein2001dualizing}.\footnote{Klein considered connected spaces $A = BG$ for topological groups $G$, and wrote $D_G$ rather than $D_{BG}$.} In case $A$ is a compact space, the functor $A_*$ already preserves colimits, and the twisted norm map $\Nm_A$ is an equivalence. When $A = M$ is a compact smooth manifold, $D_M = S^{-T_M}$ is the inverse of the one-point-compactification of the tangent bundle of $M$, and the resulting equivalence between the cohomology of $M$ and a shift of the homology of $M$ recovers twisted Poincaré duality.

In this article, we introduce a framework for twisted ambidexterity which generalizes the above story for local systems of spectra in two ways. As a first generalization, we replace the $\infty$-category of spectra by an arbitrary presentably symmetric monoidal $\infty$-category $\Cc$, in which case the transformation $A_!(- \otimes D_A) \colon \Cc^A \to \Cc$ is the universal $\Cc$-linear colimit-preserving approximation of $A_* \colon \Cc^A \to \Cc$. As a second generalization, following Ando, Blumberg and Gepner \cite{AndoBlumbergGepner2018Parametrized}, we consider homotopy theories parametrized over an arbitrary $\infty$-topos $\Bb$ in place of the $\infty$-category of spaces, allowing for applications in equivariant homotopy theory. The role of $\Cc$ is now played by certain limit-preserving functors $\Cc\colon \Bb\catop \to \CAlg(\PrL)$ known as \textit{presentably symmetric monoidal $\Bb$-categories}, see \Cref{def:Pres_Sym_Mon_BCategory}. Given an object $A \in \Bb$, one can enhance $A_!$ and $A_*$ to \textit{parametrized} functors $\Cc^A \to \Cc$, and one can again construct an object $D_A \in \Cc(A)$ and a twisted norm map $\Nm_A\colon A_!(- \otimes D_A) \to A_*$ universally approximating $A_*$ by a $\Cc$-linear colimit-preserving parametrized functor; see \Cref{prop:UniversalPropertyTwistedNormMap} for a precise statement. If $\Nm_A$ is an equivalence, we will say that $A$ is \textit{twisted $\Cc$-ambidextrous}. 

The notion of twisted ambidexterity can be regarded as a generalization of the notion of \textit{ambidexterity} by Hopkins and Lurie \cite{hopkinsLurie2013ambidexterity}: if $A$ is an $n$-truncated space, then the twisted norm map reduces to the norm map of \cite{hopkinsLurie2013ambidexterity} whenever the latter is defined, and $A$ is $\Cc$-ambidextrous in the sense of \cite[Definition~4.3.4]{hopkinsLurie2013ambidexterity} if and only if each of the spaces $A$, $\Omega A$, $\Omega^2A, \dots, \Omega^{n+1}A$ is twisted $\Cc$-ambidextrous, see \Cref{prop:AmbidexterityVsTwistedAmbidexterity}.

A \textit{relative} version of twisted ambidexterity for morphisms $f\colon A \to B$ in $\Bb$ is obtained by replacing  the $\infty$-topos $\Bb$ by its slice $\Bb_{/B}$, producing a twisted norm map $\Nm_f\colon f_!(- \otimes D_f) \to f_*(-)$.

\subsubsection{Twisted ambidexterity in equivariant homotopy theory}
The Wirthmüller isomorphism in equivariant homotopy theory discussed before may be understood as a special case of twisted ambidexterity. Given a compact Lie group $G$, we work in the setting of \textit{$G$-categories}, defined as $\infty$-categories parametrized over the $\infty$-topos of $G$-spaces. There is a $G$-category $\ulSp^G$ of genuine $G$-spectra, which assigns to the orbit space $G/H$ of a closed subgroup $H \leqslant G$ the $\infty$-category of genuine $H$-spectra. The $G$-category $\ulSp^G$ is presentably symmetric monoidal and is fiberwise stable, meaning that the $\infty$-category $\ulSp^G(A)$ is stable for every $G$-space $A$. Furthermore, it follows from results of \cite{maysigurdsson2006parametrized} that every compact $G$-space is twisted $\ulSp^G$-ambidextrous in the sense discussed above. The main result of this article is that $\ulSp^G$ is in fact \textit{universal} among $G$-categories satisfying the above properties:

\begin{thmx}[\Cref{thm:MainResult}]
	\label{introthm:UniversalPropertyGSpectra}
	For a compact Lie group $G$, the $G$-category $\ulSp^G$ is initial among fiberwise stable presentably symmetric monoidal $G$-categories $\Cc$ such that all compact $G$-spaces are twisted $\Cc$-ambidextrous.
\end{thmx}

In fact, it suffices to require that the orbit $G/H$ is twisted ambidextrous for every closed subgroup $H \leqslant G$. In this case, the resulting twisted ambidexterity isomorphism specializes to a formal Wirthmüller isomorphism of the form
\[
\Nm_{G/H}\colon \ind^G_H(- \otimes D_{G/H}) \iso \coind^G_H(-).
\]
Therefore, \Cref{introthm:UniversalPropertyGSpectra} may be interpreted as saying that genuine equivariant spectra form the universal theory of stable $G$-equivariant objects which admit formal Wirthmüller isomorphisms.

For the proof of \Cref{introthm:UniversalPropertyGSpectra}, we relate twisted ambidexterity for compact $G$-spaces to invertibility of representation spheres. If $\Cc$ is a pointed presentably symmetric monoidal $G$-category, one can show that the underlying $\infty$-category $\Cc(1)$ of $\Cc$ comes equipped with a canonical tensoring by pointed $G$-spaces, and thus the representation spheres $S^V$ act on $\Cc(1)$.

\begin{thmx}[\Cref{thm:CharacterizationGStableGCategories}]
	\label{introthm:CharacterizationGStableGCategories}
	Let $G$ be a compact Lie group and let $\Cc$ be a fiberwise stable presentably symmetric monoidal $G$-category. Then the following conditions are equivalent:
	\begin{enumerate}[(1)]
		\item For any $G$-representation $V$, the representation sphere $S^V$ acts invertibly on $\Cc(1)$;
		\item Every compact $G$-space is twisted $\Cc$-ambidextrous;
		\item For every closed subgroup $H \leqslant G$, the orbit $G/H$ is twisted $\Cc$-ambidextrous.
	\end{enumerate}
\end{thmx}

Since $\ulSp^G$ is initial with respect to condition (1), \Cref{introthm:UniversalPropertyGSpectra} is as a direct consequence of \Cref{introthm:CharacterizationGStableGCategories}. For the proof of \Cref{introthm:CharacterizationGStableGCategories}, we show in \Cref{subsec:Costenoble_Waner_Duality} that twisted ambidexterity can be formulated in terms of \textit{Costenoble-Waner duality}, a parametrized form of duality theory. The main ingredient for the implication (1) $\implies$ (2) is then a result of May and Sigurdsson \cite{maysigurdsson2006parametrized} about Costenoble-Waner duality in equivariant stable homotopy theory. The main ingredient for the implication (2) $\implies$ (1) is a result of Campion \cite{campion2023FreeDuals}, recalled in \Cref{sec:TheoremCampion}, which roughly says that dualizability of compact $G$-spaces implies invertibility of the representation spheres.

Our methods directly extend to the contexts of \textit{orbispectra} and \textit{proper genuine $G$-spectra}. We refer to \Cref{subsec:Orbispectra} and \Cref{subsec:ProperEquivariantHomotopyTheory} for precise definitions of the words appearing in the following two theorems:

\begin{thmx}[\Cref{thm:UniversalPropertyOrbispectra}]
	The orbicategory $\ulOrbSp$ of orbispectra is initial among fiberwise stable presentably symmetric monoidal orbicategories $\Cc$ such that every relatively compact morphism of orbispaces is twisted $\Cc$-ambidextrous.
\end{thmx}

\begin{thmx}[\Cref{thm:UniversalPropertyProperGspectra}]
	\label{introthm:UniversalPropertyProperGSpectra}
	For a Lie group $G$ which is not necessarily compact, the proper $G$-category $\ulSp^G$ of proper $G$-spectra is initial among fiberwise stable presentably symmetric monoidal proper $G$-categories $\Cc$ such that every relatively compact morphisms of proper $G$-spaces is twisted $\Cc$-ambidextrous.
\end{thmx}

In particular, \Cref{introthm:UniversalPropertyProperGSpectra} shows that for every closed subgroup $H$ of a Lie group $G$ with compact orbit space $G/H$ there is a formal Wirthmüller isomorphism
\[
\Nm_{G/H}\colon \ind^G_H(- \otimes D_{G/H}) \implies \coind^G_H(-).
\]

The methods developed in this article also allow us to give a formal description of the $\infty$-category of proper genuine $G$-spectra of Degrijse et al.\ \cite{DHLPS2019Proper} in case the Lie group $G$ has \textit{enough bundle representations} in the sense of \Cref{def:EnoughBundleRepresentations}: it is obtained from the $\infty$-category of pointed proper $G$-spaces by inverting the sphere bundles $S^{\xi}$ associated to finite-dimensional vector bundles $\xi$ over the classifying orbispace $\bbB G$, see \Cref{cor:EnoughBundleRepsImpliesFormalInversion}.

\subsubsection{Organization}
In \Cref{sec:parametrizedCategoryTheory}, we introduce the setting of parametrized higher category theory we will use in this article. Several foundational definitions and results from \cite{martini2021yoneda, martiniwolf2021limits, martiniwolf2022presentable} are recalled in \Cref{subsec:Recollections_Parametrized_Category_Theory}. In \Cref{subsec:Classification_CLinear_Functors}, we provide for every parametrized presentably symmetric monoidal $\infty$-category $\Cc$ a classification of $\Cc$-linear functors $\Cc^A \to \Cc^B$ in terms of objects of $\Cc(A \times B)$, where $A$ and $B$ are objects of the base $\infty$-topos. \Cref{subsec:Formal_Inversion} contains a discussion of formally inverting objects in parametrized symmetric monoidal $\infty$-categories.

In \Cref{sec:twistedambidexterity}, we introduce the notion of twisted ambidexterity in a parametrized presentably symmetric monoidal $\infty$-category $\Cc$. The twisted norm maps $\Nm_f\colon f_!(- \otimes D_f) \Rightarrow f_*(-)$ for a morphism $f\colon A \to B$ in the base $\infty$-topos are constructed in \Cref{sec:TwistedNormMap}, where also their universal property is established. In \Cref{subsec:SemiadditivityVsTwistedAmbidexterity}, we relate twisted ambidexterity to the notion of ambidexterity by \cite{hopkinsLurie2013ambidexterity} and to the notion of parametrized semiadditivity by \cite{nardin2016exposeIV, CLL_Global}. In \Cref{subsec:Costenoble_Waner_Duality} we discuss the relation between twisted ambidexterity and Costenoble-Waner duality.

In \Cref{sec:ambidexterity_in_equivariant_homotopy_theory}, we apply our methods in the context of equivariant homotopy theory. \Cref{subsec:genuineParametrizedSpectra} introduces the parametrized $\infty$-category of genuine $G$-spectra for a compact Lie group $G$, and \Cref{subsec:UniversalityOfWirthmullerIsomorphisms} establishes its universal property in terms of twisted ambidexterity. In \Cref{subsec:Orbispectra} and \Cref{subsec:ProperEquivariantHomotopyTheory} we extend this to the contexts of orbispectra and proper equivariant spectra, respectively.

\subsubsection{Relation to other work}
Our treatment of twisted ambidexterity draws on ideas from a wide range of prior work. The main inspiration is the concept of Costenoble-Waner duality in parametrized homotopy theory, introduced by Costenoble and Waner \cite{CostenobleWaner2016equivariant} under the name `homological duality' and further developed by May and Sigurdsson \cite{maysigurdsson2006parametrized}. The untwisted notion of ambidexterity appeared in the works of Hopkins and Lurie \cite{hopkinsLurie2013ambidexterity} and has been further studied by Harpaz \cite{harpaz2020ambidexterity} and Carmeli, Schlank and Yanovski \cite{CSY2022TeleAmbi, CSY2021AmbiHeight}. A treatment in the context of representation theory was given by Balmer and Dell'Ambrogio \cite{BalmerDellAmbrogio2020Mackey}. In the case of local systems on spaces, the universal property of the twisted norm map resembles classical assembly maps; see, for example, the discussion following Theorem I.4.1 of Nikolaus and Scholze \cite{NikolausScholze2018Cyclic}. The dualizing object in this case is the dualizing spectrum introduced by Klein \cite{klein2001dualizing} and studied by Bauer \cite{Bauer2004pcompact} and Rognes \cite{rognes2005stably}. A parametrized version of the assembly map was first introduced by Quigley and Shah \cite{QuigleyShay2021Tate} for equivariant genuine spectra. Other sources of inspiration are the `ambidexterity isomorphisms' of Hoyois \cite{hoyois2017sixoperations} and Bachmann and Hoyois \cite{BachmannHoyois2021Norms}, and the `purity equivalences' of Cisinski and Déglise \cite{CisinskiDeglise2019Triangulated}. Although similar in appearance, our approach is different from the formal Wirthmüller isomorphisms of Fausk, Hu and May \cite{FauskHuMay2003Isomorphisms} and Balmer, Dell'Ambrogio and Sanders \cite{BDS2016Grothendieck}, where no parametrized homotopy theory is involved.

Our characterization of genuine equivariant spectra in terms of fiberwise stability and formal Wirth-müller isomorphisms is heavily inspired by work of Blumberg \cite{Blumberg2006Continuous}, who described the category of genuine $G$-spectra for a compact Lie group $G$ in terms of continuous functors that satisfy excision and have (geometrically defined) Wirthmüller isomorphisms. For finite groups, a parametrized universal property for genuine $G$-spectra in terms of stability and Wirthmüller equivalences $\ind^G_H(-) \simeq \coind^G_H(-)$ was outlined by Nardin \cite{nardin2016exposeIV}, and a similar result in the context of global homotopy theory was proved by Lenz, Linskens and the author in \cite{CLL_Global}.

\subsubsection{Acknowledgements}
I am truly grateful to the following individuals for their contributions to this article. Stefan Schwede sparked the idea behind this project and has been a source of support and encouragement throughout the process. Louis Martini and Sebastian Wolf were invaluable resources for discussing the intricacies of parametrized category theory. Sil Linskens was always available for helpful discussions, especially regarding formal inversions and proper equivariant homotopy theory. Shachar Carmeli provided the idea for the universal property of the twisted norm map. Tim Campion graciously allowed me to include a result from his PhD dissertation in the appendix. I further had useful conversations with Tobias Lenz, Maxime Ramzi, and Lior Yanovski about some of the contents of this article. I thank Phil Pützstück for many useful comments on an earlier version of this article. Finally, I would like to thank the Max Planck Institute for Mathematics for its financial support.

\section{Parametrized category theory}
\label{sec:parametrizedCategoryTheory}
We start by recalling the setup of parametrized category theory we are working in, and establishing some analogues of well-known results in non-parametrized category theory. Our main references are the articles \cite{martini2021yoneda}, \cite{martiniwolf2021limits} and \cite{martiniwolf2022presentable} of Martini and Wolf; see in particular \cite[Section~2.6]{martiniwolf2022presentable} for a short overview of the theory. An earlier framework for parametrized category theory was given by Barwick, Dotto, Glasman, Nardin and Shah \cite{BDGNS2016parametrized,BDGNS2016ExposeI,shah2021parametrized,nardin2016exposeIV}.

\begin{convention}
	Since most of the $\infty$-categories we will be working with are presentable, we will by convention take all $\infty$-categories to be large unless explicitly specified that they are small. In particular, $\Cat_{\infty}$ denotes the (very large) $\infty$-category of large $\infty$-categories; in \cite{martiniwolf2022presentable}, this is denoted by $\widehat{\Cat}_{\infty}$ instead.
\end{convention}

\subsection{Recollections on parametrized category theory}
\label{subsec:Recollections_Parametrized_Category_Theory}
Throughout this section, we fix an $\infty$-topos $\Bb$.

\begin{definition}
	\label{def:BCategory}
	A \textit{$\Bb$-category} is a sheaf of $\infty$-categories on $\Bb$, i.e., a limit-preserving functor $\Bb\catop \to \Cat_{\infty}$. Its \textit{underlying $\infty$-category} $\Gamma(\Cc)$ is the $\infty$-category $\Cc(1)$, where $1 \in \Bb$ is the terminal object. Given two $\Bb$-categories $\Cc$ and $\Dd$, a \textit{$\Bb$-functor} is a natural transformation $\Cc \to \Dd$. We let $\Cat(\Bb) \subseteq \Fun(\Bb\catop,\Cat_{\infty})$ denote the (very large) $\infty$-category of $\Bb$-categories and $\Bb$-functors.
	
	By \cite[Proposition~3.5.1]{martini2021yoneda}, a $\Bb$-category may equivalently be encoded as a \textit{(large) category internal to $\Bb$}, which is the perspective used in \cite{martini2021yoneda, martiniwolf2021limits, martiniwolf2022presentable}.
\end{definition}

When $\Bb$ is the $\infty$-topos $\Spc$ of spaces (a.k.a.\ $\infty$-groupoids or animae), the underlying category functor provides an equivalence
\[
\Gamma\colon \Cat(\Spc) \xrightarrow{\;\simeq\;} \Cat_{\infty};
\]
its inverse sends an $\infty$-category $\Cc$ to the functor $\Spc\catop \to \Cat_{\infty}\colon A \mapsto \Fun(A,\Cc)$. More generally, when $\Bb = \PSh(T)$ is the $\infty$-topos of presheaves on some small $\infty$-category $T$, restriction to the representable objects induces an equivalence $\Cat(\PSh(T)) \iso \Fun(T\catop,\Cat_{\infty})$, and the theory of $\Bb$-categories reduces to that of $T$-$\infty$-categories studied by Barwick et al.\ \cite{BDGNS2016ExposeI}. Since all the examples of $\infty$-topoi considered in this article will be presheaf topoi, the reader uncomfortable with the language of $\infty$-topoi may replace $\Bb$ by $\PSh(T)$ throughout.

\begin{example}[$\Bb$-groupoids]
	Every object $A \in \Bb$ can naturally be regarded as a $\Bb$-category via the Yoneda embedding $\Bb \hookrightarrow \FunR(\Bb\catop,\Spc) \subseteq \FunR(\Bb\catop,\Cat_{\infty})$. The $\Bb$-categories of this form are called \textit{$\Bb$-groupoids}.
\end{example}

\begin{example}[Base change]
	\label{ex:BaseChange}
	Any geometric morphism $f^*\colon \Aa \rightleftarrows \Bb\noloc f_*$ induces an adjunction $f^*\colon \Cat(\Aa) \rightleftarrows \Cat(\Bb)\noloc f_*$, see \cite[Section~3.3]{martini2021yoneda}, \cite[Section~2.6]{martiniwolf2021limits}. The right adjoint $f_*$ is explicitly given by precomposing a $\Bb$-category $\Cc\colon \Bb\catop \to \Cat_{\infty}$ with $f^*\colon \Aa\catop \to \Bb\catop$.
\end{example}

\begin{example}[Locally constant $\Bb$-categories]
	\label{ex:LocallyConstantBCategories}
	Applying \Cref{ex:BaseChange} to the geometric morphism $\LConst\colon \Spc \rightleftarrows \Bb \noloc \Gamma$, we obtain an adjunction $\LConst\colon \Cat_{\infty} \rightleftarrows \Cat(\Bb)\noloc \Gamma$, where $\Gamma$ is the underlying $\infty$-category functor. The $\Bb$-categories in the image of $\LConst$ are called \textit{locally constant}.\footnote{In \cite{martiniwolf2021limits} these are simply called \textit{constant}.}
\end{example}

\begin{example}[Passing to slice topoi]
	\label{ex:SliceTopoi}
	For an object $B \in \Bb$, applying \Cref{ex:BaseChange} to the étale geometric morphism $- \times B = \pi_B^*\colon \Bb \rightleftarrows \Bb_{/B}\noloc (\pi_B)_*$, we get an adjunction
	\[
	\pi_B^*\colon \Cat(\Bb) \rightleftarrows \Cat(\Bb_{/B})\noloc (\pi_B)_*.
	\]
	For $\Cc \in \Cat(\Bb)$, the $\Bb_{/B}$-category $\pi_B^*\Cc$ is given by precomposing $\Cc$ with the forgetful functor $\Bb_{/B} \to \Bb$, as this is a left adjoint to $-\times B\colon \Bb \to \Bb_{/B}$.
\end{example}

\begin{example}[Parametrized functor categories]
	\label{ex:ParametrizedFunctorCategory}
	By \cite[Proposition~3.2.11]{martini2021yoneda}, the $\infty$-category $\Cat(\Bb)$ of $\Bb$-categories is cartesian closed: for all $\Cc, \Dd \in \Cat(\Bb)$ there is an internal hom-object\footnote{Martini \cite{martini2021yoneda} denotes $\ul{\Fun}_{\Bb}(\Cc,\Dd)$ by $[\Cc,\Dd]$.} $\ul{\Fun}_\Bb(\Cc,\Dd)$, called the \textit{$\Bb$-category of $\Bb$-functors from $\Cc$ to $\Dd$}.
	We let $\Fun_{\Bb}(\Cc,\Dd)$ denote the underlying $\infty$-category of $\ul{\Fun}_\Bb(\Cc,\Dd)$. Its morphisms are called \textit{$\Bb$-transformations}.
\end{example}

\begin{definition}
	A \textit{symmetric monoidal $\Bb$-category} is a commutative monoid in the $\infty$-category $\Cat(\Bb)$, or equivalently a limit-preserving functor $\Cc\colon \Bb\catop \to \CMon(\Cat_{\infty})$. For $B \in \Bb$, we denote the tensor product and monoidal unit of $\Cc(B)$ by $- \otimes_B -$ and $\unit_B$, respectively.
\end{definition}

\subsubsection{Presentable \texorpdfstring{$\Bb$}{B}-categories}

We give a brief overview of the theory of presentable $\Bb$-categories, developed by Martini and Wolf \cite{martiniwolf2022presentable}.

\begin{definition}
	\label{def:BPresentableBCategory}
	A $\Bb$-category $\Cc\colon \Bb\catop \to \Cat_{\infty}$ is called \textit{fiberwise presentable} if it factors (necessarily uniquely) through the subcategory $\PrL \subseteq \Cat_{\infty}$ of presentable $\infty$-categories and colimit preserving functors. We say that $\Cc$ is \textit{presentable} if it is fiberwise presentable and additionally satisfies the following two conditions:
	\begin{enumerate}[(1)]
		\item (Left adjoints) For every morphism $f\colon A \to B$ in $\Bb$, the restriction functor $f^*\colon \Cc(B) \to \Cc(A)$ has a left adjoint $f_!\colon \Cc(A) \to \Cc(B)$;
		\item (Left base change) For every pullback square
		\begin{equation*}
			\begin{tikzcd}
				A' \ar[r, "\alpha" above] \ar[d, "f'" left] \drar[pullback] & A \ar[d, "f"]   \\
				B' \ar[r, "\beta" below] & B
			\end{tikzcd}
		\end{equation*}
		in $\Bb$, the Beck-Chevalley transformation $f'_! \alpha^* \Rightarrow \beta^*f_!$ is an equivalence.
	\end{enumerate}
	If $\Cc$ and $\Dd$ are presentable $\Bb$-categories, we say that a $\Bb$-functor $F\colon \Cc \to \Dd$ \textit{preserves (parametrized) colimits} if the following two properties are satisfied:
	\begin{enumerate}[(1)]
		\item For every object $B \in \Bb$, the functor $F(B)\colon \Cc(B) \to \Dd(B)$ preserves small colimits;
		\item For every morphism $f\colon A \to B$ in $\Bb$, the Beck-Chevalley transformation $f_! \circ F(A) \implies F(B) \circ f_!$ is an equivalence.
	\end{enumerate}
	By \cite[Theorem~A]{martiniwolf2022presentable} and \cite[Proposition 4.2.3]{martiniwolf2021limits}, these definitions agree with the definitions of Martini and Wolf \cite{martiniwolf2022presentable}.
\end{definition}

\begin{remark}
	For a presentable $\Bb$-category, the restriction functor $f^*\colon \Cc(B) \to \Cc(A)$ is assumed to admit a right adjoint $f_*\colon \Cc(A) \to \Cc(B)$ for every morphism $f\colon A \to B$. By passing to right adjoints in condition (2) in \Cref{def:BPresentableBCategory}, we see that also the other Beck-Chevalley transformation $\beta_*f^* \Rightarrow {f'}^*\alpha_*$ is an equivalence. It follows that any presentable $\Bb$-category admits all (parametrized) limits and colimits.
\end{remark}

The next definition introduces the presentable $\Bb$-category $\Omega_{\Bb}$, the $\Bb$-parametrized analogue of the $\infty$-category of spaces.

\begin{definition}
	\label{def:UniverseOfGroupoids}
	The target functor $d_0\colon \Bb^{\Delta^1} \to \Bb$ is a cartesian fibration, and by \cite[Theorem 6.1.3.9]{lurie2009HTT} is classified by a limit-preserving functor
	\[
	\Omega_{\Bb}\colon \Bb\catop \to \PrL, \qquad B \mapsto \Bb_{/B}, \qquad (f\colon A \to B) \mapsto (f^*\colon \Bb_{/B} \to \Bb_{/A}).
	\]
	The pullback functors $f^*\colon \Bb_{/B} \to \Bb_{/A}$ have left adjoints $f \circ -\colon \Bb_{/A} \to \Bb_{/B}$ which satisfy the Beck-Chevalley condition, so $\Omega_{\Bb}$ is a presentable $\Bb$-category. We call $\Omega_{\Bb}$ the \textit{$\Bb$-category of $\Bb$-groupoids}.\footnote{It is called the \textit{universe for groupoids} by Martini \cite{martini2021yoneda}.}
\end{definition}

\begin{remark*}
    When $\Bb = \PSh(T)$ is a presheaf topos, the $\Bb$-category $\Omega_{\Bb}$ corresponds to the $T$-category $\ul{\Spc}_T$ from \cite{shah2021parametrized} and \cite[Example~2.1.11]{CLL_Global}.
\end{remark*}

\begin{definition}
	\label{def:PrL(B)}
	Let $\PrL(\Bb)$ denote the (non-full) subcategory of $\Cat(\Bb)$ spanned by the presentable $\Bb$-categories and the colimit preserving $\Bb$-functors. For presentable $\Bb$-categories $\Cc$ and $\Dd$, we let $\FunL_{\Bb}(\Cc,\Dd)$ denote the subcategory of $\Fun_{\Bb}(\Cc,\Dd)$ spanned by those $\Bb$-functors $\Cc \to \Dd$ that preserve colimits.
\end{definition}

Recall from \cite[Definition~2.7.2]{martiniwolf2021limits} that the $\infty$-categories $\Cat(\Bb_{/B})$ assemble into a (very large) $\Bb$-category $\Cat_{\Bb}$. The subcategories $\PrL(\Bb_{/B}) \subseteq \Cat(\Bb_{/B})$ assemble into a parametrized subcategory $\PrL_{\Bb} \subseteq \Cat_{\Bb}$, see \cite[Definition~6.4.1]{martiniwolf2022presentable}.

\begin{proposition}[{\cite[Proposition 4.5.1]{martiniwolf2021limits}, \cite[Definition~6.4.1, Remark~6.4.2, Corollary~6.4.11]{martiniwolf2022presentable}}]
	\label{prop:BCategory_of_Presentable_Bcategories}
	The $\Bb$-categories $\Cat_{\Bb}$ and $\PrL_{\Bb}$ are complete and cocomplete, and the inclusion $\PrL_{\Bb} \hookrightarrow \Cat_{\Bb}$ preserves limits. \qed
\end{proposition}

\begin{corollary}
	\label{cor:Base_Change_Adjunction_Presentable}
	For each $B \in \Bb$ the adjunction $\pi_B^*\colon \Cat(\Bb) \rightleftarrows \Cat(\Bb_{/B})\noloc (\pi_B)_*$ from \Cref{ex:SliceTopoi} restricts to an adjunction $\pi_B^*\colon \PrL(\Bb) \rightleftarrows \xPrLB{B}\noloc (\pi_B)_*$.
\end{corollary}
\begin{proof}
    The functor $\pi_B^*\colon \Cat(\Bb) \to \Cat(\Bb_{/B})$ restricts to presentable categories since $\PrL_{\Bb} \subseteq \Cat_{\Bb}$ is a parametrized subcategory. The adjunction $\pi_B^* \dashv (\pi_B)_*$ restricts to presentable categories since the parametrized subcategory $\PrL_{\Bb} \subseteq \Cat_{\Bb}$ is closed under (parametrized) limits.
\end{proof}

\subsubsection{Tensor products of presentable \texorpdfstring{$\Bb$}{B}-categories}

Given two presentable $\Bb$-categories $\Cc$ and $\Dd$, there exists a presentable $\Bb$-category $\Cc \otimes \Dd$ called the \textit{tensor product} of $\Cc$ and $\Dd$, which is characterized by the property that colimit-preserving $\Bb$-functors into a third presentable $\Bb$-category $\Ee$ correspond to $\Bb$-functors $\Cc \times \Dd \to \Ee$ which preserve colimits in both variables, see \cite[Section~8.2]{martiniwolf2022presentable}. The tensor product equips $\PrL(\Bb)$ with the structure of a symmetric monoidal $\infty$-category $\PrL(\Bb)^{\otimes}$ whose monoidal unit is $\Omega_{\Bb}$. In fact, in \cite[Proposition~8.2.9]{martiniwolf2022presentable} it is shown that a parametrized version of this construction equips the (very large) $\Bb$-category $\PrL_{\Bb}$ of presentable $\Bb$-categories with the structure of a symmetric monoidal $\Bb$-category. In particular, the restriction functor $\pi_B^*\colon \PrL(\Bb) \to \PrL(\Bb_{/B})$ is symmetric monoidal for every object $B \in \Bb$.

By \cite[Proposition~8.2.11]{martiniwolf2022presentable}, a formula for the tensor product of two presentable $\Bb$-categories $\Cc$ and $\Dd$ is given by $\Cc \otimes \Dd \simeq \ul{\Fun}^{\mathrm{R}}_{\Bb}(\Cc\catop,\Dd)$, the full sub-$\Bb$-category of $\ul{\Fun}_{\Bb}(\Cc\catop,\Dd)$ spanned by the limit-preserving $\Bb$-functors $\Cc\catop \to \Dd$. For a third presentable $\Bb$-category $\Ee$, there is an equivalence
\[
\FunL_{\Bb}(\Cc \otimes \Dd,\Ee) \simeq \FunL_{\Bb}(\Cc, \ul{\Fun}^{\mathrm{L}}_{\Bb}(\Dd,\Ee)),
\]
and it follows in particular that the $\Bb$-functor $\Cc \otimes - \colon \PrL_{\Bb} \to \PrL_{\Bb}$ preserves colimits.

\begin{definition}[{\cite[Definition~8.2.10]{martiniwolf2022presentable}}]
	\label{def:Pres_Sym_Mon_BCategory}
	A \textit{presentably symmetric monoidal $\Bb$-category} is a commutative algebra object in the symmetric monoidal $\infty$-category $\PrL(\Bb)^{\otimes}$.
\end{definition}

As $\PrL(\Bb)^{\otimes}$ is a subcategory of $\Cat(\Bb)^{\times}$, a symmetric monoidal $\Bb$-category $\Cc$ is presentably symmetric monoidal if and only if $\Cc$ is presentable and the tensor product $\Bb$-functor $- \otimes -\colon \Cc \times \Cc \to \Cc$ preserves colimits in both variables. Unwinding definitions, this boils down to the following two non-parametrized conditions:
\begin{enumerate}[(1)]
	\item (Fiberwise presentably symmetric monoidal) For every $B \in \Bb$, the tensor product functor $- \otimes_B -\colon \Cc(B) \times \Cc(B) \to \Cc(B)$ of $\Cc(B)$ preserves small colimits in both variables.
	\item \label{it:LeftProjectionFormula} (Left projection formula) For every morphism $f\colon A \to B$ in $\Bb$ and all objects $X \in \Cc(B)$ and $Y \in \Cc(A)$, the exchange morphism $f_!(f^*(X) \otimes_A Y) \to X \otimes_B f_!(Y)$ is an equivalence.
\end{enumerate}

In particular, the data of a presentably symmetric monoidal $\Bb$-category is the same as that of a limit-preserving functor $\Cc\colon \Bb\catop \to \CAlg(\PrL)$ such that all the symmetric monoidal restriction functors $f^*\colon \Cc(B) \to \Cc(A)$ admit left adjoints that satisfy base change and satisfy the left projection formula. Such structure has also been called a \textit{Wirthmüller context}\footnote{In \cite{AndoBlumbergGepner2018Parametrized}, the Beck-Chevalley conditions and the left projection formula were not taken as part of the axioms.} \cite[Definition~6.7, Proposition~6.8]{AndoBlumbergGepner2018Parametrized} or a \textit{presentable pullback formalism} \cite[Definition~4.5]{DrewGallauer2022Universal}.

\subsubsection{Embedding \texorpdfstring{$\Bb$}{B}-modules into presentable \texorpdfstring{$\Bb$}{B}-categories}
The $\infty$-topos $\Bb$ comes equipped with a presentably symmetric monoidal structure given by cartesian product, and thus we may consider left modules over it in $\PrL$, which we will refer to as \textit{$\Bb$-modules}. It was shown in \cite[Section~8.3]{martiniwolf2022presentable} that the $\Bb$-modules embed fully faithfully into presentable $\Bb$-categories:

\begin{proposition}[{\cite[Proposition~8.3.2, Lemma~8.3.3, Proposition~8.3.5, Proposition~8.3.6]{martiniwolf2022presentable}}]
	\label{prop:Embedding_BModules_Into_Presentable_BCategories}
	There is a symmetric monoidal fully faithful functor
	\[
	- \otimes_{\Bb} \Omega_{\Bb}\colon \Mod_{\Bb}(\PrL)
	\lhook\joinrel\xrightarrow{\;\;\;\;\;\;} 
	\PrL(\Bb)
	\]
	which admits a right adjoint $\Gamma^{\mathrm{lin}}\colon \PrL(\Bb) \to \Mod_{\Bb}(\PrL)$ whose composition with the forgetful functor $\Mod_{\Bb}(\PrL) \to \PrL$ is the global section functor $\Gamma\colon \PrL(\Bb) \to \PrL$. \qed
\end{proposition}

Given a $\Bb$-module $\Dd$, the presentable $\Bb$-category $\Dd \otimes_{\Bb} \Omega_{\Bb}$ is given at an object $B \in \Bb$ by the relative tensor product $\Dd \otimes_{\Bb} \Bb_{/B}$ in $\PrL$. The functoriality of this expression in $B$ is informally described as follows: given a morphism $f\colon A \to B$ in $\Bb$, one considers the composition functor $f_!\colon \Bb_{/A} \to \Bb_{/B}$ as a $\Bb$-linear functor in $\PrL$, tensors it with $\Dd$ over $\Bb$ to get a map $\Dd \otimes_{\Bb} \Bb_{/A} \to \Dd \otimes_{\Bb} \Bb_{/B}$, and then passes to its right adjoint. We refer to \cite[Section~8.3]{martiniwolf2022presentable} for a precise construction.

It follows directly from \Cref{prop:Embedding_BModules_Into_Presentable_BCategories} that every commutative $\Bb$-algebra $\Dd$ in $\PrL$ gives rise to a presentably symmetric monoidal $\Bb$-category $\Dd \otimes_{\Bb} \Omega_{\Bb}$:

\begin{corollary}
	\label{cor:Embedding_BAlgebras_Into_Presentably_Symmetric_Monoidal_BCategories}
	The adjunction from \Cref{prop:Embedding_BModules_Into_Presentable_BCategories} induces an adjunction
	\[
	\begin{tikzcd}
		- \otimes_{\Bb} \Omega_{\Bb}\colon \CAlg_{\Bb}(\PrL) \rar[shift left, hookrightarrow] & \CAlg(\PrL(\Bb))\noloc \Gamma^{\mathrm{lin}} \lar[shift left].
	\end{tikzcd} \qednow
	\]
\end{corollary}

We will also need a $\Cc$-linear version of this result:
\begin{lemma}
	\label{lem:AdjunctionCLinearCategories}
	For every presentably symmetric monoidal $\Bb$-category $\Cc$, the adjunction from \Cref{prop:Embedding_BModules_Into_Presentable_BCategories} induces an adjunction
	\[
	- \otimes_{\Gamma(\Cc)} \Cc \colon \Mod_{\Gamma(\Cc)}(\PrL) \rightleftarrows \Mod_{\Cc}(\PrL(\Bb))\noloc \Gamma^{\Cc}.
	\]
	Furthermore:
	\begin{enumerate}[(1)]
		\item the left adjoint $- \otimes_{\Gamma(\Cc)} \Cc$ is fully faithful and symmetric monoidal;
		\item the right adjoint $\Gamma^{\Cc}$ preserves colimits and satisfies the projection formula: the canonical map
		\[
		\Dd \otimes_{\Gamma(\Cc)} \Gamma^{\Cc}(\Ee) \to \Gamma^{\Cc}((\Dd \otimes_{\Gamma(\Cc)} \Cc) \otimes_{\Cc} \Ee)
		\]
		is an equivalence for every $\Gamma(\Cc)$-module $\Dd$ and $\Cc$-module $\Ee$.
	\end{enumerate}
\end{lemma}
\begin{proof}
	The adjunction is obtained as a concatenation of the following two adjunctions:
	\[
	\begin{tikzcd}
		\Mod_{\Gamma(\Cc)}(\PrL) \rar[shift left, hookrightarrow, "- \otimes_{\Bb} \Omega_{\Bb}"] & \Mod_{\Gamma(\Cc) \otimes_{\Bb} \Omega_{\Bb}}(\PrL(\Bb))\lar[shift left, "\Gamma^{\mathrm{lin}}"] \rar[shift left, "-\otimes_{\Gamma(\Cc) \otimes_{\Bb} \Omega_{\Bb}} \Cc"] & \Mod_{\Cc}(\PrL(\Bb) \lar[shift left, "\fgt"].
	\end{tikzcd}
	\]
	The first adjunction is an instance of \cite[Example~7.3.2.8]{lurie2016HA} while the second adjunction is the base change adjunction of the counit map $\Gamma(\Cc) \otimes_{\Bb} \Omega_{\Bb} \to \Cc$. As both left adjoints are symmetric monoidal, so is their composite $- \otimes_{\Cc} \Gamma(\Cc)$. 
	
	To show that the functor $\Gamma^{\Cc}\colon \Mod_{\Cc}(\PrL(\Bb)) \to \Mod_{\Gamma(\Cc)}(\PrL)$ preserves colimits, it suffices to show that the functor $\Gamma\colon \PrL(\Bb) \to \PrL$ preserves colimits, as colimits in module categories are computed underlying. As $\PrL(\Bb) \simeq (\PrR(\Bb))\catop$ by \cite[Proposition~6.4.7]{martiniwolf2022presentable}, we may pass to opposite categories and instead show that $\Gamma^R\colon \PrR(\Bb) \to \PrR$ preserves limits. By \cite[Proposition~6.4.10]{martiniwolf2022presentable} the inclusion $\PrR(\Bb) \hookrightarrow \Cat(\Bb)$ preserves limits, hence it suffices to show that the evaluation functor $\Gamma = \ev_1\colon \Cat(\Bb) \to \Cat_{\infty}$ preserves limits. This is clear, as limits in $\Cat(\Bb) = \FunR(\Bb\catop,\Cat_{\infty})$ are computed pointwise.
	
	We next show that the functor $\Gamma^{\Cc}$ satisfies the projection formula. Recall that this is automatic whenever $\Dd$ is dualizable in $\Mod_{\Gamma(\Cc)}(\PrL)$, see for example \cite[Proposition~3.12]{FauskHuMay2003Isomorphisms} or \cite[Lemma~4.15]{CCRY_Characters}. Since we already showed that $\Gamma^{\Cc}$ commutes with colimits, it will suffice to show that the $\infty$-category $\Mod_{\Gamma(\Cc)}(\PrL)$ is generated under colimits by dualizable objects. Since the free $\Gamma(\Cc)$-module functor $\Gamma(\Cc) \otimes -\colon \PrL \to \Mod_{\Gamma(\Cc)}(\PrL)$ is symmetric monoidal and its image generates $\Mod_{\Gamma(\Cc)}(\PrL)$ under colimits, it suffices to show that $\PrL$ is generated under colimits by dualizable objects. This holds by \cite[Lemma 7.14]{RagimovSchlank}.
	
	Finally, we show that the left adjoint $- \otimes_{\Gamma(\Cc)} \Cc$ is fully faithful, or equivalently that the unit $\Dd \to \Gamma^{\Cc}(\Dd \otimes_{\Gamma(\Cc)} \Cc)$ of the adjunction is an equivalence. This is a special case of the projection formula applied to $\Ee = \Cc$.
\end{proof}

For general $\Cc$-modules $\Dd$ and $\Ee$ in $\PrL(\Bb)$, it is not clear in general how the parametrized relative tensor product $\Dd \otimes_{\Cc} \Ee$ looks like. The situation improves when $\Dd$ comes from a $\Gamma(\Cc)$-module in $\PrL$:

\begin{corollary}
	\label{cor:TensorProductPointwise}
	Let $\Cc \in \CAlg(\PrL(\Bb))$, let $\Dd, \Ee \in \Mod_{\Cc}(\PrL)$, and assume that $\Dd = \Dd_0 \otimes_{\Gamma(\Cc)} \Cc$ for some $\Dd_0 \in \Mod_{\Gamma(\Cc)}(\PrL)$. Then there is for every $A \in \Bb$ an equivalence of $\Cc(A)$-linear $\infty$-categories
	\[
	\Dd(A) \otimes_{\Cc(A)} \Ee(A) \iso (\Dd \otimes_{\Cc} \Ee)(A).
	\]
\end{corollary}
\begin{proof}
	By passing to the slice category $\Bb_{/A}$, we may assume that $A$ is terminal in $\Bb$. In this case, \Cref{lem:AdjunctionCLinearCategories} provides equivalences of $\Gamma(\Cc)$-linear $\infty$-categories
	\[
	\Dd(1) \otimes_{\Cc(1)} \Ee(1) = \Gamma^{\Cc}(\Dd) \otimes_{\Gamma(\Cc)} \Gamma^{\Cc}(\Ee) \xleftarrow{\sim} \Dd_0 \otimes_{\Gamma(\Cc)} \Gamma^{\Cc}(\Ee) \xrightarrow{\sim} \Gamma^{\Cc}(\Dd \otimes_{\Cc} \Ee) = (\Dd \otimes_{\Cc} \Ee)(1).\qedhere
	\]
\end{proof}

\subsection{Classification of \texorpdfstring{$\Cc$}{C}-linear functors}
\label{subsec:Classification_CLinear_Functors}

Given a presentably symmetric monoidal $\Bb$-category $\Cc \in \CAlg(\PrL(\Bb))$, one may define for every object $A \in \Bb$ a presentable $\Bb$-category $\Cc^A$ given by $\Cc^A(B) = \Cc(A \times B)$. It comes equipped with a natural tensoring over $\Cc$ in $\PrL(\Bb)$. The goal of this section is to give a classification of $\Cc$-linear $\Bb$-functors $F\colon \Cc^A \to \Cc^B$ for objects $A,B \in \Bb$: we will show in \Cref{thm:Classification_CLinear_Functors} below that evaluation at the object $\Delta_!\unit_A \in \Cc(A \times A) = \Cc^A(A)$ determines an equivalence
\[
\ev_{\Delta_!\unit}\colon \Fun_{\Cc}(\Cc^A,\Cc^B) \xrightarrow{\sim} \Cc(A \times B),
\]
whose inverse sends an object $D \in \Cc(A \times B)$ to the $\Bb$-functor $(\pr_B)_!(\pr_A^*(-) \otimes D)\colon \Cc^A \to \Cc^B$.

\subsubsection{The \texorpdfstring{$\Bb$}{B}-category of \texorpdfstring{$\Cc$}{C}-linear \texorpdfstring{$\Bb$}{B}-categories}
Fix a presentably symmetric monoidal $\Bb$-category $\Cc \in \CAlg(\PrL(\Bb))$, and consider the $\infty$-category $\Mod_{\Cc}(\PrL(\Bb))$ of left $\Cc$-modules in $\PrL(\Bb)$. We refer to the objects of $\Mod_{\Cc}(\PrL(\Bb))$ as \textit{$\Cc$-linear $\Bb$-categories} and to the morphisms as \textit{$\Cc$-linear $\Bb$-functors}. In particular, $\Cc$-linear $\Bb$-functors will always preserve colimits by convention.

For any object $B \in \Bb$, we obtain an object $\pi_B^*\Cc \in \CAlg(\PrL(\Bb_{/B}))$. By \cite[Definition~7.2.2]{martiniwolf2022presentable}, the $\infty$-categories $\Mod_{\pi_B^*\Cc}(\xPrLB{B})$ assemble into a $\Bb$-category $\Mod_{\Cc}(\PrL_{\Bb})$. By \cite[Proposition~7.2.7]{martiniwolf2022presentable}, we have:
\begin{enumerate}[(1)]
	\item The $\Bb$-category $\Mod_{\Cc}(\PrL_{\Bb})$ is complete and cocomplete;
	\item The $\Bb$-functor $\Mod_{\Cc}(\PrL_{\Bb}) \to \PrL_{\Bb}$ preserves limits and colimits;
	\item The relative tensor product $- \otimes_{\Cc} -\colon \Mod_{\Cc}(\PrL_{\Bb}) \times \Mod_{\Cc}(\PrL_{\Bb}) \to \Mod_{\Cc}(\PrL_{\Bb})$ preserves colimits in both variables.
\end{enumerate}

Since $\PrL(\Bb)$ is closed symmetric monoidal, so is $\Mod_{\Cc}(\PrL(\Bb))$: for every pair of $\Cc$-linear $\Bb$-categories $\Dd$ and $\Ee$, there exists an internal hom object $\ul{\Fun}_{\Cc}(\Dd,\Ee)$. We let $\Fun_{\Cc}(\Dd,\Ee)$ denote its underlying $\infty$-category, whose objects are the $\Cc$-linear $\Bb$-functors $F\colon \Dd \to \Ee$, and whose morphisms will be referred to as \textit{$\Cc$-linear natural transformations}. Given two $\Cc$-linear $\Bb$-functors $F, F'\colon \Dd \to \Ee$, we let $\Nat_{\Cc}(F,F')$ denote the mapping space from $F$ to $F'$ in $\Fun_{\Cc}(\Dd,\Ee)$.

\begin{definition}
	\label{def:Internal_Left_Adjoint}
	A $\Cc$-linear $\Bb$-functor $F\colon \Dd \to \Ee$ is called an \textit{internal left adjoint in $\Mod_{\Cc}(\PrL(\Bb))$} if there is a $\Cc$-linear $\Bb$-functor $G\colon \Ee \to \Dd$ equipped with $\Cc$-linear transformations $\epsilon\colon FG \Rightarrow \id$ and $\eta\colon \id \Rightarrow GF$ satisfying the triangle identities.
\end{definition}

We have the following characterization of internal left adjoints:

\begin{lemma}
	\label{lem:Criterion_Internal_Left_Adjoint}
	A $\Cc$-linear $\Bb$-functor $F\colon \Dd \to \Ee$ is an internal left adjoint in $\Mod_{\Cc}(\PrL(\Bb))$ if and only if its right adjoint $G\colon \Dd \to \Ee$ in $\Cat(\Bb)$ preserves colimits and satisfies the projection formula: for every $B \in \Bb$, every $C \in \Cc(B)$ and every $E \in \Ee(B)$, the canonical map $C \otimes_B G(E) \to G(C \otimes_B E)$ is an equivalence in $\Dd(B)$.
\end{lemma}
\begin{proof}
	By \Cref{prop:RightAdjointCLinear}, the right adjoint $G$ of $G$ is a right adjoint internal to $\Mod_{\Cc}(\Cat(\Bb))$ if and only if $G$ satisfies the projection formula. The adjunction lifts further to $\Mod_{\Cc}(\PrL(\Bb))$ if and only if $G$ preserves colimits.
\end{proof}

The property for a $\Cc$-linear functor $F$ to be an internal left adjoint can be checked locally in $\Bb$:

\begin{proposition}
	\label{prop:Test_After_Pullback_NaturalEquivalence}
	Let $F\colon \Dd \to \Ee$ be a $\Cc$-linear $\Bb$-functor and assume that $B \twoheadrightarrow 1$ is an effective epimorphism in $\Bb$. Then $F$ is an internal left adjoint in $\Mod_{\Cc}(\PrL(\Bb))$ if and only if $\pi_B^*F$ is an internal left adjoint in $\Mod_{\Cc}(\xPrLB{B})$.
\end{proposition}
\begin{proof}
	Let $G\colon \Ee \to \Dd$ be the right adjoint of $F$. Since $\PrL_{\Bb}$ is a parametrized subcategory of $\Cat_{\Bb}$ and thus satisfies descent w.r.t.\ effective epimorphisms, $G$ preserves colimits if and only if $\pi_B^*G$ does. By \Cref{lem:Criterion_Internal_Left_Adjoint}, it thus remains to show that $G$ satisfies the projection formula if and only if $\pi_B^*G$ does. This is true because checking that for objects $A \in \Bb$, $C \in \Cc(A)$ and $E \in \Ee(A)$ the map $C \otimes_B G(E) \to G(C \otimes_B E)$ is an equivalence in $\Dd(A)$ can be done after pulling back along the effective epimorphism $A \times B \twoheadrightarrow A$.
\end{proof}

\subsubsection{Free and cofree \texorpdfstring{$\Cc$}{C}-linear \texorpdfstring{$\Bb$}{B}-categories}
Continue to fix a presentably symmetric monoidal $\Bb$-category $\Cc$. We will next associate to every object $A \in \Bb$ two $\Cc$-linear $\Bb$-categories $\Cc[A]$ and $\Cc^A$, the \textit{free} and \textit{cofree $\Cc$-linear $\Bb$-categories on $A$}, and prove that they are in fact equivalent, see \Cref{cor:Free_VS_Cofree_CLinear_Category} below. The arguments are entirely analogous to those of \cite[Section~4.1]{CCRY_Characters}.

\begin{definition}
	\label{def:Free_CLinear_BCategory}
	Let $\Dd \in \Mod_{\Cc}(\PrL(\Bb))$ be a $\Cc$-linear $\Bb$-category. For an object $A \in \Bb$, we define the $\Cc$-linear $\Bb$-categories $\Dd[A]$ and $\Dd^A$ by
	\[
	\Dd[A] := \colim_A \Dd = (\pi_A)_!\pi_A^*\Dd \qquad \qquad \text{ and } \qquad \qquad \Dd^A := \lim_A \Dd = (\pi_A)_*\pi_A^*\Dd,
	\]
	the $A$-indexed colimit and limit of the constant diagram on $\Dd$ in the $\Bb$-category $\Mod_{\Cc}(\PrL_{\Bb})$. For a morphism $f\colon A \to B$ in $\Bb$, we denote by
	\[
	f_!\colon \Dd[A] \to \Dd[B] \qquad \qquad \text{ and } \qquad \qquad f^*\colon \Dd^B \to \Dd^A
	\]
	the induced $\Cc$-linear $\Bb$-functors, obtained by plugging the appropriate unit or counit. Their right adjoints in $\Cat(\Bb)$ are denoted by
	\[
	f^*\colon \Dd[B] \to \Dd[A] \qquad \qquad \text{ and } \qquad \qquad f_*\colon \Dd^A \to \Dd^B.
	\]
\end{definition}

Since the forgetful functors $\Mod_{\Cc}(\PrL_{\Bb}) \to \PrL_{\Bb} \to \Cat_{\Bb}$ preserve limits by \cite[Proposition~7.2.7]{martiniwolf2022presentable} and \cite[Proposition 6.4.9]{martiniwolf2022presentable}, the underlying $\Bb$-category of $\Dd^A$ is the $\Bb$-category $(\pi_A)_*\pi_A^*\Dd = \Dd(A \times -)\colon \Bb\catop \to \Cat_{\infty}$, alternatively described as the parametrized functor category $\ul{\Fun}_{\Bb}(A,\Dd)$.

\begin{corollary}
	\label{cor:MappingOutOfFreeCModule}
	There is a natural equivalence of $\Cc$-linear $\Bb$-categories
	\[
	\ul{\Fun}_{\Cc}(\Cc[A],\Dd) \iso \Dd^A.
	\]
\end{corollary}
\begin{proof}
	The internal hom $\Bb$-functor $\ul{\Fun}_{\Cc}(-,\Dd)\colon \Mod_{\Cc}(\PrL_{\Bb})\catop \to \Mod_{\Cc}(\PrL_{\Bb})$ turns colimits in $\Mod_{\Cc}(\PrL_{\Bb})$ into limits and returns $\Dd$ when evaluated at the monoidal unit $\Cc \in \Mod_{\Cc}(\PrL(\Bb))$, hence we have
	\[
	\ul{\Fun}_{\Cc}(\Cc[A],\Dd) = \ul{\Fun}_{\Cc}(\colim_A \Cc,\Dd) \simeq \lim_A \ul{\Fun}_{\Cc}(\Cc,\Dd) \simeq \lim_A \Dd = \Dd^A. \qedhere
	\]
\end{proof}

\begin{lemma}
	For objects $A, B \in \Bb$, there is an equivalence
	\[
	\Cc[A] \otimes_{\Cc} \Dd[B] \simeq \Dd[A \times B].
	\]
\end{lemma}
\begin{proof}
	Since $- \otimes_{\Cc} -$ preserves colimits in both variables and $\Cc \otimes_{\Cc} \Dd \simeq \Dd$, this follows from the observation that $\colim_A \colim_B \simeq \colim_{A \times B}$: both sides are left adjoint to the diagonal $\Mod_{\Cc}(\PrL_{\Bb}) \to \ul{\Fun}_{\Bb}(A \times B, \Mod_{\Cc}(\PrL_{\Bb}))$.
\end{proof}

\begin{proposition}
	\label{prop:FreeCofreeCLinearCategories} 
	Let $\Dd$ be a $\Cc$-linear $\Bb$-category.
	\begin{enumerate}[(1)]
		\item For every morphism $f\colon A \to B$ in $\Bb$, the $\Cc$-linear $\Bb$-functor $f_!\colon \Cc[A] \to \Cc[B]$ admits a $\Cc$-linear right adjoint $f^*\colon \Cc[B] \to \Cc[A]$;
		\item For every pullback square in $\Bb$ as on the left, the induced square of $\Cc$-linear $\Bb$-functors as on the right is (vertically) right adjointable in $\Mod_{\Cc}(\PrL(\Bb))$:
		\[
		\begin{tikzcd}
			A' \rar{\alpha} \dar[swap]{f'} \drar[pullback] & A \dar{f} \\
			B' \rar{\beta} & B
		\end{tikzcd}
		\qquad \qquad \qquad
		\begin{tikzcd}
			\Cc[A'] \rar{\alpha_!} \dar[swap]{f'_!} & \Cc[A] \dar{f_!} \\
			\Cc[B'] \rar{\beta_!} & \Cc[B];
		\end{tikzcd}
		\]
		\item For every morphism $f\colon A \to B$ in $\Bb$, the $\Cc$-linear $\Bb$-functor $f^*\colon \Dd^B \to \Dd^A$ admits a $\Cc$-linear left adjoint $f_!\colon \Dd^A \to \Dd^B$;
		\item For every pullback square in $\Bb$ as on the left, the induced square of $\Cc$-linear $\Bb$-functors as on the right is (vertically) left adjointable in $\Mod_{\Cc}(\PrL(\Bb))$:
		\[
		\begin{tikzcd}
			A' \rar{\alpha} \dar[swap]{f'} \drar[pullback] & A \dar{f} \\
			B' \rar{\beta} & B
		\end{tikzcd}
		\qquad \qquad \qquad
		\begin{tikzcd}
			{\Dd^{A'}} & {\Dd^A} \lar[swap]{\alpha^*} \\
			{\Dd^{B'}} \uar{{f'}^*} & {\Dd^B} \uar[swap]{f^*} \lar[swap]{\beta^*};
		\end{tikzcd}
		\]
		\item \label{it:RestrictionInduction} For every morphism $f\colon A \to B$ in $\Bb$, there are naturally commutative squares of $\Cc$-linear $\Bb$-functors
		\[
		\begin{tikzcd}
			\ul{\Fun}_{\Cc}(\Cc[B],\Dd) \dar[swap]{- \circ f_!} \rar{\simeq} & 
			{\Dd^B} \dar{f^*} \\
			\ul{\Fun}_{\Cc}(\Cc[A],\Dd) \rar{\simeq} & {\Dd^A}
		\end{tikzcd}
		\qquad \text{and} \qquad
		\begin{tikzcd}
			\ul{\Fun}_{\Cc}(\Cc[A],\Dd) \dar[swap]{- \circ f^*} \rar{\simeq} & {\Dd^A} \dar{f_!} \\
			\ul{\Fun}_{\Cc}(\Cc[B],\Dd) \rar{\simeq} & {\Dd^B}.
		\end{tikzcd}
		\]
	\end{enumerate}
\end{proposition}
\begin{proof}
	Parts (3) and (4) follow immediately from (1) and (2) because of the natural equivalence $\ul{\Fun}_{\Cc}(\Cc[A],\Dd) \simeq \Dd^A$. For (5), the left-hand square is simply naturality in $A$ of the equivalence $\ul{\Fun}_{\Cc}(\Cc[A],\Dd) \simeq \Dd^A$, and the right-hand square is obtained from this by passing to internal left adjoints in $\Mod_{\Cc}(\PrL(\Bb))$. For parts (1) and (2), consider the base change functor $\Cc \otimes_{\Omega_{\Bb}} -\colon \PrL_{\Bb} \simeq \Mod_{\Omega_{\Bb}}(\PrL_{\Bb}) \to \Mod_{\Cc}(\PrL_{\Bb})$. It is a colimit-preserving symmetric monoidal 2-functor, and thus we have $\Cc[A] \simeq \Cc \otimes_{\Omega_{\Bb}} \Omega_{\Bb}[A]$. It thus suffices to show the claim when $\Cc = \Omega_{\Bb}$. In this case we have an equivalence $\Omega_{\Bb}[A] \simeq \PSh_{\Omega}(A)$ by \cite[Theorem~6.1.1]{martiniwolf2021limits}, and by \cite[Lemma~6.1.3]{martiniwolf2021limits} the functor $f_!\colon \PSh_{\Omega}(A) \to \PSh_{\Omega}(B)$ is given by left Kan extension. The right adjoint $f^*\colon \PSh_{\Omega}(B) \to \PSh_{\Omega}(A)$ of $f_!$ is therefore just the restriction functor, which is a morphism in $\PrL(\Bb)$, proving (1). For part (2), it suffices to check that for every object $C \in \Bb$, evaluating the square at $C$ gives a vertically right adjointable square of $\infty$-categories. In the case at hand, this square looks like
	\[
	\begin{tikzcd}
		\Bb_{/A' \times C} \rar{(\alpha \times 1) \circ -} \dar[swap]{(f' \times 1) \circ -} & \Bb_{/A \times C} \dar{(f \times 1) \circ -} \\
		\Bb_{/B' \times C} \rar{(\beta \times 1) \circ -} & \Bb_{/B \times C},
	\end{tikzcd}
	\]
	which is adjointable as $\Bb$ satisfies base change.
\end{proof}

\begin{corollary}
	\label{cor:FreeCLinearCategoryIsDualizable}
	For every object $A \in \Bb$, the $\Cc$-linear pairing
	\[
	\Cc[A] \otimes_{\Cc} \Cc[A] \simeq \Cc[A\times A] \xrightarrow{\Delta^*} \Cc[A] \xrightarrow{A_!} \Cc
	\]
	is part of a duality datum in $\Mod_{\Cc}(\PrL(\Bb))$, exhibiting the $\Cc$-linear $\Bb$-category $\Cc[A]$ as self-dual.
\end{corollary}
\begin{proof}
	The coevaluation is given by $\Cc \xrightarrow{A^*} \Cc[A] \xrightarrow{\Delta_!} \Cc[A \times A] \simeq \Cc[A] \otimes_{\Cc} \Cc[A]$. The first triangle identity follows from the following commutative diagram:
	\begin{equation*}
		\begin{tikzcd}
			& & {\Cc[A \times A \times A]} \drar{(\Delta \times 1)^*} \\
			& {\Cc[A \times A]} \drar{\Delta^*} \urar{(1 \times \Delta)_!} & & {\Cc[A \times A]} \drar{(\pr_2)_!} \\
			{\Cc[A]} \ar[equal]{rr} \urar{\pr_1^*} && {\Cc[A]} \urar{\Delta_!} \ar[equal]{rr} & & {\Cc[A]},
		\end{tikzcd}
	\end{equation*}
	where the square commutes via the Beck-Chevalley equivalence. The other triangle identity is analogous.
\end{proof}

\begin{lemma}
	\label{lem:DualsOfBaseChangeFunctors}
	Let $f\colon A \to B$ be a morphism in $\Bb$. Then the following diagrams commute:
	\[
	\begin{tikzcd}
		\Cc[B] \dar[swap]{\simeq} \ar{rr}{f^*} && \Cc[A] \dar{\simeq} \\
		\Cc[B]^{\vee} \ar{rr}{(f_!)^{\vee}} && \Cc[A]^{\vee}
	\end{tikzcd}
	\qquad
	\text{ and }
	\qquad
	\begin{tikzcd}
		\Cc[A] \dar[swap]{\simeq} \ar{rr}{f_!} && \Cc[B] \dar{\simeq} \\
		\Cc[A]^{\vee} \ar{rr}{(f^*)^{\vee}} && \Cc[B]^{\vee}.
	\end{tikzcd}
	\]
\end{lemma}
\begin{proof}
	We will prove the left diagram. The proof for the right diagram is analogous and is left to the reader. Expanding the definition of $(f_!)^{\vee}$ by plugging in the explicit evaluation and coevaluation maps from \Cref{cor:FreeCLinearCategoryIsDualizable}, we see it is given by the composite
	\[
	\Cc[B] \xrightarrow{(\pr_B)^*} \Cc[B \times A] \xrightarrow{(1 \times (f,1))_!} \Cc[B \times B \times A] \xrightarrow{(\Delta \times 1)^*} \Cc[B \times A] \xrightarrow{(\pr_A)_!} \Cc[A].
	\]
	Observe that the maps $1 \times (f,1)$ and $\Delta \times 1$ fit into a pullback square
	\[\begin{tikzcd}
		A \ar[pullback]{drr} && {B \times A} \\
		{B \times A} && {B \times B \times A,}
		\arrow["{1 \times (f,1)}", from=1-3, to=2-3]
		\arrow["{\Delta \times 1}", from=2-1, to=2-3]
		\arrow["{(f,1)}"', from=1-1, to=2-1]
		\arrow["{(f,1)}", from=1-1, to=1-3]
	\end{tikzcd}\]
	and thus it follows from base change that the above composite is homotopic to
	\[
	\Cc[B] \xrightarrow{(\pr_B)^*} \Cc[B \times A] \xrightarrow{(f,1)^*} \Cc[A] \xrightarrow{(f,1)_!} \Cc[B \times A] \xrightarrow{(\pr_A)_!} \Cc[A].
	\]
	But this composite is the functor $f^*\colon \Cc[B] \to \Cc[A]$, as desired.
\end{proof}

Under the equivalence $\Fun_{\Cc}(\Cc[A],\Cc^A) \simeq \Cc^A(A) = \Cc(A \times A)$, the object $\Delta_!\unit_A$ corresponds to a $\Cc$-linear $\Bb$-functor $\Cc[A] \to \Cc^A$.

\begin{corollary}
	\label{cor:Free_VS_Cofree_CLinear_Category}
	For every object $A \in \Bb$, the $\Bb$-functor $\Cc[A] \to \Cc^A$ is an equivalence of $\Cc$-linear $\Bb$-categories. 
	Furthermore, for every map $f\colon A \to B$ in $\Bb$, the following diagrams commute:
	\[
	\begin{tikzcd}
		\Cc[B] \rar{\simeq} \dar[swap]{f^*} & {\Cc^B} \dar{f^*} \\
		\Cc[A] \rar{\simeq} & {\Cc^A}
	\end{tikzcd}
	\qquad
	\text{ and }
	\qquad
	\begin{tikzcd}
		\Cc[A] \rar{\simeq} \dar[swap]{f_!} & {\Cc^A} \dar{f_!} \\
		\Cc[B] \rar{\simeq} & {\Cc^B}.
	\end{tikzcd}
	\]
\end{corollary}
\begin{proof}
	By \Cref{cor:FreeCLinearCategoryIsDualizable}, the pairing $\Cc[A] \otimes_{\Cc} \Cc[A] \to \Cc$ adjoints over to an equivalence $\Cc[A] \iso \ul{\Fun}_{\Cc}(\Cc[A],\Cc)$. Composing this with the equivalence $\ul{\Fun}_{\Cc}(\Cc[A],\Cc) \simeq \Cc^A$ from \Cref{cor:MappingOutOfFreeCModule} gives an equivalence $\Cc[A] \xrightarrow{\sim} \Cc^A$. It remains to show that this is the $\Bb$-functor $\Cc[A] \to \Cc^A$ classified by $\Delta_!\unit_A \in \Cc^A(A)$. Adjoining over once more, it suffices to show that the dual $\Cc \to \Cc^{A \times A}$ of the pairing $\Cc[A \times A] \to \Cc$ is classified by the object $\Delta_!\unit_A \in \Cc(A \times A)$. But by \Cref{prop:FreeCofreeCLinearCategories}\eqref{it:RestrictionInduction} and the construction of the pairing, this is the composite $\Cc \xrightarrow{A^*} \Cc^A \xrightarrow{\Delta_!} \Cc^{A \times A}$. This proves the first claim.
	
	The two commutative diagrams follow from a combination of \Cref{lem:DualsOfBaseChangeFunctors} and \Cref{prop:FreeCofreeCLinearCategories}\eqref{it:RestrictionInduction}.
\end{proof}

\subsubsection{Classification of \texorpdfstring{$\Cc$}{C}-linear functors}
As a result of \Cref{cor:Free_VS_Cofree_CLinear_Category}, the $\Cc$-linear $\Bb$-category $\Cc^A$ is, in a precise sense, the free $\Cc$-linear $\Bb$-category on $A$: any $\Cc$-linear $\Bb$-functor $\Cc^A \to \Dd$ to another $\Cc$-linear $\Bb$-category $\Dd$ is fully determined by where it sends the object $\Delta_!\unit_A \in \Cc(A \times A) = \Cc^A(A)$.

\begin{corollary}
	\label{cor:Classification_CLinear_Functors}
	For an object $A \in \Bb$ and a $\Cc$-linear $\Bb$-category $\Dd$, the composite
	\[
	\ul{\Fun}_{\Cc}(\Cc^A,\Dd) \xrightarrow{(-)^A} \ul{\Fun}_{\Cc}(\Cc^{A\times A},\Dd^A) \xrightarrow{\ev_{\Delta_!\unit_A}} \Dd^A
	\]
	is an equivalence of $\Cc$-linear $\Bb$-categories, where the last map denotes evaluation at $\Delta_! \unit_A \in \Cc(A \times A)$.
\end{corollary}
\begin{proof}
	Let $U \in \Cc[A](A)$ denote the object classified by the identity $\Cc[A] \to \Cc[A]$ under the equivalence $\Fun_{\Cc}(\Cc[A],\Cc[A]) \simeq \Cc[A](A)$. By \Cref{cor:Free_VS_Cofree_CLinear_Category}, it suffices to show the statement for $\Cc[A]$ instead of $\Cc^A$: the composite
	\[
	\ul{\Fun}_{\Cc}(\Cc[A],\Dd) \xrightarrow{(-)^A} \ul{\Fun}_{\Cc}(\Cc[A]^A,\Dd^A) \xrightarrow{\ev_{U}} \Dd^A
	\]
	is an equivalence of $\Cc$-linear $\Bb$-categories. We claim that this composite is simply the equivalence $\ul{\Fun}_{\Cc}(\Cc[A],\Dd) \xrightarrow{\sim} \Dd^A$ from \Cref{cor:MappingOutOfFreeCModule}. Indeed, as both are functorial in $\Dd$ it suffices to show they agree for the identity functor on $\Dd = \Cc[A]$, in which case it holds by definition of the object $U$.
\end{proof}

\begin{observation}
	Since the functor $(-)^A\colon \PrL(\Bb) \to \PrL(\Bb)$ is lax symmetric monoidal (as it is the composite of the symmetric monoidal functor $\pi_A^*\colon \PrL(\Bb) \to \xPrLB{A}$ and its right adjoint $(\pi_A)_*$), the $\Bb$-category $\Dd^A$ is canonically tensored over $\Cc^A$. We claim that the inverse of the equivalence of \Cref{cor:Classification_CLinear_Functors} is adjoint to the composite $\Dd^A \otimes_{\Cc} \Cc^A \xrightarrow{- \otimes_A -} \Dd^A \xrightarrow{A_!} \Dd$. To see this, it suffices to show that the composite $\Dd^A \to \ul{\Fun}_{\Cc}(\Cc^A,\Dd) \xrightarrow{\sim} \Dd^A$ is equivalent to the identity. Expanding definitions, this composite is given by
	\[
	\Dd^A \xrightarrow{\pr_1^*} \Dd^{A \times A} \simeq \Dd^{A \times A} \otimes \Cc \xrightarrow{1 \otimes \Delta_!\unit_A} \Dd^{A \times A} \otimes_{\Cc} \Cc^{A \times A} \xrightarrow{-\otimes_{A \times A}-} \Dd^{A \times A} \xrightarrow{{\pr_2}_!} \Dd^A.
	\]
	By the projection formula, this is equivalent to the composite
	\[
		{\pr_2}_!\Delta_!(\Delta^*\pr_1^*(-) \otimes_A \unit_A)\colon \Dd^A \to \Dd^A,
	\]
	which is equivalent to the identity since $\Delta\colon A \to A \times A$ is a section of both projections $\pr_1,\pr_2\colon A \times A \to A$.
\end{observation}

Specializing to $\Dd = \Cc^B$, we arrive at the main result of this subsection:

\begin{theorem}
	\label{thm:Classification_CLinear_Functors}
	For objects $A, B \in \Bb$, evaluation at $\Delta_!\unit_A \in \Cc^A(A)$ induces an equivalence of $\Cc$-linear $\Bb$-categories
	\[
	\ul{\Fun}_{\Cc}(\Cc^A,\Cc^{B}) \iso \Cc^{A \times B}. 
	\]
	The inverse is adjoint to the composite
	\[
	\Cc^{A \times B} \otimes \Cc^A \xrightarrow{1 \otimes \pr_A^*} \Cc^{A \times B} \otimes \Cc^{A \times B} \xrightarrow{- \otimes_{A \times B} -} \Cc^{A \times B} \xrightarrow{{\pr_B}_!} \Cc^B. \qednow
	\]
\end{theorem}

\begin{definition}
	\label{def:ClassifiedFunctor}
	Let $\Dd \in \Mod_{\Cc}(\PrL(\Bb))$ be a $\Cc$-linear $\Bb$-category and let $F\colon \Cc^A \to \Dd$ be a $\Cc$-linear $\Bb$-functor. We define $D_F := F_A(\Delta_!\unit_A) \in \Dd(A)$. Conversely, for an object $D \in \Dd(A)$, we let $F_D\colon \Cc^A \to \Dd$ denote the $\Cc$-linear $\Bb$-functor $F_D = A_!(- \otimes_A D)\colon \Cc^A \to \Dd$.
\end{definition}

More informally, \Cref{cor:Classification_CLinear_Functors} says that every $\Cc$-linear $\Bb$-functor $F\colon \Cc^A \to \Dd$ is naturally equivalent to the $\Bb$-functor $A_!(- \otimes_A D_F)$. One should be a bit careful with the interpretation of these symbols ``$A_!$'' and ``$\otimes_{A}$", as they depend on $\Dd$. For example, in the case of \Cref{thm:Classification_CLinear_Functors}, where $\Dd = \Cc^B$ for some $B \in \Bb$, the statement is that every $\Cc$-linear functor $F\colon \Cc^A \to \Cc^{B}$ is equivalent to ${\pr_{B}}_!(\pr_A^*(-) \otimes_{A \times B} D_F)$ for a unique $D_F \in \Cc^{B}(A) = \Cc(A \times B)$.

\begin{corollary}
	\label{cor:ClassifyingCLinearNaturalTransformations}
	Let $A \in \Bb$, let $\Dd$ be a $\Cc$-linear $\Bb$-category and let $F, G\colon \Cc^A \to \Dd$ be two $\Cc$-linear functors. Evaluation at $\Delta_!\unit_A$ induces an equivalence of spaces
	\begin{align*}
		\Nat_{\Cc}(F,G) \iso \Hom_{\Dd(A)}(D_F,D_G).
	\end{align*}
\end{corollary}
\begin{proof}
	This follows immediately from the fact that evaluation at $\Delta_!\unit_A$ is an equivalence from $\Fun_{\Cc}(\Cc^A,\Dd)$ to $\Dd(A)$ by \Cref{cor:Classification_CLinear_Functors}, so that it in particular induces equivalences on mapping spaces.
\end{proof}

\subsection{Formal inversions}
\label{subsec:Formal_Inversion}

If $\Cc$ is a presentably symmetric monoidal $\infty$-category and $S$ is a small subcategory of $\Cc$, there is another presentably symmetric monoidal $\infty$-category $\Cc[S^{-1}]$ equipped with the universal morphism $\Cc \to \Cc[S^{-1}]$ in $\CAlg(\PrL)$ out of $\Cc$ which sends all objects in $S$ to invertible objects, see Robalo \cite[Section~2.1]{robalo2015ktheory} and Hoyois \cite[Section~6.1]{hoyois2017sixoperations}. The goal of this subsection is to discuss a parametrized variant of this construction.

\begin{definition}[Formal inversion]
	\label{def:FormalInversion}
	Let $\Cc$ be a presentably symmetric monoidal $\Bb$-category, and let $S \subseteq \Cc$ be a full subcategory. A morphism $p\colon \Cc \to \Cc'$ in $\CAlg(\PrL(\Bb))$ is said to \textit{exhibit $\Cc'$ as a formal inversion of $S$ in $\Cc$} if for every other $\Dd \in \CAlg(\PrL(\Bb))$ the map of spaces
	\[
	\Hom_{\CAlg(\PrL(\Bb))}(\Cc',\Dd) \to \Hom_{\CAlg(\PrL(\Bb))}(\Cc,\Dd)
	\]
	given by precomposition with $p$ is a monomorphism hitting those path components corresponding to $\Bb$-functors $F\colon \Cc \to \Dd$ which invert the objects of $S$: for every $A \in \Bb$, the functor $F_A\colon \Cc(A) \to \Dd(A)$ sends every object $X \in S(A)$ to an invertible object in $\Dd(A)$.
\end{definition}

If such formal inversion $p\colon \Cc \to \Cc'$ exists, it is uniquely determined by this universal property and we will denote it by $p\colon \Cc \to \Cc[S^{-1}]$. It is not clear to the author in what generality parametrized formal inversions can be expected to exist. In this article, we will contend ourselves with some specific situations in which the formal inversion can be shown to exist. 

A first such situation is when the subcategory $S$ is generated by a (non-parametrized) full subcategory $S_0 \subseteq \Cc(1)$, in the following sense:

\begin{definition}
	\label{def:SubcategoryGenerated}
	Let $\Cc$ be a $\Bb$-category and let $S \subseteq \Cc$ be a subcategory. Given a small subcategory $S_0 \subseteq \Gamma(S) \subseteq \Gamma(\Cc)$, we will say that \textit{$S$ is generated by $S_0$} if the following condition is satisfied: for every $X \in S(B)$, there exists an effective epimorphism $f\colon A \to B$ in $\Bb$ and an object $Y \in S_0 \subseteq \Cc(1) = \Gamma(\Cc)$ such that $f^*X \simeq A^*Y$.
	
	Note that in this case, a morphism $F\colon \Cc \to \Dd$ in $\CAlg(\PrL(\Bb))$ inverts the objects of $S$ if and only if the underlying functor $\Gamma(F)\colon \Gamma(\Cc) \to \Gamma(\Dd)$ inverts the objects of $S_0$.
\end{definition}

\begin{definition}
	\label{def:InvertGlobalObjects}
	Let $\Cc \in \CAlg(\PrL(\Bb))$ and let $S \subseteq \Cc$ be a symmetric monoidal subcategory which is generated by a small subcategory $S_0 \subseteq \Gamma(\Cc)$. Let $\Gamma(\Cc)[S_0^{-1}]$ denote the (non-parametrized) formal inversion of $S_0$ in $\Gamma(\Cc)$. We define the commutative $\Cc$-algebra $\Cc[S_0^{-1}]$ in $\PrL(\Bb)$ as the image of $\Gamma(\Cc)[S_0^{-1}]$ under the adjunction
	\[
	- \otimes_{\Gamma(\Cc)} \Cc \colon \CAlg_{\Gamma(\Cc)}(\PrL) \rightleftarrows \CAlg_{\Cc}(\PrL(\Bb))\noloc \Gamma^{\Cc}
	\]
	from \Cref{lem:AdjunctionCLinearCategories}.
\end{definition}

\begin{proposition}
	\label{prop:InvertingObjectsGloballySuffices}
	Let $\Cc \in \CAlg(\PrL(\Bb))$ and let $S \subseteq \Cc$ be a symmetric monoidal subcategory which is generated by $S_0 \subseteq \Gamma(\Cc)$. Then the symmetric monoidal left adjoint $\Cc \to \Cc[S_0^{-1}]$ is a formal inversion of $S$ in $\Cc$.
\end{proposition}
\begin{proof}
	Let $F\colon \Cc \to \Dd$ be morphism in $\CAlg(\PrL(\Bb))$. By the adjunction between $- \otimes_{\Gamma(\Cc)} \Cc$ and $\Gamma^{\Cc}$, there is a one-to-one correspondence between morphisms $\Cc[S_0^{-1}] \to \Dd$ of commutative $\Cc$-algebras and morphisms $\Gamma(\Cc)[S_0^{-1}] \to \Gamma(\Dd)$ of commutative $\Gamma(\Cc)$-algebras. The claim thus follows from the universal property of the formal inversion $\Gamma(\Cc) \to \Gamma(\Cc)[S_0^{-1}]$, combined with the fact that the $\Bb$-functor $F\colon \Cc \to \Dd$ inverts the objects of $S$ if and only if the underlying functor $\Gamma(F) \colon \Gamma(\Cc) \to \Gamma(\Dd)$ inverts the objects of $S_0$.
\end{proof}

\begin{observation}
	\label{obs:GloballyInvertingMeansPointwiseInverting}
	In the situation of \Cref{def:InvertGlobalObjects}, it follows directly from fully faithfulness of the functor $\Cc \otimes_{\Gamma(\Cc)}\colon \Mod_{\Gamma(\Cc)}(\PrL) \hookrightarrow \Mod_{\Cc}(\PrL(\Bb))$ that the underlying $\infty$-category of $\Cc[S_0^{-1}]$ is given by the non-parametrized formal inversion $\Gamma(\Cc)[S_0^{-1}]$ of $S_0$ in $\Gamma(\Cc)$. In particular, \Cref{prop:InvertingObjectsGloballySuffices} not only shows that the formal inversion of $S$ in $\Cc$ exists, but also that its underlying morphism in $\CAlg(\PrL)$ is a non-parametrized formal inversion.
	
	More generally, evaluating the map $\Cc \to \Cc[S_0^{-1}]$ at an object $A \in \Bb$ gives a symmetric monoidal left adjoint
	\[
	\Cc(A) \to \Cc[S_0^{-1}](A)
	\]
	which exhibits $\Cc[S_0^{-1}](A)$ as a non-parametrized formal inversion of $S(A)$ in $\Cc(A)$. Indeed, this follows immediately by applying the above observation to the slice category $\Bb_{/A}$, where we use that passing to slices preserves formal inversions by \Cref{prop:Invert_Objects_Locally} below.
\end{observation}

The upshot of \Cref{obs:GloballyInvertingMeansPointwiseInverting} is that the formal inversion in the above setting can be obtained as a \textit{pointwise} formal inversion, in the sense that $\Cc[S^{-1}]$ is given at $B \in \Bb$ as the non-parametrized formal inversion $\Cc(B)[S(B)^{-1}]$. This observation will lead to the construction of parametrized formal inversions in more general situations. We start by recalling the functoriality of the formal inversion construction in the non-parametrized setting.

\begin{definition}
	We define $\Cat_{\infty, \aug}$ as the full subcategory of $\Ar(\Cat_{\infty}) = \Fun(\Delta^1,\Cat_{\infty})$ spanned by the morphisms in $\Cat_{\infty}$ corresponding to a fully faithful functor $\iota_S\colon S \hookrightarrow \Cc$, where $S$ is a small $\infty$-category. The functor $\iota_S\colon S \hookrightarrow \Cc$ is called an \textit{augmentation} of the $\infty$-category $\Cc$ and the pair $(\Cc,\iota_S)$ is called an \textit{augmented $\infty$-category}. We will often abuse notation and write $(\Cc,S)$ for $(\Cc,\iota_S)$, identifying $S$ with its image in $\Cc$.
\end{definition}

Note that the forgetful functor $\Cat_{\infty, \aug} \to \Cat_{\infty}, (\Cc,S) \mapsto \Cc$ is faithful: for another augmented $\infty$-category $(\Dd,T)$, a morphism $(\Cc,S) \to (\Dd,T)$ in $\smash{\Cat_{\infty, \aug}}$ may be identified with a functor $\Cc \to \Dd$ which sends the full subcategory $S \subseteq \Cc$ into $T \subseteq \Dd$.

\begin{definition}
	\label{def:AugmentedPresSymCats}
	We define the $\infty$-category $\CAlg(\PrL)_{\aug}$ of \textit{augmented presentably symmetric monoidal $\infty$-categories} as the pullback
	\[
	\begin{tikzcd}
		\CAlg(\PrL)_{\aug} \rar \dar \drar[pullback] & \Cat_{\infty, \aug} \dar \\
		\CAlg(\PrL) \rar & \Cat_\infty.
	\end{tikzcd}
	\]
	Its objects are pairs $(\Cc,S)$ of a presentably symmetric monoidal $\infty$-category $\Cc$ equipped with an augmentation $\iota_S\colon S \hookrightarrow \Cc$.
	
	We define a section $(-)_{\inv}\colon\CAlg(\PrL) \to \CAlg(\PrL)_{\aug}$ of the forgetful functor by equipping a symmetric monoidal $\Cc$ with its collection of invertible objects.
\end{definition}

\begin{lemma}
	\label{lem:FormalInversionFunctor}
	The functor $(-)_{\inv}\colon\CAlg(\PrL) \to \CAlg(\PrL)_{\aug}$ admits a left adjoint
	\[
	\Ll\colon \CAlg(\PrL)_{\aug} \to \CAlg(\PrL)
	\]
	given on objects by sending a pair $(\Cc,S)$ to the formal inversion $\Cc[S^{-1}]$.
\end{lemma}
\begin{proof}
	It suffices to show that for every pair $(\Cc,S)$, the formal inversion $\Cc[S^{-1}]$ is a left adjoint object to $(\Cc,S)$ under the functor $(-)_{\inv}\colon\CAlg(\PrL) \to \CAlg(\PrL)_{\aug}$. This follows from the universal property of $\Cc[S^{-1}]$: for every $\Dd \in \CAlg(\PrL)$, precomposition with the functor $\Cc \to \Cc[S^{-1}]$ induces an inclusion of path components
	\[
	\Hom_{\CAlg(\PrL)}(\Cc[S^{-1}],\Dd) \hookrightarrow \Hom_{\CAlg(\PrL)}(\Cc,\Dd)
	\]
	whose image precisely consists of those functors $\Cc \to \Dd$ which send all objects of $S$ to invertible objects in $\Dd$, i.e.\ the space of maps in $\CAlg(\PrL)_{\aug}$ from $(\Cc,S)$ to $\Dd_{\inv}$.
\end{proof}

We will now construct the pointwise formal inversion $\Ll(\Cc,S)$ in case $\Bb = \PSh(T)$ is a presheaf topos.
\begin{construction}
	\label{cons:PointwiseFormalInversion}
	Let $T$ be a small $\infty$-category, let $\Bb = \PSh(T)$ and let $\Cc$ be a presentably symmetric monoidal $\Bb$-category equipped with a small full subcategory $S \subseteq \Cc$. We let $(\Cc,S)\colon T\catop \to \CAlg(\PrL)_{\aug}$ denote the lift of the functor $\Cc\colon T\catop \to \CAlg(\PrL)$ which equips the $\infty$-category $\Cc(B)$ with the augmentation $S(B)$ for every $B \in T$. Composing with the functor $\Ll$ gives rise to a new functor
	\[
	\Ll(\Cc,S)\colon T\catop \to \CAlg(\PrL),
	\]
	given on objects by $B \mapsto \Cc(B)[S(B)^{-1}]$. This uniquely extends to a limit-preserving functor $\Ll(\Cc,S)\colon \Bb\catop \to \CAlg(\PrL)$.
\end{construction}

We claim that under suitable conditions, $\Ll(\Cc,S)$ is a presentably symmetric monoidal $\Bb$-category modeling the parametrized formal inversion $\Cc[S^{-1}]$. To this end, we need some basic properties of (non-parametrized) formal inversions.

\begin{lemma}
	\label{lem:FormalInversionStableUnderFactors}
	Let $\Cc \in \CAlg(\PrL)$ be a presentably symmetric monoidal $\infty$-category and let $S \subseteq S'$ be two small subcategories of $\Cc$. Assume that for every object $X' \in S'$ there exists objects $X \in S$ and $Y \in \Cc$ such that $X \simeq X' \otimes Y$. Then any formal inversion $p\colon \Cc \to \Cc[S^{-1}]$ of $S$ in $\Cc$ also exhibits $\Cc[S^{-1}]$ as a formal inversion of $S'$ in $\Cc$.
\end{lemma}
\begin{proof}
	This is immediate from the observation that any symmetric monoidal left adjoint $F\colon \Cc \to \Dd$ which inverts $S$ must also invert $S'$: if $X \simeq X' \otimes Y$, we get $F(X) \simeq F(X') \otimes F(Y)$ and thus if $F(X)$ is invertible, so must be $F(X')$ and $F(Y)$.
\end{proof}

\begin{lemma}
	\label{lem:FormalInversionViaBaseChange}
	Let $F\colon \Cc \to \Dd$ be a morphism in $\CAlg(\PrL)$ and let $S \subseteq \Cc$ be a small subcategory. Then the canonical map
	\[
	\Dd \otimes_{\Cc} \Cc[S^{-1}] \to \Dd[F(S)^{-1}]
	\]
	obtained by adjunction from the $\Cc$-linear functor $\Cc[S^{-1}] \to \Dd[F(S)^{-1}]$ is an equivalence.
\end{lemma}
\begin{proof}
	It follows from spelling out the adjunctions and universal properties that both sides are $\Dd$-algebras in $\PrL$ which admit a (necessarily unique) symmetric monoidal $\Dd$-linear map into another $\Dd$-algebra $\Ee$ if and only if the unit map $\Dd \to \Ee$ carries the objects of $F(S)$ to invertible objects in $\Ee$. The claim thus follows from the Yoneda lemma.
\end{proof}

We are now ready to prove that in certain cases the pointwise formal inversion $\Ll(\Cc,S)$ represents the parametrized formal inversion $\Cc[S^{-1}]$.

\begin{proposition}
	\label{prop:PointwiseFormalInversalIsParametrizedLocalInversion}
	Let $\Bb = \PSh(T)$ be the $\infty$-topos of presheaves on some small $\infty$-category $T$. Let $\Cc$ be a presentably symmetric monoidal $\Bb$-category equipped with a full subcategory $S \subseteq \Cc$. Assume that the following property is satisfied:
	\begin{enumerate}
		\item[$(\ast)$] For every morphism $f\colon A \to B$ in $T \subseteq \Bb$ and every object $X \in S(A)$, there exists objects $Y \in S(B)$ and $Z \in \Cc(A)$ such that $f^*Y \simeq X \otimes Z \in \Cc(A)$.
	\end{enumerate}
	Then the functor $\Ll(\Cc,S)\colon \Bb\catop \to \CAlg(\PrL)$ is a presentably symmetric monoidal $\Bb$-category and the map $\Cc = \Ll(\Cc,\emptyset) \to \Ll(\Cc,S)$ exhibits $\Ll(\Cc,S)$ as a formal inversion of $S$ in $\Cc$.
\end{proposition}
\begin{proof}
	We first show that $\Ll(\Cc,S)$ is a presentably symmetric monoidal $\Bb$-category and that the map $\Cc \to \Ll(\Cc,S)$ preserves colimits. This may be proved by pulling back $\Cc$ to the slice topos $\PSh(T)_{/B} \simeq \PSh(T_{/B})$ for every $B \in T$, hence we may assume that $T$ admits a terminal object $B$. In this case, consider the full subcategory $S' \subseteq \Cc$ given at $A \in T$ by 
	\[
	S'(A) := \{f^*X \in \Cc(A) \mid X \in S(B) \},
	\]
	where $f\colon A \to B$ is the unique map to the terminal object. Note that $S' \subseteq S$, as $S$ is a parametrized subcategory of $\Cc$. It is immediate from the definition that $S'$ is generated by $S_0 := S(B)$, in the sense of \Cref{def:SubcategoryGenerated}, and thus \Cref{prop:InvertingObjectsGloballySuffices} provides a formal inversion $F\colon \Cc \to \Cc[S_0^{-1}]$ of $S'$ in $\Cc$. As in the proof of \Cref{lem:FormalInversionStableUnderFactors}, the assumption $(\ast)$ guarantees that $F$ even inverts all the objects in $S$, and it follows that the map $\Cc \to \Cc[S_0^{-1}]$ uniquely extends to a map
	\[
	\Ll(\Cc,S) \to \Cc[S_0^{-1}]
	\]
	in $\Fun(T\catop,\CAlg(\PrL))$. It follows from \Cref{obs:GloballyInvertingMeansPointwiseInverting} and \Cref{lem:FormalInversionStableUnderFactors} that this map is pointwise an equivalence and hence that it is an equivalence of symmetric monoidal $\Bb$-categories. Since $\Cc[S_0^{-1}]$ is presentably symmetric monoidal, it follows that so is $\Ll(\Cc,S)$. Similarly we deduce that the map $\Cc \to \Ll(\Cc,S)$ preserves colimits.
	
	We will now verify that $\Cc \to \Ll(\Cc,S)$ is indeed a formal inversion. Assume that $\Dd$ is a presentably symmetric monoidal $\Bb$-category. Because of the adjunction $(-)_{\inv} \dashv \Ll$, fiberwise-colimit-preserving symmetric monoidal $\Bb$-functors $G\colon \Ll(\Cc,S) \to \Dd$ correspond to fiberwise-colimit-preserving symmetric monoidal $\Bb$-functors $G'\colon \Cc \to \Dd$ inverting the objects in $S$. It thus remains to show that if $G'$ preserves all $\Bb$-colimits, then so does its extension $G$. This may be tested after passing to the slice topoi $\PSh(T)_{/B} \simeq \PSh(T_{/B})$ for $B \in T$, and since the construction of $\Ll(\Cc,S)$ commutes with passage to slice topoi we may assume that $T$ admits a terminal object. In this case, there is a symmetric monoidal equivalence $\Ll(\Cc,S) \simeq \Cc[S_0^{-1}]$, and since $\Cc[S_0^{-1}]$ is a formal inversion of $S$ in $\Cc$ it follows that $G'$ admits a unique colimit-preserving symmetric monoidal extension to $\Cc[S_0^{-1}] \to \Dd$. This in particular preserves fiberwise colimits, so must agree with $G$, finishing the proof.
\end{proof}

We finish the section by proving that parametrized formal inversions are preserved under passing to slice topoi.

\begin{proposition}
	\label{prop:Invert_Objects_Locally}
	Let $F\colon \Cc \to \Cc'$ be a morphism in $\CAlg(\PrL(\Bb))$ which exhibits $\Cc'$ as a formal inversion of a small subcategory $S$ in $\Cc$. Then for every object $B \in \Bb$, the induced $\Bb_{/B}$-functor $\pi_B^*F \colon \pi_B^*\Cc \to \pi_B^*\Cc'$ exhibits $\pi_B^*\Cc'$ as a formal inversion of $\pi_B^*S$ in $\pi_B^*\Cc$.
\end{proposition}
\begin{proof}
	Consider $\Dd \in \CAlg(\xPrLB{B})$. Because of the adjunction
	\[
	\pi_B^*\colon \PrL(\Bb) \rightleftarrows \xPrLB{B}\noloc (\pi_B)_*
	\]
	from \Cref{cor:Base_Change_Adjunction_Presentable} and symmetric monoidality of $\pi_B^*$, a morphism $F\colon \pi_B^*\Cc \to \Dd$ in $\CAlg(\xPrLB{B})$ is the same as a morphism $F'\colon \Cc \to (\pi_B)_*\Dd$ in $\CAlg(\PrL(\Bb))$. Since the unit $\Cc \to (\pi_B)_*\pi_B^*\Cc$ and the counit $\pi_B^*(\pi_B)_* \Dd \to \Dd$ are symmetric monoidal, it follows that $F$ inverts the objects of $\pi_B^*S$ if and only if $F'$ inverts the objects of $S$. Since the map
	\[
	\Hom_{\CAlg(\xPrLB{B})}(\pi_B^*\Cc', \Dd) \to \Hom_{\CAlg(\xPrLB{B})}(\pi_B^*\Cc, \Dd)
	\]
	given by precomposition with $\pi_B^*F$ corresponds under the adjunction to the map
	\[
	\Hom_{\CAlg(\PrL(\Bb))}(\Cc', (\pi_B)_*\Dd) \to \Hom_{\CAlg(\PrL(\Bb))}(\Cc, (\pi_B)_*\Dd)
	\]
	given by precomposition with $F$, we conclude that $\pi_B^*F$ exhibits $\pi_B^*\Cc'$ as a formal inversion of $\pi_B^*S$ in $\pi_B^*\Cc$.
\end{proof}

\section{Twisted ambidexterity}
\label{sec:twistedambidexterity}
Fix an $\infty$-topos $\Bb$ and a presentably symmetric monoidal $\Bb$-category $\Cc$. In \Cref{sec:TwistedNormMap}, we will associate to every morphism $f\colon A \to B$ in $\Bb$ a \textit{relative dualizing object} $D_f \in \Cc(A)$ together with a \textit{twisted norm map} $\Nm_f\colon f_!(- \otimes D_f) \Rightarrow f_*(-)$, which informally speaking exhibits $f_!(- \otimes D_f)$ as the universal parametrized $\Cc$-linear approximation to $f_*$. We will show in \Cref{subsec:SemiadditivityVsTwistedAmbidexterity} that when $f$ is $n$-truncated for some $n$, the twisted norm map reduces to the (untwisted) norm map $\Nm_f\colon f_! \Rightarrow f_*$ from Hopkins and Lurie \cite{hopkinsLurie2013ambidexterity} whenever the latter is defined, and use this to express ambidexterity in terms of twisted ambidexterity. In \Cref{subsec:Costenoble_Waner_Duality} we will explain the close relation between twisted ambidexterity and \textit{Costenoble-Waner duality}, a parametrized form of monoidal duality due to \cite{CostenobleWaner2016equivariant, maysigurdsson2006parametrized}.

\subsection{The twisted norm map}
\label{sec:TwistedNormMap}
To define the twisted norm map $\Nm_f$, we will first treat the case where the target $B$ of $f$ is the terminal object of $\Bb$. The case for arbitrary $B$ will be obtained by passing to the slice topos $\Bb_{/B}$.

\begin{definition}
	\label{def:Dualizing_Object}
	For an object $A \in \Bb$ we define the \textit{dualizing object} $D_A \in \Cc(A)$ of $A$ as
	\begin{align*}
		D_A := {\pr_1}_*\Delta_!\unit_A \in \Cc(A),
	\end{align*}
	where $\pr_1\colon A \times A \to A$ is the first projection and $\Delta\colon A \to A \times A$ is the diagonal of $A$. We let $c_A\colon \pr_1^*D_A \to \Delta_!\unit_A$ denote the counit.
\end{definition}

Under the equivalence $\Fun_{\Cc}(\Cc^A,\Cc) \simeq \Cc(A)$ of \Cref{thm:Classification_CLinear_Functors}, the object $D_A \in \Cc(A)$ corresponds to a $\Cc$-linear $\Bb$-functor $A_!(- \otimes D_A)\colon \Cc^A \to \Cc$. As a special case of \cref{cor:ClassifyingCLinearNaturalTransformations} we immediately obtain the following corollary:

\begin{corollary}
	For an object $A \in \Bb$, evaluation at $\Delta_!\unit_A$ induces an equivalence of spaces
	\begin{align}
		\label{eq:DefinitionTwistedNormMap}
		\Nat_{\Cc}(A^*A_!(- \otimes_A D_A),\id_{\Cc^A}) \iso \Hom_{\Cc(A \times A)}(\pr_1^*D_A,\Delta_!\unit_A).
	\end{align}
\end{corollary}
\begin{proof}
	By naturality in $\Dd$ of the equivalence $\ul{\Fun}_{\Cc}(\Cc^A,\Dd) \iso \Dd^A$ from \Cref{cor:Classification_CLinear_Functors}, evaluating $A^*A_!(- \otimes_A D_A)\colon \Cc^A \to \Cc^A$ at $\Delta_!\unit_A$ gives $\pr_1^*D_A$. The statement then follows from \cref{cor:ClassifyingCLinearNaturalTransformations}, since all three functors $A^*$, $A_!$ and $- \otimes_A D_A$ are canonically $\Cc$-linear.
\end{proof}

\begin{definition}
	\label{def:Twisted_Norm_Map}
	For $A \in \Bb$, we define a $\Cc$-linear $\Bb$-transformation
	\[
	\Nmadj_A\colon A^*A_!(- \otimes_A D_A) \implies \id_{\Cc^A}
	\]
	as the $\Cc$-linear $\Bb$-transformation corresponding under the equivalence \eqref{eq:DefinitionTwistedNormMap} to the counit map $\pr_1^*D_A = \pr_1^*{\pr_1}_*\Delta_!\unit_A \to \Delta_!\unit_A$. We will refer to $\Nmadj_A$ as the \textit{adjoint twisted norm map} for $A$. We define the \textit{twisted norm map}
	\[
	\Nm_{A}\colon A_!(- \otimes_{A} D_A) \implies A_*(-)
	\]
	as the $\Bb$-transformation between $\Bb$-functors $\Cc^A \to \Cc$ adjoint to $\Nmadj_A$. 
\end{definition}

\begin{definition}
	\label{def:TwistedAmbidexterity}
	An object $A \in \Bb$ is called \textit{twisted $\Cc$-ambidextrous} if the twisted norm map $\Nm_A$ is an equivalence.
\end{definition}

A relative version of twisted ambidexterity is obtained by passing to slices of $\Bb$:

\begin{definition}
	\label{def:TwistedAmbidexterityMorphisms}
	\label{def:Relative_Twisted_Norm_Map}
	For a morphism $f\colon A \to B$ in $\Bb$, define the \textit{relative dualizing object} $D_f \in \Cc(A)$ of $f$ as $D_f := {\pr_1}_*(\Delta_f)_!\unit_A \in \Cc(A)$, where $\pr_1\colon A \times_B A \to A$ is the first projection and $\Delta_f\colon A \to A \times_B A$ is the diagonal of $f$. We define the \textit{twisted norm map}
	\[
	\Nm_{f}\colon f_!(- \otimes_{A} D_f) \implies f_*(-)
	\]
	as the $\Bb_{/B}$-transformation between $\Bb_{/B}$-functors $(\pi_B^*\Cc)^A \to \pi_B^*\Cc$ as in \Cref{def:Twisted_Norm_Map}. We will say that $f$ is \textit{twisted $\Cc$-ambidextrous} if the transformation $\Nm_f$ is an equivalence.
\end{definition}

We may obtain an explicit formula for the transformation $\Nmadj_A$ as follows.

\begin{lemma}
	\label{lem:Explicit_Formula_Twisted_Norm}
	For $A \in \Bb$, the transformation $\Nmadj_A\colon A^*A_!(- \otimes_A D_A) \Rightarrow \id_{\Cc^A}$ is equivalent to the composite
	\begin{align*}
		A^*A_!(- \otimes_A D_A) &\overset{l.b.c.}{\simeq} {\pr_2}_!\pr_1^*(- \otimes_A D_A) \\
		&\simeq {\pr_2}_!(\pr_1^*(-) \otimes_{A\times A} \pr_1^*D_A)\\
		&\Rightarrow {\pr_2}_!(\pr_1^*(-) \otimes_{A\times A} \Delta_!\unit_A) \\
		&\overset{l.p.f.}{\simeq} {\pr_2}_!\Delta_!(\Delta^*\pr_1^*(-) \otimes_{A} \unit_A) \\
		&\simeq \id_{\Cc^A}.
	\end{align*}
	Here $l.b.c.$ denotes the left base change equivalence, $l.p.f.$ denotes the left projection formula equivalence, and the non-invertible arrow in the middle is induced by the counit $\pr_1^*D_A = \pr_1^*{\pr_1}_*\Delta_!\unit_A \to \Delta_!\unit_A$.
\end{lemma}
\begin{proof}
	By definition, the adjoint twisted norm map $\Nmadj_A$ is defined to be the map whose image under the equivalence $\Fun_{\Cc}(\Cc^A, \Cc^A) \iso \Cc(A \times A)$ from \Cref{thm:Classification_CLinear_Functors} is the counit map $\pr_1^*D_A \to \Delta_!\unit_A$. Recall that an inverse to this equivalence is given by sending $D \in \Cc(A \times A)$ to the $\Cc$-linear $\Bb$-functor ${\pr_2}_!(\pr_1^*(-) \otimes_{A \times A} D)\colon \Cc^A \to \Cc^A$. It follows that morphism $\Nmadj_A$ in $\Fun_{\Cc}(\Cc^A,\Cc^A)$ is equivalent to the map ${\pr_2}_!(\pr_1^*(-) \otimes_{A \times A} \pr_1^*D_A) \to {\pr_2}_!(\pr_1^*(-) \otimes_{A \times A} \Delta_!\unit)A)$ induced by the counit. Unwinding definitions, the identifications of the source and target of these maps happens through the left base change equivalence and the left projection formula, giving the claim.
\end{proof}

The twisted norm map $\Nm_A$ exhibits the $\Bb$-functor $A_!(- \otimes D_A)\colon \Cc^A \to \Cc$ in a suitable sense as the universal $\Cc$-linear approximation of the $\Bb$-functor $A_*\colon \Cc^A \to \Cc$. More precisely, the dual adjoint norm map $\Nmadj_A$ exhibits $A_!(- \otimes D_A)$ as terminal among $\Cc$-linear $\Bb$-functor $F\colon \Cc^A \to \Cc$ equipped with a $\Cc$-linear transformation $A^*F \to \id$:

\begin{proposition}[Universal property twisted norm map]
	\label{prop:UniversalPropertyTwistedNormMap}
	For a $\Cc$-linear functor $F\colon \Cc^A \to \Cc$, the composite
	\begin{align*}
		\Nat_{\Cc}(F,A_!( - \otimes D_A)) \xrightarrow{A^* \circ -} \Nat_{\Cc}(A^*F,A^* A_!( - \otimes D_A)) \xrightarrow{\Nmadj_A \circ -} \Nat_{\Cc}(A^* F,\id_{\Cc^A})
	\end{align*}
	is an equivalence of spaces.
\end{proposition}
\begin{proof}
	Observe that the three instances of the equivalence of \cref{cor:ClassifyingCLinearNaturalTransformations} fit in the following commutative diagram:
	\begin{equation*}
		\hspace{-12.15pt}\hfuzz=12.15pt
		\begin{tikzcd}[cramped]
			\Nat_{\Cc}(F,A_!(- \otimes D_A)) 
			\dar{\ev_{\Delta_!\unit_A}}[swap]{\simeq} 
			\rar{A^* \circ -} 
			& \Nat_{\Cc}(A^*\circ F,A^* \circ A_!(- \otimes D_A)) 
			\rar{\Nmadj_A \circ -} \dar{\ev_{\Delta_!\unit_A}}[swap]{\simeq} & \Nat_{\Cc}(A^*\circ F,\id_{\Cc^A}) \dar{\ev_{\Delta_!\unit_A}}[swap]{\simeq}\\
			\Hom_{\Cc(A)}(D_F, D_A) \rar{\pr_1^*} 
			& \Hom_{\Cc(A \times_B A)}(\pr_1^*D_F, \pr_1^*D_A) \rar{c^*_{\pr_1} \circ -} 
			& \Hom_{\Cc(A \times_B A)}(\pr_1^*D_F,\Delta_!\unit_A).
		\end{tikzcd}
	\end{equation*}
	The left square commutes by naturality of the equivalence $\Fun_{\Cc}(\Cc^A,\Dd) \simeq \Dd(A)$ in $\Dd$ and the right square commutes because an equivalence of categories preserves composition. It thus remains to show that the bottom composite in the diagram is an equivalence. But this is clear since it is given by the adjunction equivalence on hom-spaces for the adjunction $\pr_1^* \dashv {\pr_1}_*$.
\end{proof}

As a consequence of the universal property of $\Nm_A$, we may express twisted ambidexterity in terms of \textit{internal left adjoints} in $\Mod_{\Cc}(\PrL(\Bb))$, in the sense of \Cref{def:Internal_Left_Adjoint}:

\begin{proposition}
	\label{prop:Twisted_Ambidexterity_In_Terms_Of_Norm_Map}
	The object $A$ is twisted $\Cc$-ambidextrous if and only if the $\Cc$-linear $\Bb$-functor $A^*\colon \Cc \to \Cc^A$ is an internal left adjoint in $\Mod_{\Cc}(\PrL(\Bb))$.
\end{proposition}
\begin{proof}
	Note that $A$ is twisted $\Cc$-ambidextrous if and only if the adjoint twisted norm map 
	\[
	\Nmadj_A \colon A^*A_!(- \otimes_A D_A) \implies \id_{\Cc^A}
	\]
	exhibits $A_!(- \otimes_A D_A)$ as a $\Cc$-linear right adjoint of $A^*$, so that one implication is clear. Conversely, if $A^*$ is an internal left adjoint with $\Cc$-linear right adjoint $A_*$, the $\Cc$-linear counit $A^*A_* \to \id$ equips $A_*$ with the same universal property of $A_!(- \otimes D_A)$ of \Cref{prop:UniversalPropertyTwistedNormMap}: for every $\Cc$-linear $\Bb$-functor $F\colon \Cc^A \to \Cc$, the composite
	\[
	\Nat_{\Cc}(F,A_*) \xrightarrow{A^* \circ -} \Nat_{\Cc}(A^*F,A^* A_*) \xrightarrow{\Nmadj_A \circ -} \Nat_{\Cc}(A^* F,\id_{\Cc^A})
	\]
	is an equivalence of spaces. It follows that $A_!(- \otimes_A D_A)$ and $A_*$ are equivalent as $\Cc$-linear $\Bb$-functors $\Cc^A \to \Cc$, necessarily via the twisted norm map.
\end{proof}

\begin{remark}
	\label{rmk:Criteria_Twisted_Ambidexterity}
	Applying \Cref{prop:Twisted_Ambidexterity_In_Terms_Of_Norm_Map} to the slice topos $\Bb_{/B}$, it follows that a morphism $f\colon A \to B$ in $\Bb$ is twisted $\Cc$-ambidextrous if and only if the $\Cc$-linear $\Bb_{/B}$-functor $f^*\colon \pi_B^*\Cc \to (\pi_B^*\Cc)^A$ is an internal left adjoint in $\Mod_{\Cc}(\xPrLB{B})$, i.e.\ the right adjoint $f_*$ preserves $\Bb_{/B}$-parametrized colimits and satisfies the projection formula. Applying the criterion \Cref{lem:Criterion_Internal_Left_Adjoint} to the slice topos $\Bb_{/B}$, this can be made explicit in terms of non-parametrized criteria: $f$ is twisted $\Cc$-ambidextrous if and only if for every pullback diagram
	\[
	\begin{tikzcd}
		A'' \dar[swap]{f''} \drar[pullback] \rar{\alpha} & A' \dar[swap]{f'} \drar[pullback] \rar & A \dar{f} \\
		B'' \rar{\beta} & B' \rar & B
	\end{tikzcd}
	\]
	in $\Bb$ the following three conditions hold:
	\begin{enumerate}[(1)]
		\item (Preserving fiberwise colimits) The functor $f'_*\colon \Cc(A') \to \Cc(B')$ preserves colimits;
		\item (Preserving groupoid-indexed colimits) The Beck-Chevalley square
		\[
		\begin{tikzcd}
			\Cc(A') \dar[swap]{f'_*} \rar{\alpha^*} & \Cc(A'') \dar{f''_*} \\
			\Cc(B')  \rar{\beta^*} & \Cc(B'')
		\end{tikzcd}
		\]
		is horizontally left adjointable;
		\item (Right projection formula) For objects $X \in \Cc(A')$ and $Y \in \Cc(B')$, the canonical morphism
		\begin{align*}
			f'_*(X) \otimes Y \to f'_*(X \otimes {f'}^*(Y))
		\end{align*}
		in $\Cc(B')$ is an equivalence.
	\end{enumerate}
	Conditions (1) and (2) correspond to the assumption that the $\Bb_{/B}$-functor $f_*$ preserves $\Bb_{/B}$-parametrized colimits, while condition (3) is equivalent to the condition $f_*$ is $\Cc$-linear.
\end{remark}

\begin{example}
	When $\Bb = \Spc$ is the $\infty$-topos $\Spc$ of spaces and $\Cc = \Sp$ is the $\infty$-category of spectra, the universal property of the twisted norm map appears as \cite[Theorem I.4.1(v)]{NikolausScholze2018Cyclic}. Since every colimit-preserving functor between stable presentable $\infty$-categories is $\Sp$-linear, the universal property simplifies to the statement that the twisted norm map $\colim_A(- \otimes D_A) \Rightarrow \lim_A(-)$ exhibits its source as the universal colimit-preserving approximation of its target. The object $D_A \in \Sp^A$ is called the \textit{dualizing spectrum} of $A$, and was studied by Klein \cite{klein2001dualizing}. As a functor $A \to \Sp$, it may be identified with the composite
	\begin{align*}
		A \xrightarrow{a \mapsto \Map_A(a,-)} \Spc^A \xrightarrow{\Sigma^{\infty}_+} \Sp^A \xrightarrow{\lim_A} \Sp.
	\end{align*}
\end{example}

\begin{example}
    Let $G$ be a compact Lie group and let $\Bb = \Spc_G$ be the $\infty$-topos of $G$-spaces. For $\Cc = \ul{\Sp}^G$ the $\Bb$-category of genuine $G$-spectra, to be introduced in \Cref{def:GCategoryOfGenuineGSpectra} below, the universal property of the twisted norm map was established by Quigley and Shah \cite[Theorem~5.47]{QuigleyShay2021Tate}. Again the statement simplifies due to $\ul{\Sp}^G$-linearity being automatic.
\end{example}

\begin{example}
	Let $\Bb = \Spc$ be the $\infty$-topos of spaces and let $\Cc$ be a presentably symmetric monoidal $\infty$-category. By \Cref{prop:Twisted_Ambidexterity_In_Terms_Of_Norm_Map}, a space $A$ is twisted $\Cc$-ambidextrous if and only if it is \textit{$\Cc$-adjointable} in the terminology of \cite[Definition~4.17]{CCRY_Characters}. In particular, the following are examples of twisted $\Cc$-ambidextrous spaces:
	\begin{enumerate}[(1)]
		\item If $\Cc$ is stable, then every compact space is twisted $\Cc$-ambidextrous, \cite[Example~4.24]{CCRY_Characters}.
		\item If $\Cc$ is $m$-semiadditive for some $m \geq -2$, then every $m$-finite space is twisted $\Cc$-ambidextrous, \cite[Example~4.22]{CCRY_Characters}.
		\item It $\Cc$ is an $n$-category (i.e., its mapping spaces are $(n-1)$-truncated), then every $n$-connected space is twisted $\Cc$-ambidextrous, \cite[Example~4.23]{CCRY_Characters}.
		\item Let $\Cc = \PrL_{\kappa}$ be the $\infty$-category of $\kappa$-presentable $\infty$-categories for a regular cardinal $\kappa$. Then \textit{every} space is twisted $\Cc$-ambidextrous, \cite[Example~4.26]{CCRY_Characters}.
	\end{enumerate}
\end{example}

The twisted $\Cc$-ambidextrous morphisms form a well-behaved class of morphisms in $\Bb$:

\begin{proposition}
	\label{prop:ClosurePropertiesTwistedAmbiMaps}
	The collection of twisted $\Cc$-ambidextrous morphisms is
	\begin{enumerate}[(1)]
		\item closed under composition;
		\item \label{it:DisjointUnions} closed under arbitrary disjoint unions;
		\item closed under base change;
		\item closed under cartesian products;
		\item \label{it:LocalClass} a local class of morphisms in $\Bb$, in the sense of \cite[Definition~6.1.3.8]{lurie2009HTT}. Equivalently, it is closed under disjoint unions, and for every pullback square
	\begin{equation}
        \label{eq:Twisted_Ambidextrous_Local_Class}
	\begin{tikzcd}
		A' \dar[swap]{f'} \rar{\alpha} \drar[pullback] & A \dar{f} \\
		B' \rar[twoheadrightarrow]{\beta} & B
	\end{tikzcd}
	\end{equation}
	in $\Bb$ such that $f'$ is twisted $\Cc$-ambidextrous and $\beta\colon B' \twoheadrightarrow B$ is an effective epimorphism, also $f$ is twisted $\Cc$-ambidextrous.
	\end{enumerate}
\end{proposition}
\begin{proof}
	(1) For closure under composition, let $f\colon A \to B$ and $g\colon B \to C$ be twisted $\Cc$-ambidextrous morphisms. We may write the $\Bb_{/C}$-functor $(gf)^*\colon [C,\Cc]_C \to [A,\Cc]_C$ as a composite of $g^*\colon [C,\Cc]_C \to [B,\Cc]_C$ and $f^*\colon [B,\Cc]_C \to [A,\Cc]_C$. These are both left adjoints internal to $\Mod_{\Cc}(\xPrLB{C})$ by \Cref{rmk:Criteria_Twisted_Ambidexterity}, and thus so is $(gf)^*$. Another application of \Cref{rmk:Criteria_Twisted_Ambidexterity} then shows that $gf$ is twisted $\Cc$-ambidextrous as well.
	
	(2) Closure under disjoint unions follows directly from \Cref{rmk:Criteria_Twisted_Ambidexterity}, using the equivalence $\Cc(\bigsqcup_i A_i) \simeq \prod_{i} \Cc(A_i)$. (3) Closure under base change is immediate from \Cref{rmk:Criteria_Twisted_Ambidexterity}. (4) Closure under cartesian products follows from closure under base change and closure under composition.
	
	(5) We check that the twisted $\Cc$-ambidextrous morphisms form a local class in $\Bb$. We showed closure under disjoint unions in part (2). Consider a pullback square \eqref{eq:Twisted_Ambidextrous_Local_Class}, where $f'$ is twisted $\Cc$-ambidextrous and $\beta\colon B' \twoheadrightarrow B$ is an effective epimorphism. We have to show that also $f$ is twisted $\Cc$-ambidextrous, i.e. the $\Bb_{/B}$-functor $f^*\colon \pi_B^*\Cc \to (\pi_B^*\Cc)^A$ is an internal left adjoint in $\Mod_{\Cc}(\xPrLB{B})$. By \Cref{prop:Test_After_Pullback_NaturalEquivalence}, it suffices to check this after pulling $f^*$ back along $\beta$ to the slice $\Bb_{/B'}$. But there it becomes the condition that the $\Bb_{/B'}$-functor ${f'}^*\colon \pi^*_{B'}\Cc \to (\pi^*_{B'}\Cc)^{A'}$ is an internal left adjoint in $\Mod_{\Cc}(\xPrLB{B'})$, which holds by assumption on $f'$.
\end{proof}

\begin{corollary}
	\label{cor:TwistedAmbi_Fiberwise}
	For $\Bb = \Spc$, the $\infty$-topos of spaces, a map $f\colon A \to B$ is twisted $\Cc$-ambidextrous if and only if each of its fibers are twisted $\Cc$-ambidextrous.
\end{corollary}
\begin{proof}
    This is immediate from parts (2) and (5) of \Cref{prop:ClosurePropertiesTwistedAmbiMaps}, since there exists an effective epimorphism $\bigsqcup_{\pi_0(B)} \pt = \pi_0(B) \twoheadrightarrow B$.
\end{proof}

\subsection{Relation to ambidexterity and parametrized semiadditivity}
\label{subsec:SemiadditivityVsTwistedAmbidexterity}
In \cite[Construction 4.1.8, Remark 4.1.12]{hopkinsLurie2013ambidexterity}, Hopkins and Lurie introduce for every `Beck-Chevalley fibration' $q\colon \Cc \to \Bb$ a collection of \textit{$\Cc$-ambidextrous morphisms $f\colon A \to B$ in $\Bb$}, each of which come equipped with a norm equivalence $\Nm_f\colon f_! \iso f_*$. In this subsection, we will compare this notion of ambidexterity with our notion of twisted ambidexterity. As a consequence, we relate twisted ambidexterity with the notion of parametrized semiadditivity introduced by Nardin \cite{nardin2016exposeIV} and Lenz, Linskens and the author \cite{CLL_Global}.

\begin{definition}[Iterated diagonals]
	\label{def:IteratedDiagonal}
	Let $f\colon A \to B$ be a morphism in $\Bb$. The \textit{diagonal} $\Delta(f)$ of $f$ is the map $(1,1)\colon A \to A \times_B A$. The \textit{iterated diagonals} $\Delta^k(f)$ of $f$ are defined inductively by letting $\Delta^0(f) := f$ and $\Delta^{k+1}(f) := \Delta(\Delta^k(f))$.
\end{definition}

The functor $\Cc\colon \Bb\catop \to \Cat_{\infty}$ can be unstraightened to a cartesian fibration $\widetilde{\Cc} \to \Bb$. Since $\Cc$ has parametrized colimits, this is a Beck-Chevalley fibration, in the sense of \cite[Definition~4.1.3]{hopkinsLurie2013ambidexterity}.

\begin{proposition}
	\label{prop:AmbidexterityVsTwistedAmbidexterity}
	Let $f\colon A \to B$ be a morphism in $\Bb$ and assume that $f$ is $n$-truncated for some natural number $n$. 
	\begin{enumerate}[(1)]
		\item The morphism $f$ is \textit{$\Cc$-ambidextrous} (in the sense of \cite[Construction~4.1.8]{hopkinsLurie2013ambidexterity} applied to the Beck-Chevalley fibration $\widetilde{\Cc} \to \Bb$) if and only if each of the iterated diagonals $\Delta^k(f)$ for $k = 0, 1, \dots, n+1$ is twisted $\Cc$-ambidextrous (in the sense of \Cref{def:TwistedAmbidexterityMorphisms}).
		\item Similarly, $f$ is weakly $\Cc$-ambidextrous (in the sense of \cite[Construction~4.1.8]{hopkinsLurie2013ambidexterity} applied to the Beck-Chevalley fibration $\widetilde{\Cc} \to \Bb$) if and only if each of the iterated diagonals $\Delta^k(f)$ for $k = 1, \dots, n+1$ is twisted $\Cc$-ambidextrous.
		\item If $f$ is weakly $\Cc$-ambidextrous, then there is an equivalence $D_f \simeq \unit_A$ between the dualizing object $D_f$ and the monoidal unit $\unit_A \in \Cc(A)$, and the composite
		\[
		\smash{f_!(-) \simeq f_!(- \otimes_A D_f) \xrightarrow{\Nm_f} f_*(-)}
		\]
		is equivalent to the norm map $f_! \to f_*$ of \cite[Remark~4.1.12]{hopkinsLurie2013ambidexterity}.
	\end{enumerate}
\end{proposition}
\begin{proof}
	We prove the three claims by simultaneous induction on $n$. For $n = -2$, $f$ is an equivalence and parts (1) and (2) are vacuous as every equivalence is both $\Cc$-ambidextrous as well as twisted $\Cc$-ambidextrous. In this case, $f_!$ and $f_*$ are both inverse to $f^*$ and therefore admit a canonical equivalence $\Nm_f\colon f_! \simeq f_*$ of parametrized functors $\Cc^A \to \Cc$, which agrees with the norm map of Hopkins and Lurie. Evaluating this equivalence at $\Delta_!\unit_A \in \Cc(A \times A) = \Cc^A(A)$ gives an equivalence $\unit_A \simeq D_f$ in $\Cc(A)$, and the last statement of (3) then holds by construction.
	
	Now assume that $n \geq -1$. The diagonal $\Delta(f)$ of $f$ is $(n-1)$-truncated, so by part (1) of the induction hypothesis it is $\Cc$-ambidextrous (that is, $f$ is weakly $\Cc$-ambidextrous) if and only if all the iterated diagonals $\Delta^k(f)$ for $k = 1, \dots, n + 1$ are twisted $\Cc$-ambidextrous, proving part (2). In this case, there is an equivalence
	\[
	D_f = {\pr_1}_*\Delta_!\unit_A \xrightarrow[\simeq]{\Nm_{\Delta}} {\pr_1}_*\Delta_*\unit_A = \unit_A \qin \Cc(A).
	\]
	Plugging in (the inverse of) this equivalence in the twisted norm map $\Nm_f$ and using the description of $\Nm_f$ given in \Cref{lem:Explicit_Formula_Twisted_Norm} (applied to the slice topos $\Bb_{/B}$), one sees that the composite in (3) is adjoint to the following composite:
	\[
	f^*f_!(-) \simeq {\pr_2}_!\pr_1^*(-) \xrightarrow{u^*_{\Delta}} {\pr_2}_!\Delta_*\Delta^*\pr_1^*(-) \xleftarrow[\simeq]{\Nm_{\Delta}} {\pr_2}_!\Delta_!\Delta^*\pr_1^*(-) \simeq \id \circ \id = \id.
	\]
	But this composite is precisely (a parametrized version of) the map $\nu_f^{(n+1)}\colon f^*f_! \to \id$ of \cite[Construction~4.1.8]{hopkinsLurie2013ambidexterity}, and its adjoint $f_! \to f_*$ is the norm map of \cite[Remark~4.1.12]{hopkinsLurie2013ambidexterity}, finishing the proof of (3). 
	
	Finally we deduce part (1) from (2) and (3). Given (2), we may assume that $f$ is weakly $\Cc$-ambidextrous, and we need to show it is twisted $\Cc$-ambidextrous if and only if it is $\Cc$-ambidextrous. In other words, we need to show that the twisted norm map is an equivalence if and only if the norm map $f_! \to f_*$ of Hopkins and Lurie is an equivalence. This is immediate from part (3).
\end{proof}

Specializing the result of \Cref{prop:AmbidexterityVsTwistedAmbidexterity} to the case where $\Bb$ is the $\infty$-topos of spaces, we obtain the following corollary:

\begin{corollary}
	\label{cor:AmbidexterityVsTwistedAmbidexterity_Space}
	Let $\Cc$ be a presentably symmetric monoidal $\infty$-category and let $A$ be a connected $n$-truncated space. Then the following conditions are equivalent:
	\begin{enumerate}[(1)]
		\item The space $A$ is \textit{$\Cc$-ambidextrous} in the sense of \cite[Definition~4.3.4]{hopkinsLurie2013ambidexterity};
		\item Each of the objects $A, \Omega A, \dots, \Omega^{n+1} A$ is twisted $\Cc$-ambidextrous in the sense of \Cref{def:TwistedAmbidexterity}.
	\end{enumerate}
\end{corollary}
\begin{proof}
	Letting $f\colon A \to \pt$ denote the map from $A$ to the point, we observe that each of the fibers of the iterated diagonal $\Delta^kf$ is given by the $k$-fold loop space $\Omega^k A$. It follows from \Cref{cor:TwistedAmbi_Fiberwise} that $\Omega^kA$ is twisted $\Cc$-ambidextrous if and only if the iterated diagonal $\Delta^kA$ of $A$ is twisted $\Cc$-ambidextrous. The claim thus follows from \Cref{prop:AmbidexterityVsTwistedAmbidexterity}.
\end{proof}

As a consequence, we obtain a characterization of higher semiadditivity in terms of twisted ambidexterity. An advantage of this characterization over the usual definition of ambidexterity is that all the twisted norm maps are a priori defined rather than through an inductive process.

\begin{corollary}
	\label{cor:HigherSemiadditivityInTermsOfTwistedAmbidexterity}
	Let $\Cc$ be a presentably symmetric monoidal $\infty$-category and let $-2 \leq m \leq \infty$. Then $\Cc$ is $m$-semiadditive if and only if each $m$-finite space is twisted $\Cc$-ambidextrous.
\end{corollary}
\begin{proof}
	This is immediate from \Cref{cor:AmbidexterityVsTwistedAmbidexterity_Space}, as the iterated loop spaces of an $m$-finite space are again $m$-finite.
\end{proof}

\begin{corollary}
	Let $\Cc$ be a presentably symmetric monoidal $\Bb$-category, let $A \in \Bb$ and let $-2 \leq m \leq \infty$. Then the following two conditions are equivalent:
	\begin{enumerate}[(1)]
		\item the $\infty$-category $\Cc(A)$ is $m$-semiadditive;
		\item the fold map $\colim_X A \to A$ is twisted $\Cc$-ambidextrous for every $m$-finite space $X$.
	\end{enumerate}
\end{corollary}
\begin{proof}
	For an object $A \in \Bb$, consider the (unique) colimit-preserving functor $L_A\colon \Spc \to \Bb$ sending the point to $A$, given on objects by $X \mapsto \colim_X A$. The $\infty$-category $\Cc(A)$ is encoded as a $\Spc$-category by the composite
	\[
	\Spc\catop \xrightarrow{L_A} \Bb\catop \xrightarrow{\Cc} \PrL.
	\]
	Since $L_A$ preserves pullbacks, a morphism $f\colon X \to Y$ of spaces is twisted $\Cc(A)$-ambidextrous if and only if the map $L_A(f)\colon L_A(X) \to L_A(Y)$ is twisted $\Cc$-ambidextrous. Applying this to $Y = \pt$ shows that $X \to *$ is twisted $\Cc(A)$-ambidextrous if and only if the map $\colim_X A = L_A(X) \to A$ is twisted $\Cc$-ambidextrous. The claim now follows from \Cref{cor:HigherSemiadditivityInTermsOfTwistedAmbidexterity}.
\end{proof}

In \cite{nardin2016exposeIV} and \cite{CLL_Global}, parametrized notions of semiadditivity were introduced. By \Cref{prop:AmbidexterityVsTwistedAmbidexterity}, these may be expressed in terms of twisted ambidexterity:

\begin{corollary}
	Let $T$ be a small $\infty$-category and let $P \subseteq T$ be an atomic orbital subcategory, in the sense of \cite[Definition~4.3.1]{CLL_Global}. Let $\Cc$ be a presentably symmetric monoidal $\PSh(T)$-category. Then $\Cc$ is $P$-semiadditive in the sense of \cite[Definition~4.5.1]{CLL_Global} if and only if $\Cc$ is fiberwise semiadditive and every morphism $p\colon A \to B$ in $P$ is twisted $\Cc$-ambidextrous.
\end{corollary}
\begin{proof}
	By \cite[Corollary~4.5.7]{CLL_Global}, $\Cc$ is $P$-semiadditive if and only if it is fiberwise semiadditive and the norm map $\Nm_p\colon p_! \to p_*$ from \cite[Construction~4.3.8]{CLL_Global} is an equivalence for every morphism $p\colon A \to B$ in $P$. By \cite[Remark~4.3.9]{CLL_Global}, this norm map agrees with the norm map $\Nm_p$ defined by Hopkins and Lurie, which in turn agrees with the twisted norm map $\Nm_p$ by \Cref{prop:AmbidexterityVsTwistedAmbidexterity}. This finishes the proof.
\end{proof}

\begin{corollary}
	\label{cor:ParametrizedSemiadditivityAsTwistedAmbidexterity}
	Let $T$ be an atomic orbital $\infty$-category, in the sense of \cite[Definition~4.1]{nardin2016exposeIV}. Let $\Cc$ be a presentably symmetric monoidal $\PSh(T)$-category. Then $\Cc$ is $T$-semiadditive in the sense of \cite[Definition~5.3]{nardin2016exposeIV} if and only if $\Cc$ is fiberwise semiadditive and every morphism $f\colon A \to B$ in $T$ is twisted $\Cc$-ambidextrous.
\end{corollary}
\begin{proof}
	This follows immediately from the previous corollary for $P = T$, since by \cite[Proposition~4.6.4]{CLL_Global} the norm map from \cite[Construction~4.3.8]{CLL_Global} is equivalent to the norm map constructed in \cite[Construction~5.2]{nardin2016exposeIV}.
\end{proof}

\subsection{Costenoble-Waner duality}
\label{subsec:Costenoble_Waner_Duality}
There is a close link between twisted ambidexterity and Costenoble-Waner duality, a form of duality theory in parametrized homotopy theory introduced in the early 2000's by Costenoble and Waner \cite{CostenobleWaner2016equivariant} and subsequently developed in more detail by May and Sigurdsson \cite{maysigurdsson2006parametrized}. The goal of this subsection is to introduce a general form of Costenoble-Waner duality in an arbitrary presentably symmetric monoidal $\Bb$-category $\Cc$ and explain its relationship with twisted ambidexterity.

Recall that an object $X$ of a symmetric monoidal $\infty$-category $\Dd$ is called \textit{dualizable} if there exists another object $Y \in \Dd$, called the \textit{dual} of $X$, together with an evaluation map $\epsilon\colon X \otimes Y \to \unit$ and a coevaluation map $\eta\colon \unit \to Y \otimes X$ satisfying the triangle identities:
\begin{equation*}
	\begin{tikzcd}[cramped, column sep = 2]
		& X \otimes Y\otimes X \drar{\epsilon \otimes X} \\
		X \otimes \unit \urar{X\otimes \eta} \rar[equal] & X \rar[equal] & \unit \otimes X
	\end{tikzcd}
	\hspace{30pt}\text{and} \hspace{30pt}
	\begin{tikzcd}[cramped, column sep = 2]
		& Y\otimes X \otimes Y\drar{Y\otimes \epsilon} \\
		\unit \otimes Y\urar{\eta \otimes Y} \rar[equal] & Y \rar[equal] & Y\otimes \unit.
	\end{tikzcd}
\end{equation*}
It is not difficult to see that this is equivalent to the $\Dd$-linear functor $X \otimes -\colon \Dd \to \Dd$ admitting a $\Dd$-linear right adjoint, necessarily of the form $Y \otimes -\colon \Dd \to \Dd$, with $\Dd$-linear unit and counit.

In our approach to Costenoble-Waner duality, we will generalize the above perspective to the parametrized setting. In place of the correspondence between objects $X \in \Dd$ and $\Dd$-linear functors $\Dd \to \Dd$ in the non-parametrized setting, we will use the equivalence $\Fun_{\Cc}(\Cc^A,\Cc^B) \simeq \Cc(A \times B)$ from \Cref{thm:Classification_CLinear_Functors} in the parametrized setting, where $\Cc$ is a presentably symmetric monoidal $\Bb$-category. In particular, any object $X \in \Cc(A \times B)$ determines a $\Cc$-linear $\Bb$-functor $F_X\colon \Cc^A \to \Cc^B$ given by the composite
\[
F_X\colon \Cc^A \xrightarrow{\pr_A^*} \Cc^{A \times B} \xrightarrow{- \otimes_{A \times B} X} \Cc^{A \times B} \xrightarrow{{\pr_B}_!} \Cc^B,
\]
and every $\Cc$-linear $\Bb$-functor $F\colon \Cc^A \to \Cc^B$ is of this form for a unique object $D_F \in \Cc(A \times B)$. 

\begin{convention}
	Whenever we write $X \in \Cc(A \times B)$, we think of $X$ as being directed from $A$ towards $B$. If we wish to think of $X$ as being directed from $B$ towards $A$, we will write $X \in \Cc(B \times A)$ instead. This also applies when $B = 1$ is the terminal object of $\Bb$, meaning that we distinguish between $X \in \Cc(A) = \Cc(A \times 1)$ and $X \in \Cc(A) = \Cc(1 \times A)$.
\end{convention}

\begin{definition}[{cf.\ \cite[Construction~17.1.3, Proposition~17.1.4]{maysigurdsson2006parametrized}}]
	\label{def:CompositionProduct}
	For an object $A \in \Bb$, we define $U_A := \Delta_!\unit_A \in \Cc(A \times A)$. For objects $A,B,C \in \Bb$ and objects $X \in \Cc(A \times B)$ and $Y \in \Cc(B \times C)$, their \textit{composition product} $Y \odot X \in \Cc(A \times C)$ is defined as
	\[
	Y \odot X := (\pr_{AC})_!(\pr_{AB}^*X \otimes \pr_{BC}^*Y),
	\]
	where $\pr_{AB}\colon A \times B \times C \to A \times B$ denotes the projection and similarly for $\pr_{BC}$ and $\pr_{AC}$. This gives rise to a functor $- \odot -\colon \Cc(B \times C) \times \Cc(A \times B) \to \Cc(A \times C)$.
\end{definition}

The following lemma may be regarded as a justification for the definition of the composition product:

\begin{lemma}
	\label{lem:CompositionProductVsComposition}
	In the above situation, there are natural equivalences
	\begin{align*}
		F_{U_A} \simeq \id_{\Cc^A}\hspace{18pt} & \qin \Fun_{\Cc}(\Cc^A,\Cc^A), \\
		F_{Y \odot X} \simeq F_{Y} \circ F_X &\qin \Fun_{\Cc}(\Cc^A, \Cc^C).
	\end{align*}
\end{lemma}
\begin{proof}
	The first equivalence is immediate, as the equivalence $\Fun_{\Cc}(\Cc^A, \Cc^A) \simeq \Cc(A \times A)$ of \Cref{thm:Classification_CLinear_Functors} is given by evaluation at $\Delta_!\unit_A$. For the second equivalence, plugging in the definition of $Y \odot X$ and using the projection formula for ${\pr_{AC}}_!$ shows that $F_{Y \odot X}$ is given by the composite
	\[
	\Cc^A \xrightarrow{\pr_A^*} \Cc^{A \times C} \xrightarrow{\pr_{AC}^*} \Cc^{A \times B \times C} \xrightarrow{- \otimes \pr_{AB}^*X \otimes \pr_{BC}^*Y} \Cc^{A \times C \times B} \xrightarrow{{\pr_{AC}}_!} \Cc^{A \times C} \xrightarrow{{\pr_C}_!} \Cc^C.
	\]
	Using symmetric monoidality of the functors $\pr_{AB}^*$ and $\pr_{BC}^*$ and using the base change equivalence $\pr_B^*{\pr_B}_! \simeq {\pr_{BC}}_!\pr_{AB}^*$, this is equivalent to the composite
	\[
	\Cc^A \xrightarrow{\pr_A^*} \Cc^{A \times B} \xrightarrow{- \otimes X} \Cc^{A \times B} \xrightarrow{{\pr_B}_!} \Cc^B \xrightarrow{\pr_B^*} \Cc^{B \times C} \xrightarrow{- \otimes Y} \Cc^{B \times C} \xrightarrow{{\pr_C}_!} \Cc^C.
	\]
	But this is simply $F_{Y} \circ F_X$, finishing the proof.
\end{proof}

We will frequently use \Cref{lem:CompositionProductVsComposition} to deduce properties of the composition product which can be somewhat tedious to prove by hand. For example, it follows directly from \Cref{lem:CompositionProductVsComposition} that the composition product is associative and unital up to homotopy. The objects $U_A = \Delta_!\unit_A \in \Cc(A \times A)$ serve as identities with respect to the composition product: for an object $X \in \Cc(A \times B)$ there are equivalences
\[
X \odot U_A \simeq X \qquad \qquadtext{and} \qquad U_B \odot X \simeq X.
\]
For brevity, we will mostly suppress the associativity and unitality equivalences from the notation and treat them as identities, just like we do for functors.

We may now introduce a parametrized analogue of monoidal duality, first discovered by Costenoble and Waner \cite{CostenobleWaner2016equivariant} in the context of equivariant homotopy theory.

\begin{definition}[{cf.\ \cite[Definition~16.4.1, Chapter~18]{maysigurdsson2006parametrized}}]
	\label{def:CostenobleWanerDuality}
	An object $X \in \Cc(A \times B)$ is called \textit{left Costenoble-Waner dualizable} if there is another object $Y \in \Cc(B \times A)$, called the \textit{left Costenoble-Waner dual of $X$}, together with morphisms
	\[
	\epsilon\colon X \odot Y \to U_B \qquad \text{ and } \qquad \eta\colon U_A \to Y \odot X
	\]
	in $\Cc(B \times B)$ resp. $\Cc(A \times A)$ satisfying the triangle identities
	\begin{equation*}
		\begin{tikzcd}[cramped, column sep = 2]
			& X \odot Y\odot X \drar{\epsilon \odot X} \\
			X \odot U_A \urar{X\odot \eta} \rar[equal] & X \rar[equal] & U_B \odot X
		\end{tikzcd}
		\hspace{30pt}\text{and} \hspace{30pt}
		\begin{tikzcd}[cramped, column sep = 2]
			& Y\odot X \odot Y\drar{Y\odot \epsilon} \\
			U_A \odot Y\urar{\eta \odot Y} \rar[equal] & Y \rar[equal] & Y\odot U_B.
		\end{tikzcd}
	\end{equation*}
	Conversely, we call $X$ the \textit{right Costenoble-Waner dual} of $Y$. Sometimes we say \textit{left dual} and \textit{right dual} for brevity. Note that $X \in \Cc(A \times B)$ is left Costenoble-Waner dualizable if and only if it is right Costenoble-Waner dualizable when treated as an object $X \in \Cc(B \times A)$.
\end{definition}

\begin{warning}
	In \cite{maysigurdsson2006parametrized}, the phrase `Costenoble-Waner duality' is only used when $A$ is the terminal object of $\Bb$. When $B$ is the terminal object, they use the phrase `fiberwise duality', compare also \Cref{lem:CostenobleWanerDualityVersusFiberwiseDuality} below.
\end{warning}

Due to the translation from \Cref{lem:CompositionProductVsComposition} between the composition product and composition of $\Cc$-linear $\Bb$-functors, we can express Costenoble-Waner duality in terms of internal adjunctions in $\Mod_{\Cc}(\PrL(\Bb))$, in the sense of \Cref{def:Internal_Left_Adjoint}.

\begin{lemma}
	\label{lem:CostenobleWanerAsInternalLeftAdjoint}
	An object $X \in \Cc(A \times B)$ is left Costenoble-Waner dualizable if and only if the $\Cc$-linear $\Bb$-functor $F_X\colon \Cc^A \to \Cc^B$ associated to $X$ is an internal left adjoint in $\Mod_{\Cc}(\PrL(\Bb))$. In this case, the $\Cc$-linear right adjoint of $F_X$ is given by $F_{Y}\colon \Cc^B \to \Cc^A$, where $Y \in \Cc(B \times A)$ is a left\footnote{The fact that \textit{left} duals correspond to \textit{right} adjoints is unfortunate but seems to be the standard convention.} Costenoble-Waner dual of $X$.
\end{lemma}
\begin{proof}
	If $X$ is left Costenoble-Waner dualizable with right dual $Y$, we may use \Cref{lem:CompositionProductVsComposition} to turn the evaluation and coevaluation $\epsilon$ and $\nu$ into $\Cc$-linear counit and unit maps $F_X \circ F_{Y} \to \id_{\Cc^B}$ and $\id_{\Cc^A} \to F_{Y} \circ F_X$, respectively. The triangle identities for the Costenoble-Waner duality between $X$ and $Y$ translate into the triangle identities of an internal adjunction $F_X \dashv F_{Y}$ in $\Mod_{\Cc}(\PrL(\Bb))$. Conversely, if $F_X$ is an internal left adjoint internal to $\Mod_{\Cc}(\PrL(\Bb))$ with $\Cc$-linear right adjoint $G\colon \Cc^B \to \Cc^A$, then it follows from \Cref{thm:Classification_CLinear_Functors} that $G$ is of the form $F_{Y}$ for some object $Y \in \Cc(B \times A)$, and by \Cref{lem:CompositionProductVsComposition} the unit and counit of the adjunction give rise to the coevaluation and evaluation satisfying the triangle identities, and thus providing duality data between $X$ and $Y$.
\end{proof}

\begin{remark}
	May and Sigurdsson \cite[Section~16.4]{maysigurdsson2006parametrized} define Costenoble-Waner duality as a special case of a notion they call \textit{duality in a closed symmetric bicategory}, applied to a certain bicategory $\Ee x$ of parametrized equivariant spectra \cite[Construction~17.1.3]{maysigurdsson2006parametrized}. From our perspective, one might think of their bicategory $\Ee x$ as the full subcategory of the homotopy 2-category of $\Mod_{\Cc}(\PrL(\Bb))$ spanned by objects of the form $\Cc^A$ for $A \in \Bb$. Indeed, due to \Cref{thm:Classification_CLinear_Functors} and \Cref{lem:CompositionProductVsComposition}, we obtain the following more explicit description of this bicategory:
	\begin{itemize}
		\item The objects $\Cc^A$ correspond to objects $A$ of $\Bb$;
		\item Given $A,B \in \Bb$, the category of morphisms $\Cc^A \to \Cc^B$ can be identified with the homotopy category $\Ho(\Cc(A \times B))$;
		\item The identity morphisms are given by $U_A \in \Cc(A \times A)$;
		\item The composition is given by the composition product $- \odot -\colon \Cc(B \times C) \times \Cc(A \times B) \to \Cc(A \times C)$.
	\end{itemize}
	When applied to the $\infty$-topos $\Spc^G$ of $G$-spaces for a compact Lie group $G$ and to the $G$-category $\ulSp^G$ of genuine $G$-spectra, to be defined in \Cref{subsec:genuineParametrizedSpectra} below, this is essentially the bicategory $\Ee x$ of May and Sigurdsson.
\end{remark}

We get the following reformulation of twisted ambidexterity in terms of Costenoble-Waner duality:

\begin{proposition}
	\label{prop:Twisted_Ambidexterity_VS_Costenoble_Waner_Dualizability}
	An object $A \in \Bb$ is twisted $\Cc$-ambidextrous if and only if the monoidal unit $\unit_A \in \Cc(A) = \Cc(1 \times A)$ is left Costenoble-Waner dualizable. In this case, the left dual of $\unit_A$ is given by the dualizing object $D_A \in \Cc(A) = \Cc(A \times 1)$.
\end{proposition}
\begin{proof}
	The $\Cc$-linear $\Bb$-functor $F_{\unit_A}\colon \Cc \to \Cc^A$ associated to $\unit_A$ is $A^*\colon \Cc \to \Cc^A$. It follows from \Cref{lem:CostenobleWanerAsInternalLeftAdjoint} that $\unit_A$ is Costenoble-Waner dualizable if and only if $A^*\colon \Cc \to \Cc^A$ is an internal left adjoint in $\Mod_{\Cc}(\PrL(\Bb))$, which by \Cref{prop:Twisted_Ambidexterity_In_Terms_Of_Norm_Map} is true if and only if $A$ is twisted $\Cc$-ambidextrous. In this case, the right adjoint of $A^*$ is classified by the object $D_A \in \Cc(A) = \Cc(A \times 1)$.
\end{proof}

\subsubsection{Characterizations of Costenoble-Waner duality}
For later use, we recall various alternative characterizations of Costenoble-Waner duality from \cite{maysigurdsson2006parametrized} and \cite{CostenobleWaner2016equivariant}.

\begin{lemma}[{cf.\ \cite[Proposition~16.4.6]{maysigurdsson2006parametrized}}]
	\label{lem:CharacterizationCostenobleWanerDuality}
	Consider objects $X \in \Cc(A \times B)$ and $Y \in \Cc(B \times A)$ and let $\epsilon\colon X \odot Y \to U_B$ be a morphism in $\Cc(B \times B)$. Then the following conditions are equivalent:
	\begin{enumerate}[(1)]
		\item The object $Y$ is a left Costenoble-Waner dual of $X$ with evaluation map $\epsilon$.
		\item For every $C \in \Bb$ and objects $W \in \Cc(A \times C)$ and $Z \in \Cc(B \times C)$, the map
		\[
		\Hom_{\Cc(A \times C)}(W, Z \odot X) \xrightarrow{- \odot Y} \Hom_{\Cc(B \times C)}(W \odot Y, Z  \odot X  \odot Y) \xrightarrow{\epsilon} \Hom_{\Cc(B \times C)}(W \odot Y, Z)
		\]
		is an equivalence;
		\item Condition (2) holds for $C = A$, $W = U_A$, $Z = Y$ and for $C = B$, $W = X$, $Z = U_B$.
	\end{enumerate}
\end{lemma}
\begin{proof}
	Given (1), an inverse to the map in (2) is given by 
	\[
	\Hom_{\Cc(B \times C)}(W \odot Y, Z) \xrightarrow{- \odot X} \Hom_{\Cc(A \times C)}(W \odot Y \odot X, Z \odot X) \xrightarrow{\eta} \Hom_{\Cc(A \times C)}(W, Z \odot X).
	\]
	It is clear that (2) implies (3). Given (3), we may take $C = A$, $W = U_A$, $Z = Y$ and define the coevaluation $\eta\colon U_A \to Y \odot X$ to be the inverse image of the identity on $Y$ under the equivalence from (2). One of the triangle identities holds by construction. For the other one, we take $C = B$, $W = X$ and $Z = U_B$, and observe that both $\id_X\colon X \to X$ as well as $(X \odot \epsilon) \circ (\eta \odot X)\colon X \to X$ are sent to the map $\epsilon\colon X \odot Y \to U_B$ by the equivalence from (2), implying that they are homotopic.
\end{proof}

\begin{definition}
	If the equivalent conditions of \Cref{lem:CharacterizationCostenobleWanerDuality} hold, we say that \textit{$\epsilon$ exhibits $Y$ as a left Costenoble-Waner dual of $X$}.
\end{definition}

Any object $X \in \Cc(A \times B)$ admits a \textit{weak} dual. For simplicity, we will only introduce this when $A$ is the terminal object of $\Bb$.

\begin{definition}[{cf.\ \cite[Definition~2.9.1]{CostenobleWaner2016equivariant}}]
	\label{def:WeakCostenobleWanerDual}
	Let $B \in \Bb$ and consider an object $X \in \Cc(B) = \Cc(1 \times B)$. We define the \textit{weak Costenoble-Waner dual} of $X$ to be
	\[
	D^{CW}_B(X) := {\pr_2}_*\iHom_{B \times B}(\pr_1^* X,\Delta_! \unit_B) \qin \Cc(B),
	\]
	where $\Delta\colon B \to B \times B$ is the diagonal of $B$, and $\pr_1,\pr_2\colon B \times B \to B$ are the two projections. 
\end{definition}

Note that for any object $Y \in \Cc(B) = \Cc(B \times 1)$, there is a one-to-one correspondence between morphisms $Y \to D^{CW}_B(X)$ in $\Cc(B)$ and morphisms $X \odot Y = \pr_1^*X \otimes_{B \times B} \pr_2^*Y \to \Delta_!\unit_B = U_B$ in $\Cc(B \times B)$. In particular, the identity on $D^{CW}_B(X)$ gives rise to a map $\epsilon_0 \colon X \odot D^{CW}_B(X) \to U_B$.

\begin{lemma}[{cf.\ \cite[Theorem~2.9.5]{CostenobleWaner2016equivariant}}]
	\label{lem:CostenobleWanerDualityInTermsOfWeakDual}
	An object $X \in \Cc(B) = \Cc(1 \times B)$ is left Costenoble-Waner dualizable if and only if the map $\epsilon_0\colon X \odot D^{CW}_B(X) \to U_B$ exhibits $D^{CW}_B(X)$ as a left Costenoble-Waner dual of $X$.
\end{lemma}
\begin{proof}
	The ``if'' direction is obvious. For the ``only if'', assume that $X$ is left Costenoble-Waner dualizable with left dual $Y \in \Cc(B \times 1)$ and let $\epsilon\colon X \odot Y \to U_B$ and $\eta\colon U_A \to Y \odot X$ be evaluation and coevaluation maps exhibiting this duality. The map $\epsilon$ adjoints over to a map $Y \to D^{CW}_B(X)$ and it will suffice to show that this is an equivalence. Indeed, a reasonably straightforward diagram chase shows that an inverse is given by the composite
	\[
	D^{CW}_B(X) = U_A \odot D^{CW}_B(X)\xrightarrow{\eta \odot 1} Y \odot X \odot D^{CW}_B(X) \xrightarrow{1 \odot \epsilon_0} Y \odot U_B = Y. \qedhere
	\]
\end{proof}

\begin{lemma}[{cf.\ \cite[Proposition~18.1.1]{maysigurdsson2006parametrized}}]
	\label{lem:CostenobleWanerDualityVersusFiberwiseDuality}
	An object $X \in \Cc(A) = \Cc(A \times 1)$ is left Costenoble-Waner dualizable if and only if it is dualizable in the symmetric monoidal $\infty$-category $\Cc(A)$. Its left Costenoble-Waner dual in $\Cc(A) = \Cc(1 \times A)$ is given by the monoidal dual in $\Cc(A)$.
\end{lemma}
\begin{proof}
	Consider objects $X \in \Cc(A) = \Cc(A \times 1)$ and $Y \in \Cc(A) = \Cc(1 \times A)$. Unwinding definitions, we see that $X \odot Y \simeq A_!(X \otimes_A Y) \in \Cc(1)$. It follows that a morphism $\epsilon\colon X \odot Y \xrightarrow{\epsilon} U_{1} = \unit$ in $\Cc(1)$ is the same data as a morphism $\epsilon'\colon X \otimes_A Y \to A^*\unit = \unit_A$ in $\Cc(A)$. We claim that $\epsilon$ exhibits $Y$ as a left Costenoble-Waner dual to $X$ if and only if $\epsilon'$ exhibits $Y$ as a dual of $X$ in $\Cc(A)$. Using some diagram chasing, this follows from \Cref{lem:CharacterizationCostenobleWanerDuality}. As the proof is entirely analogous to the proof of \cite[Proposition~18.1.1]{maysigurdsson2006parametrized}, we will omit it.
\end{proof}

\subsubsection{Preservation properties of Costenoble-Waner duality}
Costenoble-Waner dualizable objects are preserved under a variety of constructions. 

\begin{lemma}
	\label{lem:SymMonFunctorsPreserveCWDuality}
	Let $F\colon \Cc \to \Dd$ be a symmetric monoidal left adjoint between presentably symmetric monoidal $\Bb$-categories. Then $F$ preserves left/right Costenoble-Waner dualizable objects.
\end{lemma}
\begin{proof}
	It suffices to observe that $F$ commutes with the composition products on $\Cc$ and $\Dd$, which is immediate from the definition.
\end{proof}

\begin{corollary}
	\label{cor:SymMonFunctorsPreserveTwistedAmbidexterity}
	Let $F\colon \Cc \to \Dd$ be a symmetric monoidal left adjoint between presentably symmetric monoidal $\Bb$-categories. Then any twisted $\Cc$-ambidextrous object in $\Bb$ is also twisted $\Dd$-ambidextrous. Consequently, any twisted $\Cc$-ambidextrous morphism in $\Bb$ is also twisted $\Dd$-ambidextrous.
\end{corollary}
\begin{proof}
	Since $F$ preserves monoidal units, the first statement is a consequence of \Cref{lem:SymMonFunctorsPreserveCWDuality} and \Cref{prop:Twisted_Ambidexterity_VS_Costenoble_Waner_Dualizability}. The second statement follows by passing to slice topoi $\Bb_{/B}$.
\end{proof}

\begin{lemma}
	If $X \in \Cc(A \times B)$ and $X' \in \Cc(B \times C)$ are left Costenoble-Waner dualizable with left duals $Y$ and $Y'$, then so is their composition product $X' \odot X \in \Cc(A \times C)$, with left dual $Y \odot Y'$.
\end{lemma}
\begin{proof}
	This is immediate from \Cref{lem:CostenobleWanerAsInternalLeftAdjoint}, since internal adjunctions compose.
\end{proof}

\begin{lemma}
	\label{lem:PreservationCostenobleWanerDuality}
	Let $X \in \Cc(A \times B)$ be left Costenoble-Waner dualizable with left dual $Y \in \Cc(B \times A)$
	\begin{enumerate}[(1)]
		\item For a map $f\colon A' \to A$, the object $(f \times 1)^*X \in \Cc(A' \times B)$ is left Costenoble-Waner dualizable, with left dual given by $(1 \times f)^*Y$.
		\item For a map $g\colon B \to B'$, the object $(1 \times g)_!X \in \Cc(A \times B')$ is left Costenoble-Waner dualizable, with left dual given by $(g \times 1)_!Y$.
		\item For a twisted $\Cc$-ambidextrous map $f\colon A \to A'$, the object $(f \times 1)_!X \in \Cc(A' \times B)$ is left Costenoble-Waner dualizable, with left dual given by $(1 \times f)_!Y$.
		\item For a twisted $\Cc$-ambidextrous map $g\colon B' \to B$, the object $(1 \times g)^*X \in \Cc(A \times B')$ is left Costenoble-Waner dualizable, with left dual given by $(g \times 1)^*Y$.
	\end{enumerate}
\end{lemma}
\begin{proof}
	In each case, the claim follows from \Cref{prop:FreeCofreeCLinearCategories}(\ref{it:RestrictionInduction}) and the fact that adjunctions compose, using that for a morphism $f\colon A \to B$, the $\Bb$-functor $f_!\colon \Cc^A \to \Cc^B$ is always an internal left adjoint, while $f^*\colon \Cc^B \to \Cc^A$ is an internal left adjoint whenever $f$ is twisted $\Cc$-ambidextrous.
\end{proof}

\begin{corollary}
    \label{cor:Costenoble_Waner_Duality_Preserved_Under_Colimits}
	Assume that $X \in \Cc(1 \times A)$ is left Costenoble-Waner dualizable with left dual $Y \in \Cc(A \times 1)$. Then $A_!X \in \Cc(1)$ is a dualizable object with dual $A_!Y$. \qed
\end{corollary}

\begin{corollary}
	\label{cor:ColimitOfCWDblObjectIsDbl}
	Assume $A \in \Bb$ is twisted $\Cc$-ambidextrous. Then the object $A_!\unit_A \in \Cc(1)$ is dualizable, with dual $A_!D_A$.
\end{corollary}
\begin{proof}
	Combine the previous corollary with \Cref{prop:Twisted_Ambidexterity_VS_Costenoble_Waner_Dualizability}.
\end{proof}

\begin{proposition}[{cf.\ \cite[Proposition~16.8.1]{maysigurdsson2006parametrized}}]
	\label{prop:AtomicObjectsFormThickSubcategory}
	For every object every $B \in \Bb$, the collection of left Costenoble-Waner dualizable objects in $\Cc(1 \times B)$ is closed under retracts. If $\Cc$ is fiberwise stable, then these objects form a stable (and hence thick) subcategory of $\Cc(1 \times B)$.
\end{proposition}
\begin{proof}
	By \Cref{lem:CostenobleWanerDualityInTermsOfWeakDual} and \Cref{lem:CharacterizationCostenobleWanerDuality}, an object $X \in \Cc(1 \times B)$ is left Costenoble-Waner dualizable if and only if for all $C \in \Bb$, $W \in \Cc(1 \times C)$ and $Z \in \Cc(B \times C)$, a certain map
	\[
	\Hom_{\Cc(1 \times C)}(W, Z \odot X) \to \Hom_{\Cc(B \times C)}(W \odot D^{CW}_B(X), Z)
	\]
	is an equivalence. Since this map is natural in $X$, it follows that the collection of objects for which it is an equivalence is closed under retracts.
	
	If $\Cc$ is fiberwise stable, we can lift the above map to a map at the level of mapping spectra. In that case both sides are exact in $X$, and it follows that the collection of objects for which it is an equivalence forms a stable subcategory.
\end{proof}

\begin{corollary}
	\label{cor:Twisted_Ambidextrous_Objects_Closed_Under_Finite_Limits}
	The collection of twisted $\Cc$-ambidextrous objects of $\Bb$ is closed under retracts. If $\Cc$ is fiberwise stable, then the collection of twisted $\Cc$-ambidextrous objects of $\Bb$ is also closed under finite colimits.
\end{corollary}
\begin{proof}
	Consider a retract diagram $A \xrightarrow{s} B \xrightarrow{r} A$ in $\Bb$, i.e.\ we have $rs = \id_A$. Assume that $B$ is twisted $\Cc$-ambidextrous. It follows from \Cref{prop:Twisted_Ambidexterity_VS_Costenoble_Waner_Dualizability} that $\unit_B \in \Cc(B)$ is left Costenoble-Waner dualizable, so by \Cref{lem:PreservationCostenobleWanerDuality} also $r_!\unit_B$ is left Costenoble-Waner dualizable. Since $\unit_A \in \Cc(A)$ is a retract of $r_!\unit_B$, it is left Costenoble-Waner dualizable by \Cref{prop:AtomicObjectsFormThickSubcategory}, and thus $A$ is twisted $\Cc$-ambidextrous by \Cref{prop:Twisted_Ambidexterity_VS_Costenoble_Waner_Dualizability}.
	
	Now assume that $\Cc$ is fiberwise stable. Since $\Cc(\emptyset) = *$, it follows from pointedness of $\Cc$ that the initial object $\emptyset$ is twisted $\Cc$-ambidextrous. Consider a pushout diagram in $\Bb$
	\[
	\begin{tikzcd}
		A \dar[swap]{f} \rar{g} \drar[pushout] & B \dar{h} \\
		C \rar{k} & D,
	\end{tikzcd}
	\]
	and assume that $A$, $B$ and $C$ are twisted $\Cc$-ambidextrous. We need to show that $D$ is twisted $\Cc$-ambidextrous. By descent, the functors $h^*$ and $k^*$ induce an equivalence $(h^*,k^*)\colon \Cc(D) \xrightarrow{\sim} \Cc(B) \times_{\Cc(A)} \Cc(C)$. It follows that the monoidal unit $\unit_D \in \Cc(D)$ sits in a cofiber sequence
	\[
	(hg)_!\unit_A \to h_!\unit_B \oplus k_!\unit_C \to \unit_D:
	\]
    the fact that $(hg)_!\unit_A$ is a pushout of $h_!\unit_B$ and $k_!\unit_C$ along $\unit_D$ translates under the adjunction equivalences to the description of hom-spaces in $\Cc(D)$ as a pullback of the hom-spaces in $\Cc(B)$ and $\Cc(C)$ over $\Cc(A)$. By \Cref{prop:Twisted_Ambidexterity_VS_Costenoble_Waner_Dualizability}, the objects $\unit_A$, $\unit_B$ and $\unit_C$ are left Costenoble-Waner dualizable, and thus by \Cref{lem:PreservationCostenobleWanerDuality} so are $(hg)_!\unit_A$, $h_!\unit_B$ and $k_!\unit_C$. It thus follows from \Cref{prop:AtomicObjectsFormThickSubcategory} that $\unit_D$ is left Costenoble-Waner dualizable, hence $D$ is twisted $\Cc$-ambidextrous by \Cref{prop:Twisted_Ambidexterity_VS_Costenoble_Waner_Dualizability}. This finishes the proof.
\end{proof}

\section{Equivariant homotopy theory}
\label{sec:ambidexterity_in_equivariant_homotopy_theory}

The goal of this section is to investigate the notion of twisted ambidexterity in stable equivariant homotopy theory. Throughout the section, we will use the terminology `$G$-category' to refer to a $\Bb$-category when $\Bb$ is the $\infty$-topos $\Spc^G$ of $G$-spaces. Since $\Spc^G$ is equivalent to the presheaf category of the orbit $\infty$-category $\Orb_G$ of $G$, a $G$-category may equivalently be encoded as a functor $\Cc \colon \Orb_G\catop \to \Cat_{\infty}$. We let $\PrL_G := \PrL(\Spc^G)$ denote the (very large) $\infty$-category of presentable $G$-categories.

In \Cref{subsec:genuineParametrizedSpectra}, we introduce the $G$-category $\ulSp^G$ of genuine $G$-spectra, informally given by sending an orbit $G/H$ to the $\infty$-category $\Sp^H$ of genuine $H$-spectra. In \Cref{subsec:UniversalityOfWirthmullerIsomorphisms} we prove that $\ulSp^G$ is the initial stable presentably symmetric monoidal $G$-category for which all compact $G$-spaces are twisted ambidextrous. In \Cref{subsec:Orbispectra} and \Cref{subsec:ProperEquivariantHomotopyTheory}, we extend this result to the contexts of \textit{orbispectra} and \textit{proper equivariant homotopy theory}, respectively. In particular, we will obtain for every (not necessarily compact) Lie group $G$ and cocompact subgroup $H$ a Wirthmüller isomorphism
\[
\ind^G_H(- \otimes D_{G/H}) \simeq \coind^G_H(-)
\]
in the $\infty$-category of proper genuine $G$-spectra.

\subsection{Parametrized genuine \texorpdfstring{$G$}{G}-spectra}
\label{subsec:genuineParametrizedSpectra}
Let $G$ be a compact Lie group, fixed throughout this subsection. The goal of this subsection is to introduce the $G$-category $\ulSp^G$ of genuine $G$-spectra and discuss its universal property in terms of inverting representation spheres.

\begin{definition}
	\label{def:GenuineGSpectra}
	Let $\{S^V\}$ be a set of representation spheres, where $V$ runs over a set of representatives for the isomorphism classes of finite-dimensional irreducible $G$-representations.\footnote{Running over \textit{all} finite-dimensional $G$-representations gives an equivalent $\infty$-category.} We define the presentably symmetric monoidal $\infty$-category $\Sp^G$ of \textit{genuine $G$-spectra} as the formal inversion
	\begin{align*}
		\Sp^G := (\Spc^G_*)[\{S^V\}^{-1}]
	\end{align*}
	of the representation spheres $S^V$ in the $\infty$-category $\Spc^G_*$ of pointed $G$-spaces, see \cite[Definition~2.6]{robalo2015ktheory} or \Cref{subsec:Formal_Inversion}.
\end{definition}

By Gepner and Meier \cite[Corollary~C.7]{gepnerMeier2020equivariant}, the $\infty$-category $\Sp^G$ is equivalent to the $\infty$-category underlying the model category of orthogonal $G$-spectra with the stable model structure constructed in \cite[Section~III.4]{mandellMay2002equivariant}. References on equivariant orthogonal spectra include \cite{HHR2016nonexistence}, \cite{schwede2020lectures}, \cite{schwede2018global}.

The $\infty$-categories $\Spc^G_*$ and $\Sp^G$ are presentably symmetric monoidal and come equipped with a symmetric monoidal left adjoint from $\Spc^G$, making them into commutative $\Spc^G$-algebras in $\PrL$. We may now use the fully faithful functor
\[
- \otimes_{\Spc^G} \Omega_{\Spc^G}\colon \CAlg_{\Spc^G}(\PrL) \hookrightarrow \CAlg(\PrL_G)
\]
from \Cref{cor:Embedding_BAlgebras_Into_Presentably_Symmetric_Monoidal_BCategories} to regard them as presentably symmetric monoidal $G$-categories:

\begin{definition}
	\label{def:GCategoryOfGenuineGSpectra}
	We define presentably symmetric monoidal $G$-categories $\ulSpc^G$, $\ulSpc_*^G$ and $\ulSp^G$ as
	\[
	\ulSpc^G := \Omega_{\Spc^G}, \qquad \ulSpc_*^G := \Spc_*^G \otimes_{\Spc^G} \Omega_{\Spc^G}, \qquad \ulSp^G := \Sp^G \otimes_{\Spc^G} \Omega_{\Spc^G},
	\]
	called the $G$-categories of \textit{$G$-spaces}, \textit{pointed $G$-spaces} and \textit{genuine $G$-spectra}, respectively. We let
	\[
	(-)_+\colon \ulSpc^G \to \ulSpc^G_* \qquadtext{ and } \Sigma^{\infty}_+\colon \ulSpc^G \to \ulSp^G
	\]
	denote the induced maps, which are the unit maps for the algebra structures of  $\ulSpc^G_*$ and $\ulSp^G$ in $\PrL_G$. Unwinding definitions, we see that these $G$-categories are given at a $G$-space $B$ as follows:
	\begin{itemize}
		\item The $\infty$-category $\ulSpc^G(B)$ is the slice $\Spc^G_{/B}$ of $G$-spaces over $B$.
		\item The $\infty$-category $\ulSpc_*^G(B)$ is the relative tensor product $\Spc^G_{/B} \otimes_{\Spc^G} \Spc^G_*$, which by \cite[Example~4.8.1.21]{lurie2016HA} is equivalent to the  $\infty$-category $(\Spc^G_{/B})_*$ of \textit{retractive $G$-spaces} over $B$.
		\item The $\infty$-category $\ulSp^G(B)$ is the relative tensor product $\Spc^G_{/B} \otimes_{\Spc^G} \Sp^G$. As $\Sp^G$ is pointed, this is equivalent to the relative tensor product $(\Spc^G_{/B})_* \otimes_{\Spc^G_*} \Sp^G$, which by \Cref{lem:FormalInversionViaBaseChange} is in turn equivalent to $(\Spc^G_{/B})_*[\{S^V_B\}^{-1}]$, the formal inversion of the trivial sphere bundles $S^V_B := S^V \times B \to B$ in $(\Spc^G_{/B})_*$.
	\end{itemize}
\end{definition}

When $B = G/H$ is an orbit for some subgroup $H \leqslant G$, there are equivalences
\[
\ulSpc^G(G/H) \simeq \Spc^H, \qquad \ulSpc^G_*(G/H) \simeq \Spc^H_*, \qquadtext{ and } \ulSp^G(G/H) \simeq \Sp^H.
\]
Indeed, the slice of $\Spc^G$ over $G/H$ is equivalent to $\Spc^H$ by taking the fiber over $eH$, giving the first two equivalences. For the third equivalence, we observe that by \Cref{lem:FormalInversionStableUnderFactors} the $\infty$-category of genuine $H$-spectra can be obtained from the $\infty$-category of pointed $H$-spaces by just inverting the restricted representation spheres $\res^G_H(S^V)$ for irreducible $G$-representations $V$, as every irreducible $H$-representation is a direct summand of the restriction to $H$ of an irreducible $G$-representation, see Bröcker and tom Dieck \cite[Theorem~4.5]{BroeckerTomDieck1985Representations}.

The $G$-categories $\ulSpc^G$, $\ulSpc^G_*$ and $\ulSp^G$ admit the following universal properties:
\begin{proposition}
	\label{prop:UniversalPropertyGSpectraFormalInversion}
	Let $\Cc$ be a presentably symmetric monoidal $G$-category.
	\begin{enumerate}[(1)]
		\item There exists a unique symmetric monoidal left adjoint $G$-functor $F\colon \ulSpc^G \to \Cc$;
		\item The $G$-functor $F$ from (1) extends to a symmetric monoidal left adjoint $F'\colon \ulSpc^G_* \to \Cc$ if and only if $\Cc(1)$ is pointed, in which case the extension $F'$ is unique.
		\item If $\Cc(1)$ is pointed, the $G$-functor $F'$ from (2) extends to a symmetric monoidal left adjoint $F''\colon \ulSp^G \to \Cc$ if and only if the functor $F'(1) \colon \Spc^G_* \to \Cc(1)$ inverts the representation spheres $S^V$, in which case the extension $F''$ is unique.
	\end{enumerate}
\end{proposition}
\begin{proof}
	Part (1) is immediate as $\ulSpc^G$ is the monoidal unit of $\PrL_G$. Parts (2) and (3) follow by combining the adjunction from \Cref{cor:Embedding_BAlgebras_Into_Presentably_Symmetric_Monoidal_BCategories} with the analogous universal properties of $\Spc^G_* \simeq \Spc^G \otimes \Spc_*$ and $\Sp^G = \Spc^G_*[\{S^V\}^{-1}]$ in $\PrL$.
\end{proof}

For a $G$-space $B$, the functor $F_B\colon \Spc^G_{/B} \to \Cc(B)$ from part (1) sends a morphism $f\colon A \to B$ to the object $f_!\unit_A \in \Cc(B)$. The functor $F'_B\colon (\Spc^G_{/B})_* \to \Cc(B)$ sends a morphism $f\colon A \to B$ with section $s\colon B \to A$ to the cofiber of $\unit_B \simeq f_!s_!s^*\unit_A \to f_!\unit_A$ in $\Cc(B)$.

\subsection{Twisted ambidexterity for genuine \texorpdfstring{$G$}{G}-spectra}
\label{subsec:UniversalityOfWirthmullerIsomorphisms}

We continue to fix a compact Lie group $G$. The presentably symmetric monoidal $G$-category $\ulSp^G$ of genuine $G$-spectra is \textit{fiberwise stable}, meaning that it takes values in the subcategory $\PrL_{\st} \subseteq \PrL$ of stable presentable $\infty$-categories. Moreover, one can show that $\ulSp^G$ satisfies twisted ambidexterity for all compact $G$-spaces. The goal of this section is to show that $\ulSp^G$ is in a precise sense \textit{universal} with these two properties, see \Cref{thm:MainResult} below.

We start by recalling a result of May and Sigurdsson \cite{maysigurdsson2006parametrized} on Costenoble-Waner duality for $G$-spaces, based on ideas of Costenoble and Waner \cite{CostenobleWaner2016equivariant}.

\begin{construction}[{cf.\ \cite[Construction~3.2.7]{schwede2018global}}]
	\label{cons:EvaluationMapCWDualityGH}
	Let $H \leqslant G$ be a closed subgroup of the compact Lie group $G$, and let $L = T_{eH}(G/H)$ denote the tangent $H$-representation of $G/H$. Choose an embedding $i\colon G/H \hookrightarrow V$ of $G/H$ into a finite-dimensional orthogonal $G$-representation $V$, and let $W := V - (di)_{eH}(L)$ denote the orthogonal complement of the image of $L$ in $V$. By scaling, we may assume that the map $j\colon G \times_H D(W) \to V$ given by $[g,w] \mapsto g \cdot (v_0 + w)$ is an embedding, where $D(W)$ is the unit disc of $W$; see \cite[Ch.0, Thm. 5.2, Ch. II, Cor.~5.2]{Bredon1972Introduction} for proofs that these choices are possible. The map $j$ gives rise to a $G$-equivariant collapse map $c\colon S^V \to G_+ \wedge_H S^W$. Passing to genuine $G$-spectra and smashing with $S^{-V}$ then gives a map of genuine $G$-spectra
	\[
	\eta\colon \S_G \to G_+ \wedge_H (S^{-V} \wedge S^W) \simeq G_+ \wedge_H S^{-L}.
	\]
    Working in the $G$-category $\ul{\Sp}^G$, we may regard $S^{-L}$ as an object of $\Sp^H = \ul{\Sp}^G(G/H \times 1)$, and similarly considering $\S_H \in \ul{\Sp}^G(1 \times G/H)$ we see that $G_+ \wedge_H S^{-L}$ is equivalent to the composition product $S^{-L} \odot \S_H \in \ul{\Sp}^G(1 \times 1) = \Sp^G$.
\end{construction}

\begin{theorem}[{\cite[Theorem~18.6.5]{maysigurdsson2006parametrized}}]
	\label{thm:MaySigurdssonCWDuality}
	Let $H \leqslant G$ be a closed subgroup of a compact Lie group $G$. Then the map $\eta\colon \S_G \to S^{-L} \odot \S_H$ from \Cref{cons:EvaluationMapCWDualityGH} exhibits $S^{-L} \in \Sp^H$ as left Costenoble-Waner dual to $\S_H \in \Sp^H$. \qed
\end{theorem}

\begin{warning}
	May and Sigurdsson use different foundations on parametrized stable homotopy theory than we do, based on orthogonal spectrum objects in retractive topological spaces over $B$. For this reason, we need to be careful when citing results from \cite{maysigurdsson2006parametrized}. Although one cannot strictly speaking cite \cite[Theorem~18.6.5]{maysigurdsson2006parametrized} in the case of \Cref{thm:MaySigurdssonCWDuality}, one observes that their proof carries through verbatim in our setting: May and Sigurdsson construct the relevant duality data already at the level of topological $G$-spaces using a notion of \textit{$V$-duality}, see \cite[Section~18.6]{maysigurdsson2006parametrized}, and the same commutative diagrams prove \Cref{thm:MaySigurdssonCWDuality}.
\end{warning}

\begin{remark}
    \label{rmk:Description_Costenoble_Waner_Duality_For_Genuine_Spectra}
    Using \Cref{cor:Costenoble_Waner_Duality_Preserved_Under_Colimits}, we see that the previous theorem is a parametrized strengthening of the well-known statement that the genuine $G$-spectrum $\Sigma^{\infty}_+ G/H$ is dualizable with dual given by $G_+ \wedge_H S^{-L}$. The corresponding twisted norm map takes the form $\ind^G_H(- \otimes S^{-L}) \iso \coind^G_H(-)$, and is known as the \textit{Wirthmüller isomorphism}.
\end{remark}

We can now prove our main result.

\begin{theorem}
	\label{thm:CharacterizationGStableGCategories}
	Let $G$ be a compact Lie group and let $\Cc$ be a fiberwise stable presentably symmetric monoidal $G$-category. Let $F'\colon \ulSpc^G_* \to \Cc$ be the unique symmetric monoidal left adjoint provided by \Cref{prop:UniversalPropertyGSpectraFormalInversion}(2). Then the following conditions are equivalent:
	\begin{enumerate}[(1)]
		\item The functor $F'(1)\colon \Spc^G_* \to \Cc(1)$ inverts the representation sphere $S^V$ for every $G$-representation $V$;
		\item The $G$-functor $F'$ extends\footnote{If an extension exists, it is necessarily unique.}
		to a symmetric monoidal left adjoint $F''\colon \ulSp^G \to \Cc$;
		\item For every pair of closed subgroups $K \leqslant H \leqslant G$, the map of $G$-spaces $G/K \to G/H$ is twisted $\Cc$-ambidextrous;
		\item For every closed subgroup $H \leqslant G$, the orbit $G/H$ is twisted $\Cc$-ambidextrous;
		\item Every compact $G$-space is twisted $\Cc$-ambidextrous;
		\item The functor $F'(1)\colon \Spc^G_* \to \Cc(1)$ sends compact pointed $G$-spaces to dualizable objects.
	\end{enumerate}
	If the group $G$ is finite, these conditions are moreover equivalent to:
	\begin{enumerate}[(7)]
		\item[(7)] The $G$-category $\Cc$ is $G$-semiadditive, in the sense of \cite[Definition~5.3]{nardin2016exposeIV};
		\item[(8)] The $G$-category $\Cc$ is $G$-stable, in the sense of \cite[Definition~7.1]{nardin2016exposeIV}.
	\end{enumerate}
\end{theorem}

The condition in part (4) of the theorem says that for every closed subgroup $H \leqslant G$, the twisted norm map $\ind^G_H(- \otimes D_{G/H}) \to \coind^G_H(-)$ is an equivalence of $G$-functors $\Cc^{G/H} \to \Cc$. In light of \Cref{rmk:Description_Costenoble_Waner_Duality_For_Genuine_Spectra}, we may think of this condition as saying that \textit{$\Cc$ admits formal Wirthmüller isomorphisms}.

\begin{proof}
	The implication (1) $\implies$ (2) was treated in \Cref{prop:UniversalPropertyGSpectraFormalInversion}. For the implication (2) $\implies$ (3), it suffices by \Cref{cor:SymMonFunctorsPreserveTwistedAmbidexterity} to show that the morphism $G/K \to G/H$ is twisted $\ulSp^G$-ambidextrous. Identifying the slice $\Spc^G_{/(G/H)}$ with the $\infty$-category of $H$-spaces, the morphism $G/K \to G/H$ corresponds to the $H$-space $H/K$, and thus we need to show that $H/K$ is twisted ambidextrous for the $H$-category $\pi_H^*\ulSp^G \simeq \ulSp^H$. Using \Cref{prop:Twisted_Ambidexterity_VS_Costenoble_Waner_Dualizability}, this is an instance of \Cref{thm:MaySigurdssonCWDuality}, applied to $K \leqslant H$. The fact that (3) implies (4) is clear. The implication (4) $\implies$ (5) follows from \Cref{cor:Twisted_Ambidextrous_Objects_Closed_Under_Finite_Limits}, since every compact $G$-space is a retract of a finite $G$-CW-complex, and finite $G$-CW-complexes are built from the orbits $G/H$ using finite colimits. The implication (5) $\implies$ (6) holds by \Cref{cor:ColimitOfCWDblObjectIsDbl}, as $F'(1)(B)$ is the cofiber of the map $\unit \to B_!\unit_B$ and dualizable objects in $\Cc(1)$ are closed under cofibers. The implication (6) $\implies$ (1) holds by  \Cref{thm:CampionInvertingSpheresVsDualizingSpheres}. This shows that conditions (1)-(6) are equivalent.
	
	If $G$ is a finite group, conditions (7) and (8) are equivalent since $\Cc$ is already assumed to be fiberwise stable. Since $\Cc$ is in particular fiberwise semiadditive, it follows from \Cref{cor:ParametrizedSemiadditivityAsTwistedAmbidexterity} that conditions (3) and (7) are equivalent. This finishes the proof.
\end{proof}

\begin{theorem}
	\label{thm:MainResult}
	Let $G$ be a compact Lie group. Then the $G$-category $\ulSp^G$ is initial among fiberwise stable presentably symmetric monoidal $G$-categories $\Cc$ such that all compact $G$-spaces are twisted $\Cc$-ambidextrous.
\end{theorem}
\begin{proof}
	This is immediate from the equivalence between (2) and (5) in \Cref{thm:CharacterizationGStableGCategories}.
\end{proof}

\begin{theorem}
	Let $G$ be a finite group. Then the $G$-category $\ulSp^G$ is the initial presentably symmetric monoidal $G$-category which is $G$-stable in the sense of \cite{nardin2016exposeIV}.
\end{theorem}
\begin{proof}
	This is immediate from the equivalence between (2) and (8) in \Cref{thm:CharacterizationGStableGCategories}.
\end{proof}

\subsection{Orbispectra}
\label{subsec:Orbispectra}

In this subsection, we will study a global analogue of the results established in the previous subsection: instead of working only with subgroups of a fixed compact Lie group $G$, we work with an indexing category $\Orb$ containing all compact Lie groups and \textit{injective} group homomorphisms.

\subsubsection{The global orbit category}
We start by recalling the definition of the global orbit category $\Orb$.

\begin{definition}
	\label{def:TopologicalGroupoids}
	Let $\TopGrpd_1$ denote the ordinary category of topological groupoids. It is naturally enriched over itself via the internal mapping objects. Using the finite product preserving geometric realization functor
	\[
	\abs{-}\colon \TopGrpd_1 \hookrightarrow \sTop \xrightarrow{\abs{-}} \Top,
	\]
	we obtain a topological enrichment on $\TopGrpd_1$. Explicitly, the geometric realization $\abs{\Gg}$ of a topological groupoid $\Gg$ is given by the coend
	\[
	\abs{\Gg} = \int^{[n] \in \Delta} \Delta^n \times \Gg_n \qin \Top,
	\]
	where $\Delta^n$ denotes the topological $n$-simplex and where $\Gg_n = \Gg_1 \times_{\Gg_0} \Gg_1 \times_{\Gg_0} \dots \times_{\Gg_0} \Gg_1$. We let $\TopGrpd$ denote the homotopy coherent nerve of the topologically enriched category $\TopGrpd_1$.
\end{definition}

\begin{definition}[Global indexing category]
	\label{def:GlobalIndexingCategory}
	Given a topological group $G$, we let $\bbB G$ denote the associated one-object topological groupoid. We define the $\infty$-category $\Glo$ as the full subcategory of $\TopGrpd$ spanned by the topological groupoids $\bbB G$ for compact Lie groups $G$. We call $\Glo$ the \textit{global indexing category}.
\end{definition}

Given compact Lie groups $H$ and $G$, the mapping space $\Hom_{\Glo}(\bbB H,\bbB G)$ is the homotopy type of the geometric realization of the topological groupoid $\mathbb{H}\mathrm{om}(\bbB H, \bbB G)$. This topological groupoid is equivalent to the action groupoid of the topological $G$-space $\Hom_{\mathrm{Lie}}(H,G)$ of continuous group homomorphisms from $H$ to $G$, with $G$-action given by conjugation. As a consequence, we get an identification
\[
\Hom_{\Glo}(\bbB H,\bbB G) \simeq \Hom_{\mathrm{Lie}}(H,G)_{hG} \qin \Spc.
\]

\begin{definition}[Global orbit category]
	\label{def:GlobalOrbitCategory}
	We let $\Orb \subseteq \Glo$ denote the wide subcategory, whose morphism spaces
	\[
	\Hom_{\Orb}(\bbB H, \bbB G) \subseteq \Hom_{\Glo}(\bbB H,\bbB G)
	\]
	consist of those components corresponding to \textit{injective} group homomorphisms $H \hookrightarrow G$. We call $\Orb$ the \textit{global orbit category}. We denote the presheaf category $\PSh(\Orb)$ by $\Orb\Spc$ and refer to this as the $\infty$-category of \textit{orbispaces}. We refer to $\Orb\Spc$-categories as \textit{orbicategories}.
\end{definition}

\begin{remark}
	Our definitions of $\Glo$ and $\Orb$ agree with those of \cite[2.2]{rezk2014global} and \cite[Definition~6.1]{linskensNardinPol2022Global}. A close analogue of the definition was originally given by Gepner and Henriques \cite[Section~4.1]{GepnerHenriques2007Orbispaces} for arbitrary topological groups, where both $\Glo$ and $\Orb$ were denoted as $\Orb$. A slight difference with the definition of \cite{GepnerHenriques2007Orbispaces} is that they use the \textit{fat realization} $\abs{\abs{-}}\colon \TopGrpd \to \Top$ as opposed to the usual (`thin') geometric realization. As $\abs{\abs{-}}$ does not preserve finite products on the nose, defining composition is slightly subtle, but when taking sufficient care the resulting $\infty$-categories will be equivalent; see \cite[Remark~3.10]{korschgen2018} for a more detailed discussion.
\end{remark}

Crucial for the comparison with equivariant homotopy theory for a compact Lie group $G$ is the statement that the slice of $\Orb$ over $\bbB G$ is equivalent to the orbit category $\Orb_G$ of $G$. This was proved at the level of simplicially enriched categories by \cite[Proposition~2.15]{gepnerMeier2020equivariant}, and at the level of $\infty$-categories by \cite[Lemma~6.13]{linskensNardinPol2022Global}. It follows that there is an equivalence of $\infty$-categories
\[
\Spc^G \simeq \Orb\Spc_{/\bbB G}
\]
between the $\infty$-category of $G$-spaces and the $\infty$-category of orbispaces over $\bbB G$. Given a $G$-space $A$, we denote its associated orbispace by $A//G$, so that $*//G = \bbB G$. By restricting along the functor $-//G\colon \Spc_G \to \Orb\Spc$, any orbicategory $\Cc$ has an underlying $G$-category which we will denote by $\pi_G^*\Cc$.

\subsubsection{The orbicategory of orbispectra}
We will now define the orbicategory of orbispectra, informally given by the assignment $\bbB G \mapsto \Sp^G$, and prove its universal property in terms of twisted ambidexterity.

\begin{definition}
	\label{def:OrbicategoryOfOrbispaces}
	Define the presentably symmetric monoidal orbicategories $\ulOrbSpc$ and $\ulOrbSpc_*$ of \textit{orbispaces} resp.\ \textit{pointed orbispaces} as
	\[
	\ulOrbSpc := \Omega_{\Orb\Spc}, \qquad \qquad \ulOrbSpc_* := \Orb\Spc_* \otimes_{\Orb\Spc} \Omega_{\Orb\Spc},
	\]
	using the fully faithful embedding $- \otimes_{\Orb\Spc} \Omega_{\Orb\Spc}\colon \CAlg_{\Orb\Spc}(\PrL) \hookrightarrow \CAlg(\PrL_{\Orb})$ from \Cref{prop:Embedding_BModules_Into_Presentable_BCategories}. Explicitly, they are given at an orbispace $B$ by
	\[
	\ulOrbSpc(B) = \Orb\Spc_{/B}, \qquad \qquad \ulOrbSpc_*(B) = (\Orb\Spc_{/B})_*.
	\]
\end{definition}

\begin{definition}
	\label{def:Orbispectra}
	We let $S \subseteq \ulOrbSpc_*$ denote the subcategory spanned by those objects $X \in (\Orb\Spc_{/B})_*$ whose restriction along any map $\bbB G \to B$ corresponds to a $G$-representation sphere in $(\Orb\Spc_{/\bbB G})_* \simeq \Spc^G_*$. We define the presentably symmetric monoidal orbicategory $\ulOrbSp$ of \textit{orbispectra} as
	\[
	\ulOrbSp := \Ll(\ulOrbSpc_*, S) \in \CAlg(\PrL_{\Orb}),
	\]
	using \Cref{cons:PointwiseFormalInversion}. We define the orbifunctor $\Sigma^{\infty}\colon \ulOrbSpc_* \to \ulOrbSp$ as
	\[
	\Sigma^{\infty}\colon \ulOrbSpc_* = \Ll(\ulOrbSpc_*,\emptyset) \to \Ll(\ulOrbSpc_*,S) = \ulOrbSp.
	\]
	We let $\Orb\Sp$ denote the underlying $\infty$-category of $\ulOrbSp$, referred to as the \textit{$\infty$-category of orbispectra}. For an orbispace $B$, we also write $\Orb\Sp(B)$ for $\ulOrbSp(B)$ and call it the \textit{$\infty$-category of orbispectra parametrized over $B$}.
\end{definition}

Pardon \cite{Pardon2023Orbifold} has previously defined a notion of orbispectra in the setting of topological stacks. Although his definition seems close in spirit to our definition, the precise mathematical connection is not known to the author.

\begin{proposition}
	The orbicategory $\ulOrbSp$ is presentably symmetric monoidal and the orbifunctor $\Sigma^{\infty}\colon \ulOrbSpc_* \to \ulOrbSp$ exhibits it as the formal inversion of the representation spheres in $\ulOrbSpc_*$. 
\end{proposition}
\begin{proof}
	This is an instance of \Cref{prop:PointwiseFormalInversalIsParametrizedLocalInversion}. The condition $(*)$ of that proposition is satisfied: for a closed subgroup $H \leqslant G$, every irreducible $H$-representation is a direct summand of the restriction to $H$ of an irreducible $G$-representation, see Bröcker and tom Dieck \cite[Theorem~4.5]{BroeckerTomDieck1985Representations}.
\end{proof}

The orbicategory $\ulOrbSp$ recovers the $G$-category $\ulSp^G$ of genuine $G$-spectra for every compact Lie group $G$:

\begin{lemma}
	\label{lem:orbispectraRestrictToGenuineGSpectra}
	For every compact Lie group $G$, the $G$-category $\pi_G^*\ulOrbSp$ underlying the orbicategory $\ulOrbSp$ of orbispectra is equivalent to the $G$-category $\ulSp^G$ of genuine $G$-spectra.
\end{lemma}
\begin{proof}
	Both come equipped with a symmetric monoidal left adjoint from $\ulSpc^G_*$ exhibiting them as formal inversions of the $G$-representation spheres. For $\ulSp^G$ this is by \Cref{prop:InvertingObjectsGloballySuffices}, while for $\pi_G^*\ulOrbSp$ this is by \Cref{prop:Invert_Objects_Locally}.
\end{proof}

\begin{remark}
	\label{rkm:ComparisonDiagramWithLNP}
	The functor $\ulOrbSp\colon \Orb\catop \to \CAlg(\PrL)$ is the restriction along $\Orb \hookrightarrow \Glo$ of the functor $\Sp_{\bullet}\colon \Glo\catop \to \CAlg(\PrL)$ constructed by \cite[Section~10]{linskensNardinPol2022Global}. 
	Indeed, they construct a natural transformation $\Sigma^{\infty}_{\bullet} \colon \Ss_{\bullet,*} \to \Sp_{\bullet}$, where the functor $\Ss_{\bullet,*}\colon \Glo\catop \to \Cat_{\infty}$ constructed in \cite[Construction~6.16]{linskensNardinPol2022Global} restricts to the functor $\ulOrbSpc_*$ using the natural equivalence $\PSh(\Orb_{/-}) \simeq \PSh(\Orb)_{/-}$. Furthermore, it is shown in \cite[Proposition~10.5]{linskensNardinPol2022Global} that this transformation $\Sigma^{\infty}_{\bullet}$ is pointwise given by the standard suspension spectrum functor $\Spc^G_* \to \Sp^G$, so that it exhibits its target as a pointwise formal inversion of its source.
\end{remark}

From the results of \Cref{subsec:UniversalityOfWirthmullerIsomorphisms} we may deduce a universal property of $\ulOrbSp$ in terms of twisted ambidexterity. To this end, recall that a morphism $f\colon A \to B$ in an $\infty$-topos is called \textit{relatively compact} if for every compact object $K$ in $\Bb$ and every morphism $K \to B$, the pullback $A \times_B K \to K$ is a compact object in the slice $\Bb_{/K}$ (or equivalently in $\Bb$ by \cite[Lemma~3.1.5]{GHK2022Analytic}). If $\Bb = \PSh(T)$ is a presheaf topos, it suffices to check this when $K$ is a representable object. In particular, a morphism of orbispaces $f\colon A \to B$ is relatively compact if and only if for every compact Lie group $G$ and every map $\bbB G \to B$ of orbispaces, the pullback $A \times_B \bbB G$ corresponds to a compact $G$-space.

\begin{proposition}
	\label{prop:CharacterizationOrbStableOrbCategories}
	Let $\Cc$ be a fiberwise stable presentably symmetric monoidal orbicategory. Then the following conditions are equivalent:
	\begin{enumerate}[(1)]
		\item The unique symmetric monoidal left adjoint $F'\colon \ulOrbSpc_* \to \Cc$ inverts the representation spheres: for every compact Lie group $G$ and every $G$-representation $V$, $F(S^V) \in \Cc(\bbB G)$ is invertible;
		\item The functor $F'\colon \ulOrbSpc_* \to \Cc$ extends (necessarily uniquely) to a symmetric monoidal left adjoint $F'\colon \ulOrbSp \to \Cc$;
		\item For every compact Lie group $G$ and every compact $G$-space $A$, the map $A//G \to *//G = \bbB G$ of orbispaces is twisted $\Cc$-ambidextrous;
		\item Every relatively compact morphism $f\colon A \to B$ of orbispaces is twisted $\Cc$-ambidextrous.
	\end{enumerate}
\end{proposition}
\begin{proof}
	The equivalence between (1) and (2) is immediate from the universal property of $\ulOrbSp$. Note that condition (1) is satisfied if and only if each of the $G$-functors $\pi_G^*F' \colon \pi_G^*\ulOrbSpc_* \to \pi_G^*\Cc$ satisfies condition (1) of \Cref{thm:CharacterizationGStableGCategories}, while (3) is satisfied if and only if the $G$-category $\pi_G^*\Cc$ satisfies condition (5) of \Cref{thm:CharacterizationGStableGCategories}, so the equivalence between (1) and (3) holds by applying \Cref{thm:CharacterizationGStableGCategories} to the $G$-category $\pi_G^*\Cc$ for every $G$. 
	The equivalence between (3) and (4) holds by \Cref{prop:ClosurePropertiesTwistedAmbiMaps}\eqref{it:LocalClass} and the above characterization of relatively compact morphisms in $\Orb \Spc$.
\end{proof}

\begin{theorem}\label{thm:UniversalPropertyOrbispectra}
	The orbicategory $\ulOrbSp$ is initial among fiberwise stable presentably symmetric monoidal orbicategories $\Cc$ such that every relatively compact morphism of orbispaces is twisted $\Cc$-ambidextrous.
\end{theorem}
\begin{proof}
	This is immediate from the equivalence between (2) and (4) in the previous proposition.
\end{proof}

\subsection{Proper equivariant stable homotopy theory}
\label{subsec:ProperEquivariantHomotopyTheory}
For a Lie group $G$, not assumed to be compact, Degrijse et al.\ \cite{DHLPS2019Proper} introduced an $\infty$-category $\Sp^G$ of \textit{proper genuine $G$-spectra}. In this subsection, we will see that this $\infty$-category can be identified with the $\infty$-category of orbispectra parametrized over a certain orbispace $\bbB G$. As an application, we show that when $G$ has \textit{enough bundle representations}, the $\infty$-category $\Sp^G$ may be obtained from the $\infty$-category of pointed proper $G$-spaces by inverting the sphere bundles $S^{\xi}$ associated to finite-dimensional vector bundles $\xi$ over $\bbB G$, see \Cref{prop:EnoughBundleRepsImpliesFormalInversion}.

\begin{definition}
	\label{def:ProperOrbitCategory}
	For a Lie group $G$, we define its \textit{proper orbit category} as the full subcategory $\Orb^{\prop}_G \subseteq \Orb_G$ spanned by the orbits $G/K$ for compact subgroups $K \leqslant G$. We define the $\infty$-category $\Spc^G_{\prop}$ of \textit{proper $G$-spaces} as the presheaf category $\PSh(\Orb^{\prop}_G)$. A \textit{proper $G$-category} is a $\Spc^G_{\prop}$-category, equivalently encoded by a functor $(\Orb_G^{\prop})\catop \to \Cat_{\infty}$.
\end{definition}

We start by identifying the $\infty$-category of proper $G$-spaces with a slice of the $\infty$-category of orbispaces.

\begin{definition}[Classifying orbispace of a Lie group]
	In analogy with $\Orb \subseteq \TopGrpd$, we define the $\infty$-category $\Orb' \subseteq \TopGrpd$ as the (non-full) subcategory whose objects are the one-point topological groupoids $\bbB G$ for (not necessarily compact) Lie groups $G$, and whose mapping spaces
	\[
	\Hom_{\Orb'}(\bbB H, \bbB G) \subseteq \Hom_{\TopGrpd}(\bbB H, \bbB G) \simeq \Hom_{\mathrm{Lie}}(H,G)_{hG}
	\]
	consist of those path components corresponding to \textit{injective} continuous group homomorphisms $H \to G$. It is immediate that $\Orb'$ contains $\Orb$. Given a Lie group $G$, we define its \textit{classifying orbispace} $\bbB G$ as the composite
	\[
	\bbB G\colon \Orb\catop \hookrightarrow (\Orb')\catop \xrightarrow{\Hom_{\Orb'}(-,\bbB G)} \Spc.
	\]
	Note that when $G$ is compact, this is just the representable presheaf on $\bbB G \in \Orb$.
\end{definition}

\begin{construction}
	For a Lie group $G$, let $\Orb_{/\bbB G}$ denote the full subcategory of $\Orb'_{/\bbB G}$ spanned by the morphisms of the form $\bbB K \to \bbB G$ for compact subgroups $K \leqslant G$. At the level of topological categories, one can define a topologically enriched functor from $\Orb_G$ to $\Orb'$ by sending an orbit $G/H$ to $\bbB H$, which induces a functor of $\infty$-categories between their homotopy coherent nerves. As $\Orb_G$ admits a terminal object $G/G$, this canonically gives rise to a functor $\Orb_G \to \Orb'{/\bbB G}$, which is easily seen to restrict to a functor $\Orb^{\prop}_G \to \Orb_{/\bbB G}$.
\end{construction}

\begin{lemma}
	\label{lem:SliceOfOrbLie}
	For a Lie group $G$, the above functor $\Orb^{\prop}_G \to \Orb_{/\bbB G}$ is an equivalence.
\end{lemma}
\begin{proof}
	We will argue just like \cite[Lemma~6.13]{linskensNardinPol2022Global}. The functor is essentially surjective by definition, so we must show it is fully faithful. Consider two objects $G/H$ and $G/K$ in $\Orb^{\prop}_G$, where we (non-canonically) choose representatives $H, K \subseteq G$ of the conjugacy classes $[H]$ and $[K]$ of subgroups of $G$. We have to show that the square
	\[
	\begin{tikzcd}
		(G/H)^K \dar \rar & \Hom_{\TopGrpd}(\bbB K, \bbB H) \dar \rar{\simeq} & \Hom(H,K)_{hK} \dar \\
		* \rar & \Hom_{\TopGrpd}(\bbB K, \bbB G) \rar{\simeq} & \Hom_{\mathrm{Lie}}(H,G)_{hG}
	\end{tikzcd}
	\]
	is homotopy cartesian. The argument for this is identical to that of \cite[Lemma~6.13]{linskensNardinPol2022Global}, so we will not repeat it here. Note that the quotient map $G \to G/C(H)$ used in that proof is still a fibration, as it is a locally trivial fiber bundle by \cite[Corollary~4.1]{Palais1961Slices}.
\end{proof}

\begin{corollary}
	\label{cor:OrbispacesOverBGAreProperGSpaces}
	For every Lie group $G$, there is an equivalence $\Orb\Spc_{/\bbB G} \simeq \Spc_{\prop}^G$ between the $\infty$-category of orbispaces over $\bbB G$ and the $\infty$-category of proper $G$-spaces.
\end{corollary}
\begin{proof}
	By \Cref{lem:SliceOfOrbLie}, there is an equivalence $\Spc_{\prop}^G = \PSh(\Orb^{\prop}_G) \simeq \PSh(\Orb_{/\bbB G})$. The $\infty$-category $\Orb_{/\bbB G}$ is equivalent to the subcategory of $\Orb\Spc_{/\bbB G}$ spanned by the maps $\bbB K \to \bbB G$ for compact Lie groups $K$, since both embed fully faithfully into $\PSh(\Orb')_{/\bbB G}$ with the same image. It follows that there is an equivalence $\PSh(\Orb_{/\bbB G}) \simeq \Orb\Spc_{/\bbB G}$, finishing the proof.
\end{proof}

\begin{definition}
	Let $G$ be a Lie group. By restricting along the forgetful functor $\Spc^G_{\prop} \simeq \Orb\Spc_{/\bbB G} \xrightarrow{\fgt} \Orb\Spc$, any orbicategory $\Cc$ gives rise to a proper $G$-category $\pi_G^*$. We define the proper $G$-category of \textit{proper genuine $G$-spectra} $\ulSp^G$ as
	\[
	\ulSp^G := \pi_G^*\ulOrbSp,
	\]
	the underlying proper $G$-category of the orbicategory of orbispectra. When $G$ is compact this agrees with the $G$-$\infty$-category of genuine $G$-spectra $\ulSp^G$ by \Cref{lem:orbispectraRestrictToGenuineGSpectra}.
\end{definition}

\begin{proposition}[{Linskens-Nardin-Pol \cite[Theorem~12.11]{linskensNardinPol2022Global}}]
	For every Lie group $G$, the underlying $\infty$-category of the proper $G$-category $\ulSp^G$ is equivalent to the $\infty$-category $\Sp^G$ of proper genuine $G$-spectra defined by \cite{DHLPS2019Proper}.
\end{proposition}
\begin{proof}
	The underlying $\infty$-category of $\ulSp^G$ is given as the limit of the functor $\ulSp^G\colon (\Orb_G^{\prop})\catop \to \Cat_{\infty}$. By \Cref{rkm:ComparisonDiagramWithLNP}, this functor is equivalent to the composite
	\[
	(\Orb_G^{\prop})\catop \simeq \Orb_{/\bbB G}\catop \to \Orb\catop \hookrightarrow \Glo\catop \xrightarrow{\Sp_{\bullet}} \Cat_{\infty}.
	\]
	The limit of this diagram was shown by \cite[Theorem~12.11]{linskensNardinPol2022Global} to be equivalent to $\Sp^G$, finishing the proof.
\end{proof}

\begin{corollary}
	The proper $G$-category $\ulSp^G$ satisfies twisted ambidexterity for all relatively compact morphisms of proper $G$-spaces. In particular, if $H \leqslant G$ is a cocompact subgroup, meaning that $G/H$ is a compact topological space, then there is a formal Wirthmüller isomorphism
	\[
	\ind^G_H(- \otimes D_{G/H}) \simeq \coind^G_H(-) \colon \Sp^H \to \Sp^G.
	\]
\end{corollary}
\begin{proof}
	The first statement is immediate as $\ulSp^G$ is the restriction of the orbicategory $\ulOrbSp$ to the slice $\Orb_{/\bbB G}$ and $\ulOrbSp$ satisfies twisted ambidexterity for all relatively compact morphisms of orbispaces. The second statement follows from the observation that the map of orbispaces $\bbB H \to \bbB G$ is relatively compact, by compactness of $G/H$.
\end{proof}

\begin{theorem}
	\label{thm:UniversalPropertyProperGspectra}
	The proper $G$-category $\ulSp^G$ is the initial fiberwise stable presentably symmetric monoidal proper $G$-category satisfying twisted ambidexterity for all relatively compact morphisms of proper $G$-spaces.
\end{theorem}
\begin{proof}
	Just like in \Cref{prop:CharacterizationOrbStableOrbCategories}, one deduces from \Cref{thm:CharacterizationGStableGCategories} that twisted ambidexterity for relatively compact morphisms of proper $G$-spaces is equivalent to the invertibility of representation spheres. The claim follows, since $\ulSp^G$ is a formal inversion of representation spheres in $\ulSpc^G_*$ by \Cref{prop:Invert_Objects_Locally}.
\end{proof}

\subsubsection{Enough bundle representations}
Our results on formal inversions allow us to prove that in certain cases the $\infty$-category of proper genuine $G$-spectra may be obtained from the $\infty$-category of pointed proper $G$-spaces by inverting the sphere bundles of vector bundles over $\bbB G$.

\begin{definition}
	\label{def:EnoughBundleRepresentations}
	Let $\Rep\colon \Orb\catop \to \Cat_{\infty}$ denote the functor which sends $\bbB G$ to the ordinary category of finite-dimensional $G$-representations. We may limit-extend this to a functor
	\[
	\Vect\colon \Orb\Spc\catop \to \Cat_{\infty},
	\]
	and we refer to $\Vect(B)$ as the category of \textit{vector bundles over $B$}. We say that an orbispace $B$ \textit{has enough bundle representations} if for any compact Lie group $G$, any map of orbispaces $\bbB G \to B$ and any $G$-representation $V$, there exists a vector bundle $\xi \in \Vect(B)$ such that the restriction $\xi\vert_{\bbB G} \in \Vect(\bbB G) = \Rep(G)$ contains $V$ as a direct summand.
\end{definition}

\begin{example}
	For every compact Lie group $G$, $\bbB G$ has enough bundle representations by Bröcker and tom Dieck \cite[Theorem~4.5]{BroeckerTomDieck1985Representations}.
\end{example}

\begin{example}
	Let $G$ be a discrete group and assume that $\bbB G$ is a finite orbispace, that is, it lies in the subcategory of $\Orb\Spc$ generated under finite colimits by the $\bbB K$ for compact Lie groups $K$. Then $\bbB G$ has enough bundle representations. Indeed, under the identification $\Orb\Spc_{/\bbB G} \simeq \Spc^G_{\prop}$, the orbispace $\bbB G$ corresponds to the universal proper $G$-space $\underline{EG}$, and by the assumption this is a finite proper $G$-CW-complex. Given a finite subgroup $K \leqslant G$, the unique map $\phi\colon G/K \to \underline{EG}$ is a map of finite proper $G$-CW-complexes, hence by Lück and Oliver \cite[Lemma~3.7]{LuckOliver2001Completion} any $K$-representation is a direct summand of the restriction along $\phi$ of a $G$-vector bundle $V$ over $\underline{EG}$. Since $V$ in particular gives rise to a vector bundle over the orbispace $\bbB G$, this finishes the proof.
\end{example}

For every compact Lie group $G$ and every $G$-representation $V$, we may consider its associated representation sphere $S^V \in \Spc_{G,*} \simeq \ulOrbSpc(\bbB G)_*$. Since this assignment is functorial in both $V$ and $G$, we obtain a natural transformation $S^{(-)}\colon \Rep \to \ulOrbSpc(\bbB -)_*$ of functors $\Orb\catop \to \Cat_{\infty}$, which uniquely extends to a natural transformation $S^{(-)}\colon \Vect \to \ulOrbSpc_*$ of functors $\Orb\Spc\catop \to \Cat_{\infty}$.

\begin{proposition}
	\label{prop:EnoughBundleRepsImpliesFormalInversion}
	Assume the orbispace $B$ has enough bundle representations. Then the $\infty$-category $\Orb\Sp(B)$ of orbispectra parametrized over $B$ is equivalent to the formal inversion of sphere bundles $\{S^{\xi} \mid \xi \in \Vect(B)\}$ in the $\infty$-category $(\Orb\Spc_{/B})_*$ of retractive orbispaces over $B$:
	\[
	\Orb\Sp(B) \simeq (\Orb\Spc_{/B})_*[\{S^{\xi}\}^{-1}]. \qednow
	\]
\end{proposition}
\begin{proof}
	Let $\Bb = \Orb\Spc_{/B}$ be the $\infty$-topos of orbispaces over $B$. By \Cref{prop:Invert_Objects_Locally}, the $\Bb$-functor $\pi_B^*\ulOrbSpc_* \to \pi_B^*\ulOrbSp$ is a formal inversion of the representation spheres $S^V \in \Spc_*^G \simeq \pi_B^*\ulOrbSpc_*(\bbB G)$ for every compact Lie group $G$ and a map of orbispaces $\bbB G \to B$. By the assumption that $B$ has enough bundle representations, the parametrized subcategory of representation spheres is generated (in the sense of \Cref{def:SubcategoryGenerated}) by the objects $S^{\xi} \in \Spc_*^B$ for all vector bundles $\xi \in \Vect(B)$. It thus follows from \Cref{obs:GloballyInvertingMeansPointwiseInverting} that the underlying functor $(\Orb\Spc_{/B})_* \to \Orb\Sp(B)$ of this $\Bb$-functor is a formal inversion of the objects $S^{\xi}$, finishing the proof.
\end{proof}

\begin{corollary}
	\label{cor:EnoughBundleRepsImpliesFormalInversion}
	Assume that $G$ is a Lie group which has enough bundle representations. Then the $\infty$-category of proper genuine $G$-spectra is the formal inversion of the $\infty$-category proper pointed $G$-spaces by inverting the sphere bundles $S^{\xi}$ associated to all finite-dimensional vector bundles $\xi$ over $\bbB G$:
	\[
	\Sp^G \simeq (\Spc^G_*)[\{S^{\xi}\}^{-1}]. \qedhere
	\]
\end{corollary}

\appendix

\section{Symmetric monoidal unstraightening}
\label{chap:SymmericMonoidalUnstraightening}

Let $\Bb$ be an $\infty$-topos and let $\Oo^{\otimes}$ be an $\infty$-operad. In this appendix, we recall unstraightening techniques from Lurie \cite{lurie2016HA} and Drew and Gallauer \cite[Appendix~A]{DrewGallauer2022Universal} to describe $\Oo$-algebras in the $\infty$-category $\Cat(\Bb)$ of $\Bb$-categories in terms of suitable cocartesian fibrations.

Recall from \cite[2.1.2.13]{lurie2016HA} that an \textit{$\Oo^{\otimes}$-monoidal $\infty$-category} is an $\infty$-category $\Cc^{\otimes}$ equipped with a cocartesian fibration $p^{\otimes}\colon \Cc^{\otimes} \to \Oo^{\otimes}$ such that the composition $\Cc^{\otimes} \to \Oo^{\otimes} \to \Fin_*$ exhibits $\Cc^{\otimes}$ as an $\infty$-operad. An \textit{$\Oo$-monoidal functor} is a morphism of operads over $\Oo^{\otimes}$ which preserves $\Oo^\otimes$-cocartesian edges. We let $\Cat_{\infty}^{\Oo^{\otimes}} \subseteq \CAlg(\Cat_{\infty})_{/\Oo^{\otimes}}$ denote the (non-full) subcategory of $\Oo^{\otimes}$-monoidal $\infty$-categories.





\begin{proposition}[{\cite[Corollary~A.12]{DrewGallauer2022Universal}}]
	\label{prop:tensor-un-straightening}
	Let $\Oo^{\otimes} = \Bb^{\op,\sqcup} := (\Bb\catop)^{\sqcup}$ denote the symmetric monoidal category corresponding to $\Bb\catop$ equipped with the cocartesian monoidal structure. Then straightening\,/\,unstraightening induces an equivalence
	\[
	\Cat_{\infty}^{\Bb^{\op,\sqcup}} \simeq \Fun(\Bb\catop,\CAlg(\Cat_{\infty})).
	\]
\end{proposition}

Given a functor $\Cc\colon \Bb\catop \to \CAlg(\Cat_{\infty})$, the resulting  $\Bb^{\op,\sqcup}$-monoidal $\infty$-category $p^\otimes:\Cc^\boxtimes\to \Bb^{\op,\sqcup}$ may informally be described as follows:
\begin{itemize}
	\item The objects of $\Cc^\boxtimes$ are pairs $(B,X)$ where $B$ is an object in $\Bb$, and $X$ is an object in~$\Cc(B)$.
	\item A morphism $(B,X)\to (B',X')$ in $\Cc^\boxtimes$ consists of a morphism $f:B'\to B$ in $\Bb$, and a morphism $f^*X\to X'$ in $\Cc(B')$.
	\item The tensor product of $(B,X)$ and $(B',X')$ is the ``external product''
	\[
	X\boxtimes X':=\pr_B^*X\otimes_{B\times B'}\pr_{B'}^*X' \qin \Cc(B \times B'),
	\]
	where $\pr_B:B\times B'\to B$ and $\pr_{B'}:B\times B'\to B'$ are the canonical projections in $\Bb$.
\end{itemize}

Conversely, if $p^\otimes:\Cc^\boxtimes\to \Bb^{\op,\sqcup}$ is a $\Bb^{\op,\sqcup}$-monoidal $\infty$-category, we may straighten the underlying cocartesian fibration $p:(\Cc^{\boxtimes})_1 \to \Bb\catop$ to a functor $\Cc:\Bb\catop\to\Cat_{\infty}$, sending $B \in \Bb$ to the fiber of $p$ over $B$. The symmetric monoidal structure on this fiber $\Cc(B)$ may be described as follows: given $X,X'\in \Cc(B)$, their tensor product is the object $\Delta^*(X\boxtimes X')$ where $\Delta$ denotes the diagonal map $B\to B\times B$ in $\Bb$.

As there is an equivalence $\Fun(\Bb\catop,\CAlg(\Cat_{\infty})) \simeq \CAlg(\Fun(\Bb\catop,\Cat_{\infty}))$, it follows from \Cref{prop:tensor-un-straightening} that we may regard every symmetric monoidal $\Bb$-category $\Cc \in \CAlg(\Cat(\Bb))$ as a cocartesian fibration $\Cc^\boxtimes \to \Bb^{\op,\sqcup}$. We will now discuss how one may describe $\Oo$-algebras in $\Cat(\Bb)$ for an arbitrary $\infty$-operad $\Oo^{\otimes}$ in a similar fashion.

\begin{definition}
	We define the $\infty$-operad $\Bb^{\op,\Oo}$ via the following pullback diagram:
	\begin{equation*}
		\begin{tikzcd}
			\Bb^{\op,\Oo} \dar[swap]{q} \rar \drar[pullback] & \Bb^{\op,\sqcup} \dar{p} \\
			\Oo^{\otimes} \rar & \Fin_*.
		\end{tikzcd}
	\end{equation*}
	Being the pullback of a cocartesian fibration, the map $q\colon \Bb^{\op,\Oo} \to \Oo^{\otimes}$ is a cocartesian fibration.
\end{definition}

\begin{proposition}
	\label{prop:operadic-un-straightening}
	There is an equivalence
	\begin{align*}
		\Alg_{\Oo}(\Fun(\Bb\catop,\Cat_{\infty})) \simeq \Cat_{\infty}^{\Bb^{\op,\Oo}},
	\end{align*}
	natural in $\Oo^{\otimes}$, which for $\Oo^{\otimes} = \Cc \mathrm{omm}^{\otimes}$ reduces to the equivalence of \Cref{prop:tensor-un-straightening}.
\end{proposition}
\begin{proof}
	This follows from the following equivalences:
	\begin{align*}
		\Alg_{\Oo}(\Fun(\Bb\catop,\Cat_{\infty})) & \simeq \Fun(\Bb\catop,\Alg_{\Oo}(\Cat_{\infty})) \\
		&\simeq \Alg_{\Bb^{\op,\Oo}}(\Cat_{\infty}) &&\text{\cite[Theorem~2.4.3.18]{lurie2016HA}}\\
		&\simeq \Cat_{\infty}^{\Bb^{\op,\Oo}}. &&\text{\cite[Remark~2.4.2.6]{lurie2016HA}}
	\end{align*}
	For $\Oo^{\otimes} = \Cc \mathrm{omm}^{\otimes}$, this reduces to the equivalence of \Cref{prop:tensor-un-straightening} given in \cite[Corollary~A.12]{DrewGallauer2022Universal}.
\end{proof}

We are mainly interested in the case $\Oo^{\otimes} = \Ll\Mm^{\otimes}$ in order to describe left $\Cc$-modules in $\Cat(\Bb)$.

\begin{corollary}
	\label{cor:un-straightening-left-modules}
	Consider $\Cc \in \CAlg(\Cat(\Bb))$ and let $\Cc^{\boxtimes} \in \Cat_{\infty}^{\Bb^{\op,\sqcup}}$ denote the associated $\Bb^{\op,\sqcup}$-monoidal $\infty$-category. The equivalence of \cref{prop:operadic-un-straightening} restricts to an equivalence
	\begin{align*}
		\LMod_{\Cc}(\Fun(\Bb\catop,\Cat_{\infty})) \simeq \Cat_{\infty}^{\Bb^{\op,\Ll\Mm}} \times_{\Cat_{\infty}^{\Bb^{\op,\sqcup}}} \{\Cc^{\boxtimes}\}.
	\end{align*}
	The full subcategory $\LMod_{\Cc}(\Cat(\Bb))$ of the left-hand side is equivalent to a full subcategory of the right-hand side spanned by those $\Bb^{\op,\Ll\Mm}$-monoidal $\infty$-categories $\Mm^{\boxtimes} \to \Bb^{\op,\Ll\Mm}$ whose pullback along the inclusion $\Bb\catop \hookrightarrow \Bb^{\op,\Ll\Mm}$ is corresponds to a $\Bb$-category (i.e., its straightening $\Bb\catop \to \Cat_{\infty}$ preserves limits).
\end{corollary}

From this description of $\Cc$-modules in $\Cat(\Bb)$ in terms of cocartesian fibrations, we can deduce a non-parametrized criterion for a $\Cc$-linear $\Bb$-functor $F\colon \Dd \to \Ee$ to have a $\Cc$-linear right adjoint.\footnote{In this appendix, contrary to the convention used in the body of the text, $\Cc$-linear $\Bb$-functors are not assumed to preserve colimits.}

\begin{proposition}
	\label{prop:RightAdjointCLinear}
	Let $\Cc$ be a symmetric monoidal $\Bb$-category, let $\Dd$ and $\Ee$ be $\Bb$-categories tensored over $\Cc$, and let $F\colon \Dd \to \Ee$ be a $\Cc$-linear functor. Assume that $F$ admits a parametrized right adjoint $G\colon \Ee \to \Dd$ satisfying the following projection formula: for every $B \in \Bb$, every $C \in \Cc(B)$ and every $E \in \Ee(B)$, the map $C \otimes_B G(E) \to G(C \otimes_B E)$ adjoint to the composite
	\begin{align*}
		F(C \otimes_B G(E)) \simeq C \otimes_B F(G(E)) \xrightarrow{C \otimes_B \textup{unit}} C \otimes_B E
	\end{align*}
	is an equivalence in $\Dd(B)$. Then the right adjoint $G$ admits canonical $\Cc$-linear structure and the adjunction enhances to an adjunction in $\Mod_{\Cc}(\Cat(\Bb))$.
\end{proposition}
\begin{proof}
	By \Cref{cor:un-straightening-left-modules}, we may identify the $\Cc$-module structures on $\Dd$ and $\Ee$ with cocartesian fibrations over $\Bb^{\op,\Ll\Mm}$ whose restriction to $\Bb^{\op,\sqcup}$ is $\Cc^{\boxtimes}$. The map $F$ thus corresponds to a map
	\begin{equation*}
		\begin{tikzcd}
			\Dd^{\boxtimes} \ar{rr}{F^{\boxtimes}} \drar & & \Ee^{\boxtimes} \dlar \\
			&\Bb^{\op,\Ll\Mm}
		\end{tikzcd}
	\end{equation*}
	of cocartesian fibrations. If $F\colon \Dd \to \Ee$ has a parametrized right adjoint, then in particular $F^{\boxtimes}$ has fiberwise right adjoints, and by \cite[Proposition~7.3.2.1]{lurie2016HA} these assemble into a relative right adjoint $G^{\boxtimes}\colon \Ee^{\boxtimes} \to \Dd^{\boxtimes}$ over $\Bb^{\op,\Ll\Mm}$ which moreover is a map of $\infty$-operads. In particular $G$ preserves inert maps. Our goal is to show that $G$ in fact preserves all cocartesian edges. 
	By the product description of the mapping spaces in an $\infty$-operad, we may restrict attention to cocartesian morphism in $\Ee^{\boxtimes}$ whose target lies over $\langle 1 \rangle \in \Fin_*$. Given such morphism, let $(B_1, \dots, B_n) \to B$ be the image in $\Bb^{\op,\Ll\Mm}$.\footnote{This notation is abusive: the $\Ll\Mm^{\otimes}$-component is hidden in the notation.} This map is a composite of two maps $(B_1, \dots, B_n) \to B_1 \times \dots \times B_n \to B$, so we may assume without loss of generality that either $n = 1$ or that the map is cocartesian for the cocartesian fibration $\Bb^{\op,\Ll\Mm} \to \Fin_*$. When $n = 1$, this is just the condition that $G\colon \Ee \to \Dd$ is a parametrized functor, by assumption on $F$. So assume the map $(B_1, \dots, B_n) \to B$ is cocartesian for $\Bb^{\op,\Ll\Mm} \to \Fin_*$. When $B$ lies over $\mathfrak{a} \in \Ll\Mm$, the claim is clear since $F$ is the identity over $\Aa\textup{ssoc} \subseteq \Ll\Mm$. So we may assume without loss of generality that $B$ lies over $\mathfrak{m} \in \Ll\Mm$, and moreover that $B_n$ lies over $\mathfrak{m}$ while all the other $B_i$ lie over $\mathfrak{a}$. By induction, we may assume that $n = 2$, so that the map in $\Bb^{\op,\Ll\Mm}$ takes the form $(B_1,B_1) \to B_1 \times B_2$ and lies over the map $(\mathfrak{a},\mathfrak{m}) \to \mathfrak{m}$ in $\Ll\Mm$. A cocartesian edge in $\Ee^{\boxtimes}$ over this map has the form $(C,E) \to C \boxtimes E$. The condition that the induced map $G(C,E) = (C,G(E)) \to G(C \boxtimes E)$ is again cocartesian is equivalent to the condition that the map $C \boxtimes G(E) \to G(C \boxtimes E)$ is an equivalence in $\Dd( A \times B)$. Since for $C \in \Cc(A)$, $E \in \Ee(B)$ and $E' \in \Ee(A)$ there are equivalences
	\begin{align*}
		\begin{split}
			C \boxtimes E \simeq \pi_A^*C \otimes_{A \times B} \pi_B^*E \\
			C \otimes_A E' \simeq {\Delta_A}^*(C \boxtimes E'),
		\end{split}
	\end{align*}
	and since $G$ commutes with base change, it follows that this condition is equivalent to the assumption on $G$.
\end{proof}

\section{Duality in equivariant stable homotopy theory}
\label{sec:TheoremCampion}

Let $G$ be a compact Lie group. In his PhD dissertation, Campion \cite{campion2023FreeDuals} proved a universal property of the $\infty$-category $\Sp^G$ of genuine $G$-spectra: the suspension spectrum functor $\Sigma^{\infty}\colon \Spc^G_* \to \Sp^G$ in $\CAlg(\PrL)$ is initial among symmetric monoidal left adjoints $F\colon \Spc_*^G \to \Dd$ into a stable presentably symmetric monoidal $\infty$-category $\Dd$ such that $F(X)$ is dualizable for every compact pointed $G$-space $X$. Since we know from \cite[Corollary~C.7]{gepnerMeier2020equivariant} that $\Sigma^{\infty}$ is initial among functors $F$ that \textit{invert representation spheres}, the crux of Campion's result is that these two conditions on $F\colon \Spc_*^G \to \Dd$ are in fact equivalent:

\begin{theorem}[{cf.\ Campion \cite[Section~5]{campion2023FreeDuals}}]
	\label{thm:CampionInvertingSpheresVsDualizingSpheres}
	Let $G$ be a compact Lie group, let $\Dd \in \CAlg(\PrL_{\st})$ be a stable presentably symmetric monoidal $\infty$-category, and let $F\colon \Spc^G_* \to \Dd$ be a symmetric monoidal left adjoint. Then the following are equivalent:
	\begin{enumerate}[(1)]
		\item For every orthogonal $G$-representation $V$, the functor $F$ sends the representation sphere $S^V$ to an invertible object of $\Dd$;
		\item For every compact pointed $G$-space $X \in \Spc^G_*$, the object $F(X)$ is dualizable in $\Dd$;
		\item For every orthogonal $G$-representation $V$, the functor $F$ sends the representation sphere $S^V$ to a dualizable object of $\Dd$.
	\end{enumerate}
\end{theorem}

For completeness, we will include a proof of this theorem. While we present both the statement and the proof slightly differently, the core ideas are taken from Campion's thesis \cite{campion2023FreeDuals}. The first core idea is that certain dualizable objects are already close to being invertible:

\begin{definition}[{Campion \cite[Definition~2.1.2]{campion2023FreeDuals}}]
	Let $\Dd$ be a presentably symmetric monoidal $\infty$-category, let $T \in \Dd$ be an object and let $t\colon T \to T$ be an endomorphism of $T$. We say that $T$ has \textit{$t$-twisted trivial braiding} if the twist morphism $\sigma_{T,T}\colon T \otimes T \to T \otimes T$ is homotopic to the map $1 \otimes t\colon T \otimes T \to T \otimes T$.
\end{definition}

\begin{proposition}[{cf.\ Campion \cite[Proposition~2.3.1]{campion2023FreeDuals}}]
	\label{prop:Campiontwistedtrivialbraidingimpliesselfstable}
	Let $T \in \Dd$ be a dualizable object. Assume that $T$ has $t$-twisted trivial braiding for some endomorphism $t\colon T \to T$. Then the morphism
	\begin{align*}
		T \otimes T^{\vee} \otimes T \xrightarrow{1 \otimes \ev} T
	\end{align*}
	is an equivalence.
\end{proposition}
\begin{proof}
	Since $1 \otimes \ev_T$ always admits a right-inverse given by $\coev_T \otimes 1\colon T \to T \otimes T^{\vee} \otimes T$, it remains to show that $1 \otimes \ev_T$ also admits a left-inverse. We claim that in fact already the evaluation map $\ev_T$ itself admits a left-inverse, given by the following composite:
	\begin{align*}
		\unit_{\Dd} \xrightarrow{\coev} T \otimes T^{\vee} \xrightarrow{t \otimes T^{\vee}} T \otimes T^{\vee} \xrightarrow{\sigma_{T,T^{\vee}}} T^{\vee} \otimes T.
	\end{align*}
	Indeed, this follows from the following string diagram:
	
	\usetikzlibrary[decorations.markings]
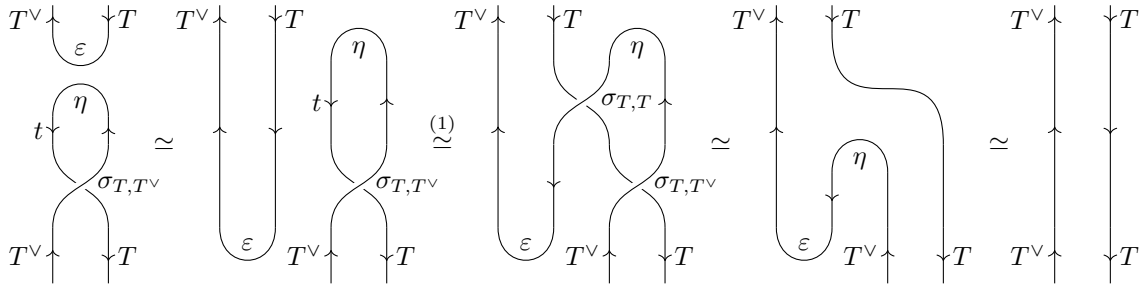
	\[
	\hskip-2.1pt\hfuzz=2.1pt
	\begin{tikzpicture}[arr/.style={postaction={decorate},decoration={markings,mark=at position #1 with {\arrow{>}}}}, scale=0.74]
		\draw[arr=0.6] (1,5) -- (1,4.5) node[pos=0.5,right]{$T$};
		\draw[arr=0.6] (0,4.5) -- (0,5) node[pos=0.5,left]{$T^{\vee}$};
		\draw (1,4.5) to[out=270,in=270,looseness=2] (0,4.5);
		\draw (0,3) to[out=90,in=90,looseness=2] (1,3);
		\draw[arr=0.6] (1,2.5) -- (1,3);
		\draw[arr=0.6] (0,3) -- (0,2.5) node[pos=0.5,left]{$t$};
		\draw (0,2.5) to[out=270,in=90,looseness=1.2] (1,1);
		\draw [white,line width=4pt] (1,2.5) to[out=270,in=90,looseness=1.2] (0,1);
		\draw (1,2.5) to[out=270,in=90,looseness=1.2] (0,1);
		\draw[arr=0.6] (0,0) -- (0,1) node[pos=0.5,left]{$T^{\vee}$};
		\draw[arr=0.6] (1,1) -- (1,0) node[pos=0.5,right]{$T$};
		\path
		(0.5,4.2) node{$\epsilon$}
		(0.5,3.2) node{$\eta$}
		(1.4,1.75) node{$\sigma_{T,T^{\vee}}$}
		(2,2.5) node{$\simeq$};
		\draw[arr=0.6] (4,5) -- (4,4.5) node[pos=0.5,right]{$T$};
		\draw[arr=0.6] (3,4.5) -- (3,5) node[pos=0.5,left]{$T^{\vee}$};
		\draw[arr=0.52] (4,4.5) -- (4,1);
		\draw[arr=0.52] (3,1) -- (3,4.5);
		\draw (4,1) to[out=270,in=270,looseness=2] (3,1);
		\draw (5,4) to[out=90,in=90,looseness=2] (6,4);
		\draw[arr=0.55] (6,2.5) -- (6,4);
		\draw[arr=0.55] (5,4) -- (5,2.5) node[pos=0.5,left]{$t$};
		\draw (5,2.5) to[out=270,in=90,looseness=1.2] (6,1);
		\draw [white,line width=4pt] (6,2.5) to[out=270,in=90,looseness=1.2] (5,1);
		\draw (6,2.5) to[out=270,in=90,looseness=1.2] (5,1);
		\draw[arr=0.6] (5,0) -- (5,1) node[pos=0.5,left]{$T^{\vee}$};
		\draw[arr=0.6] (6,1) -- (6,0) node[pos=0.5,right]{$T$};
		\path
		(3.5,0.7) node{$\epsilon$}
		(5.5,4.2) node{$\eta$}
		(6.4,1.75) node{$\sigma_{T,T^{\vee}}$}
		(7,2.75) node{$\overset{(1)}{\simeq}$};
		\draw[arr=0.6] (9,5) -- (9,4.5) node[pos=0.5,right]{$T$};
		\draw[arr=0.6] (8,4.5) -- (8,5) node[pos=0.5,left]{$T^{\vee}$};
		\draw[arr=0.52] (9,2.5) -- (9,1);
		\draw (9,4.5) -- (9,4);
		\draw[arr=0.52] (8,1) -- (8,4.5);
		\draw (9,1) to[out=270,in=270,looseness=2] (8,1);
		\draw (10,4) to[out=90,in=90,looseness=2] (11,4);
		\draw[arr=0.55] (11,2.5) -- (11,4);
		\draw (9,4) to[out=270,in=90,looseness=1.2] (10,2.5);
		\draw [white,line width=4pt] (10,4) to[out=270,in=90,looseness=1.2] (9,2.5);
		\draw (10,4) to[out=270,in=90,looseness=1.2] (9,2.5);
		\draw (10,2.5) to[out=270,in=90,looseness=1.2] (11,1);
		\draw [white,line width=4pt] (11,2.5) to[out=270,in=90,looseness=1.2] (10,1);
		\draw (11,2.5) to[out=270,in=90,looseness=1.2] (10,1);
		\draw[arr=0.6] (10,0) -- (10,1) node[pos=0.5,left]{$T^{\vee}$};
		\draw[arr=0.6] (11,1) -- (11,0) node[pos=0.5,right]{$T$};
		\path
		(8.5,0.7) node{$\epsilon$}
		(10.5,4.2) node{$\eta$}
		(11.4,1.75) node{$\sigma_{T,T^{\vee}}$}
		(10.32,3.25) node{$\sigma_{T,T}$}
		(12,2.5) node{$\simeq$};
		\draw[arr=0.6] (14,5) -- (14,4.5) node[pos=0.5,right]{$T$};
		\draw[arr=0.6] (13,4.5) -- (13,5) node[pos=0.5,left]{$T^{\vee}$};
		\draw[arr=0.52] (14,2) -- (14,1);
		\draw[arr=0.52] (13,1) -- (13,4.5);
		\draw (14,1) to[out=270,in=270,looseness=2] (13,1);
		\draw (14,4.5) to[out=270,in=90,looseness=1.8] (16,2.5);
		\draw (16,2.5) -- (16,1);
		\draw (14,2) to[out=90,in=90,looseness=2] (15,2);
		\draw (15,1) -- (15,2);
		\draw[arr=0.6] (15,0) -- (15,1) node[pos=0.5,left]{$T^{\vee}$};
		\draw[arr=0.6] (16,1) -- (16,0) node[pos=0.5,right]{$T$};
		\path
		(13.5,0.7) node{$\epsilon$}
		(14.5,2.2) node{$\eta$}
		(17,2.5) node{$\simeq$};
		\draw[arr=0.6] (19,5) -- (19,4.5) node[pos=0.5,right]{$T$};
		\draw[arr=0.6] (18,4.5) -- (18,5) node[pos=0.5,left]{$T^{\vee}$};
		\draw[arr=0.52] (19,4.5) -- (19,1);
		\draw[arr=0.52] (18,1) -- (18,4.5);
		\draw[arr=0.6] (18,0) -- (18,1) node[pos=0.5,left]{$T^{\vee}$};
		\draw[arr=0.6] (19,1) -- (19,0) node[pos=0.5,right]{$T$};
	\end{tikzpicture}
	\]
	
	Here we have abbreviated $\epsilon = \ev_T$ and $\eta = \coev_T$. Each of the five diagrams represents a composite of morphisms in $\Dd$, where each downward pointing arrow denotes a tensor factor of $T$ while an upward pointing arrow denotes a tensor factor of $T^{\vee}$. For example, the third diagram is to be read as the following composite:
	\[
	T^{\vee} \otimes T \xrightarrow{1 \otimes \eta} T^{\vee} \otimes T \otimes T \otimes T^{\vee} \xrightarrow{1 \otimes \sigma_{T,T} \otimes 1} T^{\vee} \otimes T \otimes T \otimes T^{\vee} \xrightarrow{1 \otimes \sigma_{T,T^{\vee}}} T^{\vee} \otimes T \otimes T^{\vee} \otimes T \xrightarrow{\epsilon \otimes 1} T^{\vee} \otimes T.
	\]
	The equivalence labelled (1) holds because of the assumption that $T$ has a $t$-twisted trivial braiding. The other equivalences are easy manipulations of string diagrams. This finishes the proof. 
\end{proof}

\begin{lemma}[{Campion \cite[Example~2.1.6]{campion2023FreeDuals}}]
	\label{lem:Campionrepresentationspheretwistedtrivialbraiding}
	For every $G$-representation $V$, the representation sphere $S^V \in \Spc^G_*$ has $(-1)$-twisted trivial braiding, where $-1\colon S^V \to S^V\colon v \mapsto -v$.
\end{lemma}
\begin{proof}
	Consider the family of linear $G$-equivariant maps $H\colon [0,1] \to \End_G(V \times V)$ sending $t$ to the matrix
	\begin{align*}
		\begin{pmatrix}
			-\sin(\frac{\pi}{2} t) & \cos(\frac{\pi}{2} t) \\ \cos(\frac{\pi}{2} t) & \sin(\frac{\pi}{2} t)
		\end{pmatrix}
		\colon V \times V \to V \times V.
	\end{align*}
	We have $H_0(v_1,v_2) = (v_2,v_1)$ and $H_1(v_1,v_2) = (-v_1,v_2)$. Since $H_t$ is invertible at all times, it induces a pointed map $S^{H_t}\colon S^{V\times V} \to S^{V\times V}$. As $S^{V \times V} \cong S^V \wedge S^V$, it follows that the twist map of $S^V$ in $\Spc^G_*$ is homotopic to the map $(-1) \wedge \id$ as desired.
\end{proof}

To continue, we need the notion of $G$-spaces concentrated at a single conjugacy class. We fix a compact Lie group $G$ and a conjugacy class $(H)$ of subgroups of $G$. 

\begin{definition}
	We say that a pointed $G$-space $X \in \Spc^G_*$ is \textit{concentrated at $H$} if we have $X^K = *$ for any subgroup $K \subseteq G$ that is not in the conjugacy class of $H$.
\end{definition}

Let $W = W_G(H) = N_G(H)/H$ denote the Weyl group of $H$ in $G$. Observe that the full subcategory of the orbit category $\Orb_G$ spanned by $G/H$ is equivalent to the classifying space $BW$ of $W$, as $W$ is equivalent to the endomorphism space of $G/H$ in $\Orb_G$. It follows that the $H$-fixed point functor $(-)^H\colon \Spc^G_* = \Fun(\Orb_G,\Spc_*) \to \Fun(BW,\Spc_*)$ becomes an equivalence when restricted to the pointed $G$-spaces concentrated at $H$. 

We let $S^{\{H\}}$ denote a copy of $S^0$ concentrated at $H$ corresponding:
\begin{align*}
	(S^{\{H\}})^K = \begin{cases} S^0 & K \text{ conjugate to } H; \\ * & \text{otherwise.}\end{cases}
\end{align*}
Note that the functor $- \wedge S^{\{H\}}\colon \Spc^G_* \to \Spc^G_*$ is a localization onto the pointed $G$-spaces concentrated at $H$.

\begin{lemma}
	\label{lem:Campionuntwistinglemma}
	Let $V$ be a $G$-representation and let $n := \dim(V^H)$. Then there is an equivalence of pointed $G$-spaces $S^V \wedge S^{\{H\}} \wedge G/H_+ \simeq \Sigma^n(S^{\{H\}} \wedge G/H_+)$.
\end{lemma}
\begin{proof}
	Since both sides are concentrated at $H$, the equivalence can be checked after forgetting to $\Fun(BW,\Spc_*)$ after passing to $H$-fixed points. As the $H$-fixed points of $G/H_+$ are $W_+$, this follows from the sequence of equivalences
	\begin{align*}
		(S^V)^H \wedge W_+ \simeq S^{V^H} \wedge W_+ \simeq S^n \wedge W_+ = \Sigma^n(W_+).
	\end{align*}
	Here the second equivalence uses the shear isomorphism for $W$-spaces, using that the underlying pointed space of $S^{V^H}$ is $S^n$.
\end{proof}

We are now ready for the proof of \Cref{thm:CampionInvertingSpheresVsDualizingSpheres}.
\begin{proof}[Proof of \Cref{thm:CampionInvertingSpheresVsDualizingSpheres}.]
	Let $F\colon \Spc^G_* \to \Dd$ be a symmetric monoidal left adjoint as in the statement of the theorem. We prove that 1), 2) and 3) are equivalent. The implication (2) $\implies$ (3) is obvious. For the implication (1) $\implies$ (2), note that (1) implies by the universal property of $\Sp^G = (\Spc^G_*)[\{S^V\}^{-1}]$ that $F$ uniquely extends to a symmetric monoidal left adjoint $\widetilde{F}\colon \Sp^G \to \Dd$. Since the suspension functor $\Sigma^{\infty}_+\colon \Spc^G_* \to \Sp^G$ sends compact $G$-spaces to dualizable objects, the same follows for $F$.
	
	It remains to show that 3) implies 1). Let $V$ be a $G$-representation and assume that $F(S^V)$ is dualizable. To prove that $F(S^V)$ is invertible, we need to show that the evaluation map $\ev_{F(S^V)}\colon F(S^V) \otimes D(F(S^V)) \to \unit_\Dd$ is an equivalence, or equivalently that the cofiber
	\[
	Q := \cofib(F(S^V) \otimes D(F(S^V)) \xrightarrow{\ev_{F(S^V)}} \unit_\Dd)
	\]
	is zero in $\Dd$. We will inductively show that $Q \otimes F(X)$ is zero for a large family of $G$-spaces $X$.
	
	\textit{Step 1:} We have seen in \Cref{lem:Campionrepresentationspheretwistedtrivialbraiding} that $X = S^V$ has twisted trivial braiding, and since the functor $F$ is symmetric monoidal it follows that $F(S^V)$ also has twisted-trivial braiding. \Cref{prop:Campiontwistedtrivialbraidingimpliesselfstable} then gives us that the map $\ev_{F(S^V)} \colon F(S^V) \otimes D(F(S^V)) \to \unit_{\Dd}$ becomes an equivalence after tensoring with $F(S^V)$. In particular, $Q \otimes F(S^V) = 0$.
	
	\textit{Step 2:} We will show that $Q \otimes F(S^{\{H\}}) = 0$ for every subgroup $H$ of $G$. By step 1, we know that $Q \otimes F(S^V \wedge S^{\{H\}} \wedge G/H_+) \simeq Q \otimes F(S^V) \otimes F(S^{\{H\}} \wedge G/H_+) = 0$. From the equivalence of \Cref{lem:Campionuntwistinglemma} and exactness of $Q \otimes F(-)$, it follows that 
	\[
	0 = Q \otimes F(S^V \wedge S^{\{H\}} \wedge G/H_+) \simeq \Sigma^n(Q \otimes F(S^{\{H\}} \wedge G/H_+)),
	\]
	so that also $Q \otimes F(S^{\{H\}} \wedge G/H_+) = 0$ by stability of $\Dd$. Since $Q \otimes F(-)$ commutes with homotopy orbits, we get
	\[
	Q \otimes F(S^{\{H\}}) \simeq Q \otimes F\left((S^{\{H\}} \wedge G/H_+)_{hW}\right) \simeq \left(Q \otimes F(S^{\{H\}} \wedge G/H_+)\right)_{hW} = 0.
	\]
	
	\textit{Step 3:} We show that $Q = 0$. For a family $\Ff$ of subgroups of $G$ we let $E_{\Ff}$ denote the universal $G$-space with $\Ff$-isotropy, characterized by the property that
	\begin{align*}
		(E_\Ff)^K = \begin{cases} * & K \in \Ff; \\ \emptyset & \text{otherwise.}\end{cases}
	\end{align*}
	We let $E_{\Ff,+} := (E_\Ff)_+$. Consider $\Uu$ be the poset of all families $\Ff$ of subgroups of $G$ for which $Q \otimes F(E_{\Ff,+}) = 0$. Our goal is to show that the family of \textit{all} subgroups is contained in $\Uu$, as in that case we have $E_{\Ff,+} = S^0$ so that $Q \otimes F(E_{\Ff,+}) = Q \otimes F(S^0) \simeq Q \otimes \unit_{\Dd} \simeq Q$.
	
	Note that $\Uu$ contains $\Ff = \emptyset$ since $E_{\Ff} = \emptyset$ and $Q \otimes F(-)$ preserves the initial object. Furthermore, for a chain $\Ff_1 \subseteq \Ff_2 \subseteq \Ff_3 \subseteq \dots$ of nested families in $\Uu$ with union $\Ff$, we have $E_{\Ff} = \colim_i E_{\Ff_i}$, as is easily checked on $K$-fixed points for every $K \leqslant G$. As $Q \otimes F(-)$ preserves colimits, it follows that also $\Ff \in \Uu$. Thus $\Uu$ satisfies the conditions of Zorn's lemma and thus contains a maximal element $\Ff_m$.
	
	We claim that $\Ff_m$ must contain all subgroups of $G$. Assume that a subgroup $H$ is not contained in $\Ff_m$. Since the poset of subgroups of $G$ is well-founded, we may assume that $H$ is minimal with this property, i.e.\ every strict subgroup of $H$ is in $\Ff_m$. We now define $\Ff'$ as the family of subgroups either contained in $\Ff_m$ or conjugate to $H$. This is indeed a family by assumption on $H$. We then have a cofiber sequence
	\begin{align*}
		E_{\Ff_m,+} \to E_{\Ff',+} \to S^{\{H\}}.
	\end{align*}
	in $\Spc^G_*$. Since $Q \otimes F(-)$ preserves cofiber sequences, the sequence
	\begin{align*}
		Q \otimes F(E_{\Ff_m,+}) \to Q \otimes F(E_{\Ff',+}) \to Q \otimes F(S^{\{H\}})
	\end{align*}
	is again a cofiber sequence in $\Dd$. Since $\Ff_m \in \Uu$, the left term is zero, and by step 2) also the right term is zero. It follows that the middle term is zero and thus $\Ff' \in \Uu$. This contradicts the maximality of $\Ff_m$. We conclude that $H$ must have already be contained in $\Ff_m$ and thus that $\Ff_m$ contains all subgroups of $G$. This finishes the proof of the implication (3) $\implies$ (1), thus finishing the proof of the theorem.
\end{proof}

\bibliographystyle{alpha}
\bibliography{Bibliography.bib}
\addcontentsline{toc}{section}{References}

\end{document}